\pdfoutput=1
\documentclass[10pt]{article}
\usepackage{times}
\usepackage{mathptmx}

\usepackage{lipsum}
\usepackage{amsfonts}
\usepackage{graphicx}
\usepackage{epstopdf}
\usepackage{algorithmic}
\usepackage{relsize}
\usepackage{tikz}

\usepackage{tabularx}
\usepackage{booktabs}

\usetikzlibrary{positioning}
\ifpdf
  \DeclareGraphicsExtensions{.eps,.pdf,.png,.jpg}
\else
  \DeclareGraphicsExtensions{.eps}
\fi

\usepackage{enumitem}
\usepackage{tikz}
\usepackage{tkz-euclide}
\setlist[itemize]{leftmargin=.5in}

\usepackage{amsmath,amsthm}
\usepackage{amssymb}
\usepackage[hypertexnames=false]{hyperref}
\usepackage{algorithm}
\usepackage{geometry}
\usepackage{esint}
\usepackage{array}
\usepackage[linewidth=2pt]{mdframed}
\newcounter{taggedeq}
\pretocmd{\equation}{\stepcounter{taggedeq}}{}{}

\newcommand\sbullet[1][.5]{\mathbin{\vcenter{\hbox{\scalebox{#1}{$\bullet$}}}}}
\providecommand{\keywords}[1]{\textbf{\textit{Keywords: }} #1}

\newmdenv[
topline=false,
bottomline=false,
rightline=false,
skipabove=\topsep,
skipbelow=\topsep,
linewidth=4
]{siderules}
\usepackage{xcolor}
\usepackage[most]{tcolorbox}
\usepackage[nottoc]{tocbibind}
\newcommand{\hookdoubleheadrightarrow}{\hookrightarrow\mathrel{\mspace{-15mu}}\rightarrow}

\newcommand{\supp}{\text{supp}}
\newcommand{\diam}{\text{diam}}
\newcommand{\BC}{\text{BC}}

\newcommand{\loc}{\text{loc}}
\newcommand{\id}{\text{id}}

\newtheorem{theorem}{Theorem}
\newtheorem{proposition}[theorem]{Proposition}
\newtheorem{remark}[theorem]{Remark}
\newtheorem{definition}[theorem]{Definition}
\newtheorem{lemma}[theorem]{Lemma}

\newtheorem{example}[theorem]{Example}
\newtheorem{assumption}[theorem]{Assumption}
\newtheorem*{proposition*}{Proposition}

\newcommand*\dx{\mathop{}\!\mathrm{d}}

\DeclareMathOperator*{\esssup}{ess\,sup}

\title{On uniqueness in structured model learning}

\author{Erion Morina\thanks{IDea\_Lab - The Interdisciplinary Digital Lab at the University of Graz, University of Graz, Austria.
		\{\href{mailto:erion.morina@uni-graz.at}{erion.morina@uni-graz.at}, \href{mailto:martin.holler@uni-graz.at}{martin.holler@uni-graz.at}\}}
	\textsuperscript{,}\thanks{Corresponding author.}
	\and Martin Holler\footnotemark[1]}

\usepackage{amsopn}

\usepackage{setspace}
\usepackage{placeins}

\usepackage{etoolbox}
\usepackage[numbers]{natbib}

\makeatletter
\g@addto@macro\appendix{}
\makeatother
\usepackage{bookmark}
\hypersetup{hidelinks,pdftitle={On uniqueness in structured model learning},pdfauthor={Erion Morina, Martin Holler}}

\newcommand{\R}{\mathcal{R}}
\newcommand{\alttext}[1]{}

\numberwithin{equation}{section}
\begin{document}

\maketitle
\begin{abstract}
This paper addresses the problem of uniqueness in learning physical laws for systems of partial differential equations (PDEs). Contrary to most existing approaches, it considers a framework of \emph{structured model learning}, where existing, approximately correct physical models are augmented with components that are learned from data. The main results of the paper are a uniqueness and a consistency result that cover a large class of PDEs and a suitable class of neural networks used for approximating the unknown model components. The uniqueness result shows that, in the limit of full, noiseless measurements, a unique identification of the unknown model components \emph{as functions} is possible as classical regularization-minimizing solutions of the PDE system. These reproduce the data-generating dynamics along the states visited by the data, while their behavior elsewhere is determined by the choice of regularization. This result is complemented by a consistency result showing that model components learned as parameterized neural networks from incomplete, noisy measurements approximate the regularization-minimizing solutions of the PDE system in the limit. These results are possible under specific properties of the approximating neural networks and due to a dedicated choice of regularization. With this, a practical contribution of this analytic paper is to provide a class of model learning frameworks different to standard settings where uniqueness can be expected in the limit of full measurements.
\end{abstract}

\section{Introduction} Learning nonlinear differential-equation-based models from data is a highly active field of research. Its general goal is to gain information on a (partially) unknown differential-equation-based physical model from measurements of its state. Information on the model here means to either directly learn a parametrized version of the model or to learn a corresponding parametrized solution map. In both cases, neural networks are used as parametrized approximation classes in most of the existing recent works.
Important examples, reviewed in \cite{boulle23}, are physics-informed neural operators \cite{kovachki24}, DeepONets \cite{Lu19},
Fourier Neural Operators \cite{Azizzadenesheli20}, Graph Neural Networks \cite{Zongyi20}, Wavelet Neural Operators \cite{Tripura22}, DeepGreen \cite{Shea21} and model reduction \cite{bhattacharya21}, among others.
The comprehensive reviews \cite{Azizzadenesheli2024,blechschmidt21,kutz23, deryck24,Kovachki2024,tanyu23} and the references therein, provide an overview of the state of the art.

\noindent \paragraph{Scope.} The above works all focus on \emph{full model learning}, i.e., learning the entire dif\-fer\-en\-tial-equation-based model from data. In contrast to this, the approach considered here is focused on \emph{structured model learning}, where we assume that an approximately correct physical model is available, and only extensions of the model (corresponding to fine-scale hidden physics not present in the approximate model) are learned from data. Specifically, we are concerned with the problem of identifying an unknown nonlinear term $f$ together with physical parameters $\varphi$ of a system of partial differential equations (PDEs)
\begin{align}
	\label{introsystem} \partial_t u = F(t, u, \varphi)+f(t,u), \qquad (t,x)\in (0,T)\times \Omega,
\end{align}
from indirect, noisy measurements of the state $u$. Here, $T>0$, $\Omega$ is a domain, $F$ is the known physical model and all involved quantities can potentially be vector-valued such that systems of PDEs are covered.  Also note that the terms $F$ and $f$ can act on values and higher-order derivatives of the state. Given this, even though we focus on non-trivial physical models $F$, our work covers also the setting of full model learning by setting $F(t,u,\varphi)=0$. The physical term $F$ models knowledge that is available, e.g. from first principles, while $f$ accounts for unknown effects, e.g. because of a coarse modeling process due to incomplete knowledge of the underlying dynamics. A schematic illustration of structured model learning is provided in Figure \ref{fig:struc_model_learning}.

\begin{figure}[htbp]
	\centering
	\includegraphics[width=0.85\linewidth]{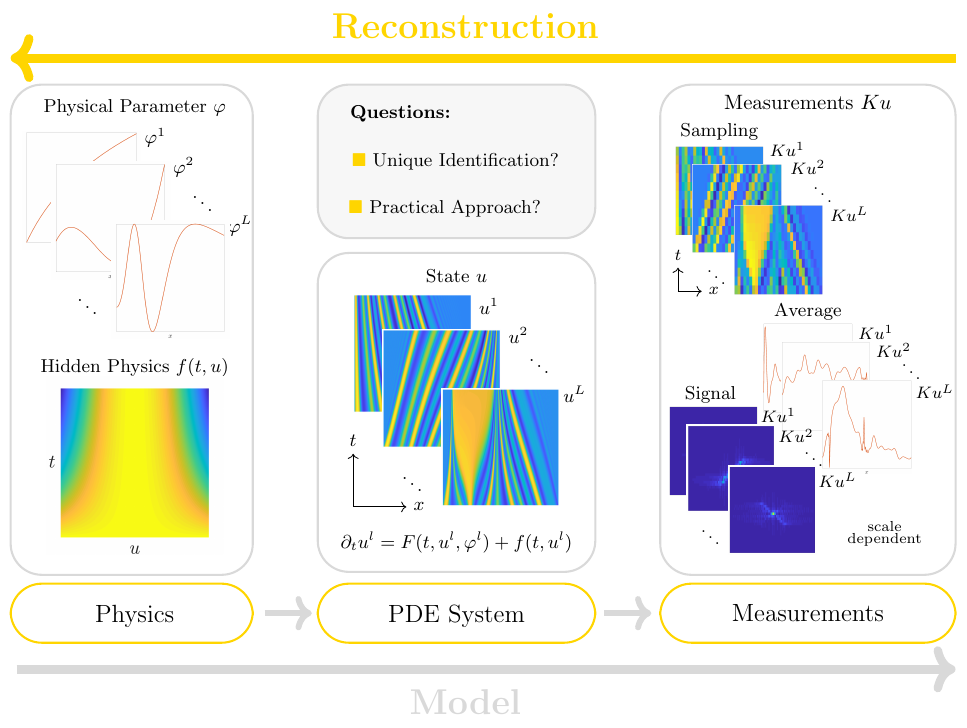}
	\caption{Schematic illustration of \textit{structured model learning}. It is understood as the identification of a hidden physics component $f$ and physical parameters $\varphi$ satisfying system \eqref{introsystem} from indirect measurements $Ku$ of $u$ for a given physical model $F$.}
	\alttext{Three-panel flow diagram, physics, PDE system and measurements, and their relation in the forward (modelling) and the inverse (reconstruction) problem.}
	\label{fig:struc_model_learning}
\end{figure}

One might think of absorbing the known model into the unknown one and learning $F+f$ at once. The advantage of the structured setting is that it allows for formulations where a structurally different model, has to be learned, while the available structure can be retained. Indeed, the unknown term may depend on fewer or more variables, differ in smoothness and magnitude compared to the entire right-hand side of \eqref{introsystem}, and all requirements which our analysis imposes on the learned term and on the class of functions used to approximate it concern this part of the model only. If, for instance, the dependence of the model on the highest-order derivatives of the state is carried by $F$, only a function of lower order derivatives needs to be learned, which in addition allows one to control structural properties of the learned equation alone. Note that the structured model learning formulation assumes the available model $F$ to be correct, errors in $F$ being compensated by $f$ and $\varphi$.

The main question considered in this work is to what extent measurements $Ku^l$ of system states $u^l$ corresponding to (unknown) parameters $\varphi^l$, $l=1,\ldots,L$, allow to uniquely identify the nonlinearity $f$. Already in the simple setting that $f$ acts pointwise, i.e., $f(\cdot,u)(t,x) = f(u(t,x))$, it is clear that, without further specification, this question only has a trivial answer: Even if $(u^l,\varphi^l)_l$ is known entirely, $f$ is only determined on $\bigcup _{l=1}^L \{ u^l(t,x) \,|\, (t,x) \in (0,T) \times \Omega \}$. A natural way to overcome this, as done in \cite{kutyniok23} (and \cite{scholl_icassp}) for full model learning, is to consider particular types of functions $f$: Specifying to the case $F(t,u,\varphi)=0$, a result of \cite{kutyniok23} is that a linear or algebraic function $f$ is uniquely identifiable from full state measurements if and only if the state variables (and their derivatives in case $f$ acts also on derivatives) are linearly or algebraically independent, respectively. Similarly, \cite{kutyniok23} shows that a smooth $f$ is uniquely reconstructable from full state measurements if the values of the state variables (and their derivatives) are dense in the underlying Euclidean vector space. Consistent with this result, \cite{shumaylov2025} shows that equation discovery typically requires chaotic behavior. While these results provide answers in rather general settings, the conditions on $u$ that guarantee unique recovery are difficult to verify exactly in practice (in \cite{kutyniok23} unique identifiability is classified via SVD-based thresholding, with extension to the noisy regime in \cite{Kut24}).

A different possibility to address the uniqueness problem would be to consider a specific parametrized class of functions $\{f_\theta \,|\, \theta \in \Theta \}$ for approximating $f$, and to investigate uniqueness of the parameters. In case of simple approximation classes such as polynomials, this indeed provides a simple solution (e.g., parameters of a $n$-degree polynomial are uniquely determined by $n+1$ different values of the state). In case of more complex approximation classes such as neural networks however, this even introduces an additional difficulty, namely that different sets of parameters might represent the same function.

The approach we take in this work to address the uniqueness problem in model learning follows classical inverse-problems techniques for unique parameter identification via regularization-minimizing solutions. Specifically, covering also the setting of non-trivial physical $F$, additional, unknown parameters $(\varphi^l)_l$ and non-trivial forward models, we consider uniqueness of the function $f$ (and the corresponding parameters $\varphi=(\varphi^l)_l$ and states $u=(u^l)_l$) as solutions to the full measurement/vanishing noise limit problem
\begin{equation}
	\label{demo_dagger} \tag{$p^\dagger$} \min_{\varphi,u,f}\mathcal{R}^\dagger(\varphi,u,f)\qquad \text{s.t. }\forall l: \quad \partial_t u^l = F(t, u^l, \varphi^l)+f(t,u^l), \quad K^\dagger u^l = \hat y^l
\end{equation}
where $K^\dagger$ is the injective full measurement operator and $y=(\hat y^l)_l$ is the corresponding full-measurement data. With this, we allow $\R^\dagger$ to incorporate prior information on $f$ that can be used to resolve unique identifiability. In view of the above-described alternative works on uniqueness in model learning, this is related to considering $f$ to be out of a class of functions for which the measurements are sufficient for unique identifiability, only that we do not explicitly rely on such a setting, but rather provide a general framework that will always approximate a desired solution $f^\dagger$ as unique regularization-minimizing solution consistent with the measurement data. Of course, if the measurement data is sufficient to uniquely identify a ground-truth $f$ without the use of any additional prior information, our setting will recover this ground-truth. Concretely, the state is identified exactly, as the full measurement operator is injective, and consequently the learned model reproduces the data-generating dynamics along the states visited by the data. If in addition the parameters of the known physical term are identified or no such parameters are present, then also the learned nonlinearity coincides with the data-generating one along these states, while its behavior elsewhere, as well as the splitting of the dynamics into a known and an unknown part, is determined by the regularization. We refer to the discussion after Proposition \ref{prop:p_dagger}. In addition to the question of recovering $(\varphi^\dagger, u^\dagger, f^\dagger)$ as unique solution to \eqref{demo_dagger}, it is necessary to analyze in what sense parametrized solutions $(\varphi, u, f_\theta)$ of
\begin{equation}
	\label{demo_m} \tag{$p^m$} \min_{\varphi,u,\theta}\mathcal{R}_m(\varphi,u,\theta) +\sum_{l=1}^L (\lambda^m\Vert\partial_t u^l-F(t, u^l, \varphi^l)-f_\theta(t,u^l)\Vert^q+   \mu^m \Vert K^mu^l-y^{m,l}\Vert^r)
\end{equation}
converge to solutions of \eqref{demo_dagger} for some $1\leq q,r<\infty$. Here, $(K^m)_m$ is a sequence of measurement operators suitably approaching $K^\dagger$, $(y^{m,l})_m$ with $y^{m,l}\approx K^m u^{\dagger,l}$  is a sequence of (noisy) measured data and $\lambda^m, \mu^m>0$ are regularization parameters. More concretely, we suppose the measured data $(y^{m,l})_m$ to fulfill the noise estimation given by
\begin{align}
	\label{noise_estimation1} \Vert y^{m,l}-K^m u^{\dagger, l}\Vert\leq \delta(m)
\end{align}
such that $\delta(m)\to 0$ as $m\to \infty$. This is in fact our only requirement on the noise model, i.e., our analytic results do not require assumptions on the noise distribution nor the nature of the noise such as homo-/heteroscedasticity. Note further that although \eqref{noise_estimation1} anticipates comparable noise levels on the different measurements $y^{m,l}$ for $l=1,\dots, L$, an extension to different noise levels for different measurements $l$ is straightforward by introducing different weightings in \eqref{demo_m} for the data fidelities.

In order to obtain our uniqueness and consistency results, a suitable regularity of $f$, approximation properties of the parametrized approximation class $\mathcal{F} = \{ f_\theta|\, \theta \in \Theta \}$ (such as neural networks) as well as a suitable choice of the regularization functionals $\mathcal{R}_m$ and $\mathcal{R}^\dagger$ are necessary. It turns out from our analysis that the class of locally $W^{1,\infty}$-regular functions is suitable for $f$ and that parameter-growth estimates and local $W^{1,\infty}$ approximation capacities are required for $\mathcal{F}$. We refer to Assumption \ref{ass_uniqueness}, iv) below for precise requirements on $\mathcal{F}$ which are, as we argue in our work, satisfied for example by certain classes of neural networks. Regarding the regularization functionals, a suitable choice is
\begin{equation}
	\label{intro_regularizations}
	\begin{aligned}
		\mathcal{R}_m(\varphi,u,\theta) & = \mathcal{R}_0(\varphi,u)+\Vert f_\theta\Vert_{L^\rho}^\rho+\Vert\nabla f_\theta\Vert_{L^\infty}+\nu^m\Vert\theta\Vert, \\
		\mathcal{R}^\dagger(\varphi,u,f) & =\mathcal{R}_0(\varphi,u)+\Vert f\Vert_{L^\rho}^\rho+\Vert\nabla f\Vert_{L^\infty},
	\end{aligned}
\end{equation}
with the parameters $\nu^m$ appropriately converging to zero as $m \rightarrow \infty$ and $1< \rho< \infty$. Here, the norms $\Vert  \cdot \Vert_{L^\rho}^\rho+\Vert\nabla (\cdot )\Vert_{L^\infty} $ (as opposed to a standard $L^p$ norm) are necessary to ensure convergence of $f_\theta$ to $f$ as functions in $W^{1,\infty}$, which in turn is necessary for convergence of the PDE model. We showcase the necessity of choosing this $W^{1,\infty}$-type norm in contrast to only using the $\Vert  \cdot \Vert_{L^\rho}^\rho$-norm in an example in Appendix \ref{app:analytic_example}: In this example, it is possible to recover the underlying unique hidden physics when using the suggested $W^{1,\infty}$-type regularization approach, while a standard $\Vert  \cdot \Vert_{L^\rho}^\rho$-type approach fails to do so.

We emphasize that the single index $m$ is used throughout this work as a common scale parameter which simultaneously controls
\begin{itemize}
	\item~ the measurement operators $K^m$ and the measured data $y^{m,l}$, together with their noise level $\delta(m)$,
	\item~ the regularization parameters $\lambda^m,\mu^m$ and $\nu^m$,
	\item~ and the approximation class $\mathcal{F}$, whose capacity is refined as $m$ increases, see Section \ref{sec:problem_setting}.
\end{itemize}
The rates at which these quantities have to be coupled are stated in Theorem \ref{thm:uniqueness}.

The norm  $\| \theta \| $ on the finite-dimensional parameters $\theta$ in \eqref{intro_regularizations} is necessary for well-posedness of \eqref{demo_m}, but will vanish in the limit as $m\rightarrow \infty$. The choice $1<\rho<\infty$ is necessary for ensuring uniqueness of a regularization-minimizing solution \eqref{demo_dagger} via strict convexity, and $\mathcal{R}_0(\varphi,u)$ can be any problem-dependent regularization. Note that here, the main ingredients for obtaining uniqueness are that $f$ is no longer parametrized by $\theta$ in the limit (e.g. can be any $W^{1,\infty}$ function) together with a classical strict convexity argument.
While our results on using $W^{1,\infty}$ regularization are formulated for the structured-model-learning-setting here, we note that similar requirements and results can also be expected when learning entire PDE models e.g. with neural operators.

An important question from the computational perspective is how the $W^{1,\infty}$-norm can be approximated in practice or whether there exist scalable surrogates. In fact, a direct way to estimate the $\Vert \nabla (\cdot)\Vert_{L^\infty}$ term is to draw a certain number of uniformly random samples of the gradient over the considered domain and determine a global extremum over the samples. Aside from this direct approach, this question has been broadly considered in literature in the context of ensuring Lipschitz-stability of neural networks, see for example \cite{Latorre2020Lipschitz}, where the Lipschitz constant of neural networks is estimated via sparse polynomial optimization using linear or semidefinite programming, \cite{Jordan2020}, where the local Lipschitz constant of ReLU networks is computed exactly using mixed-integer programming and by providing upper bound in case of an early stop of the algorithm, \cite{Gouk2021} where a regularization of the Lipschitz constant is carried out by computing upper bounds during the training process using layerwise constants, \cite{huang2023on} which applies a least squares regression algorithm for estimating the Lipschitz constant and also provides lower bounds on the sample complexity of the underlying problem and \cite{Fazlab2019}, where the problem of estimating the Lipschitz constant is interpreted as a semidefinite program. See also \cite{Bungert2021} which considers variational regularization for controlling the Lipschitz constant of a neural network.

\noindent \textbf{Contributions.} Following the above concept, we provide a comprehensive analysis of structured model learning in a general setting. Our main contribution is a precise mathematical setup under which we prove the above-mentioned uniqueness and consistency results. In detail, the contributions of this work are the following.
\begin{itemize}[leftmargin=*,itemsep=2pt,topsep=4pt, labelsep=5pt]
	\item A general functional analytic framework for structured model learning, covering the learning problem \eqref{min_prob} and the assumptions under which it is analyzed (Section \ref{sec:problem_setting} and Subsection \ref{subsec:assumptions}).
	\item \textit{Uniqueness} (Proposition \ref{prop:p_dagger}): In the limit of full, noiseless measurements, the state, the unknown nonlinearity and the parameters of the known physical term are uniquely determined as regularization-minimizing solutions of the limit problem \eqref{eq:limit_problem}.
	\item \textit{Consistency} (Theorem \ref{thm:uniqueness}): Reconstructions obtained from incomplete, noisy measurements with the parameterized classes $\mathcal{F}^m$ approximate this unique solution as $m\to\infty$.
	\item Well-posedness of the learning problem \eqref{min_prob} for fixed $m$ (Appendix \ref{app:existence}).
	\item Concrete classes to which the framework applies: Classes of neural networks for $\mathcal{F}$ (Subsection \ref{subsec:neuralnet}) as well as linear and nonlinear (in the state) examples for the physical term $F$ (Subsection \ref{subsec:physicalterm}).
\end{itemize}
\noindent Proposition \ref{prop:p_dagger} and Theorem \ref{thm:uniqueness} are the two main results of this work and are compared in Table \ref{tab:main_results} below. Notably, our setup differs from standard model-learning frameworks commonly used in practice, in particular with respect to the choice of regularization for the approximating functions. In view of this, a practical consequence of our work can be a suggestion of appropriate regularization functionals for model learning that ensure unique recovery in the full-measurement/ vanishing noise limit. Indeed, as an example in Appendix \ref{app:analytic_example} shows, without appropriate regularization a unique recovery in the limit may fail. The following proposition, which is a consequence of Proposition \ref{prop:p_dagger} and Theorem \ref{thm:uniqueness} below, showcases our main results for a specific, linear example.
\begin{proposition} 
	\label{prop:demo} Let the space setup be given by the state space $V = H^1(\Omega)$, the image space $W = L^{2}(\Omega)$, the measurement space $Y = L^2(\Omega)$ and parameter space $X_\varphi = H^1(\Omega)$ for a bounded interval $\Omega\subseteq \mathbb{R}$ with the time extended spaces
	\[ \mathcal{V}=W^{1,2,2}(0,T;V),\quad \mathcal{W}=L^2(0,T;W),\quad \mathcal{Y}=L^2(0,T;Y).
	\] Consider the one-dimensional convection equation with unknown reaction term
	\begin{align}
		\label{eq:transport_pde} \partial_t u^l = \varphi^l\cdot \nabla u^l+f(u^l)
	\end{align}
	where $\varphi^l\in X_\varphi$ for $1\leq l\leq L$ subject to $K^\dagger u^l = \hat y^l$ with $K^\dagger:\mathcal{V}\to \mathcal{Y}$ an injective, linear and weak-strong continuous operator (e.g., an injective compact operator) and $(\hat y^l)_l\subseteq\mathcal{Y}$ full measurement data. Suppose that there exist admissible $\hat{f}\in W^{1,\infty}(\mathbb{R})$, $\hat{u}\in \mathcal{V}^L$ and $\hat{\varphi}\in X_\varphi^L$ solving \eqref{eq:transport_pde} such that $K^\dagger \hat{u}^l=\hat{y}^l$ for $1\leq l\leq L$. Assume that $f$ is approximated by neural networks $f_\theta$ of the form in \cite[Theorem 1]{belomestny23} parameterized by $\theta\in \Theta^m$ with $m\in\mathbb{N}$ a scale of approximation. Suppose that $(K^m)_m$ is a sequence of bounded linear operators converging to $K^\dagger$ in the operator norm and $(y^{m,l})_m\subseteq \mathcal{Y}$ a sequence of measurement data converging to $y^l$. Let further $U\subseteq \mathbb{R}$ be a sufficiently large bounded open interval. 
	
	Then the following two assertions hold true.
	\begin{enumerate}[label=\roman*),leftmargin=*,labelsep=5pt]
		\item \textbf{Uniqueness of the limit problem} (cf.\ Proposition \ref{prop:p_dagger}). There exists a unique solution $(\varphi^\dagger, u^\dagger, f^\dagger)$ to the vanishing noise limit problem
	\begin{equation}
		\tag{$p^\dagger$}
		\begin{aligned}
			&\min_{\substack{\varphi\in H^1(\Omega)^L,u\in\mathcal{V}^L,\\f\in W^{1,\infty}(U)}}\sum_{l=1}^L(\Vert \varphi^l\Vert_{H^1(\Omega)}^2+\Vert u^l\Vert_{\mathcal{V}}^2)+\Vert f\Vert_{L^2(U)}^2+\Vert\nabla f\Vert_{L^\infty(U)}\\
			&\text{s.t. }\forall l: \quad \partial_t u^l = \varphi^l\cdot \nabla u^l+f(u^l), \quad K^\dagger u^l = \hat y^l.
		\end{aligned}
	\end{equation}
	
		\item \textbf{Consistency} (cf.\ Theorem \ref{thm:uniqueness}). Let $f^\dagger\in \mathcal{C}^3(U)$ and let $\lambda^m, \mu^m\to \infty$, $\nu^m\to 0$ as $m\to \infty$ at a certain rate depending on the neural network architectures and $(y^{m,l})_m$. For each $m\in\mathbb{N}$ let $(\varphi_m, u_m, \theta_m)$ be a solution to
	\begin{equation}
		\tag{$p^m$}
		\begin{aligned}
			&\min_{\varphi\in H^1(\Omega)^L,u\in\mathcal{V}^L,\theta\in \Theta^m}\sum_{l=1}^L(\Vert \varphi^l\Vert_{H^1(\Omega)}^2+\Vert u^l\Vert_{\mathcal{V}}^2)+\Vert f_\theta\Vert_{L^2(U)}^2+\Vert\nabla f_\theta\Vert_{L^\infty(U)}\\
			&\qquad+\nu^m\Vert\theta\Vert+\sum_{l=1}^L (\lambda^m\Vert\partial_t u^l-\varphi^l\cdot\nabla u^l-f_\theta(u^l)\Vert^2_{\mathcal{W}}+ \mu^m \Vert K^mu^l-y^{m,l}\Vert^2_{\mathcal{Y}})
		\end{aligned}.
	\end{equation}		
		Then it holds true that $\varphi_m\rightharpoonup \varphi^\dagger$ in $H^1(\Omega)^L$, $u_m\rightharpoonup u^\dagger$ in $\mathcal{V}^L$ and $f_{\theta_m}\overset{*}{\rightharpoonup} f^\dagger$ in $W^{1,\infty}(U)$.
	\end{enumerate}
\end{proposition}
\begin{proof}
	See Appendix \ref{app:demo}.
\end{proof}
It is important to emphasize that, among the assumptions stated in Proposition \ref{prop:demo}, the most restrictive one from a practical perspective is the one on existence of an admissible solution $\hat{u}$ with sufficiently high regularity. This can be viewed as an implicit assumption on the existence of a sufficiently regular transport field $\hat{\varphi}$ that is compatible with $\hat{u}$. Indeed, the regularity of $\hat u$ as a solution to the transport equation \eqref{eq:transport_pde} depends not only on the source term $\hat f$, which is essentially Lipschitz continuous, but crucially on the regularity of the transport field $\hat \varphi$. The relationship between the smoothness of the transport field and the well-posedness and regularity of solutions has been extensively studied. The foundational work \cite{DiPerna1989} establishes well-posedness for transport equations when the transport field has Sobolev regularity, connecting the regularity of solutions to that of the flow generated by the field. A more recent survey of these results is given in \cite{Ambrosio2017}. Further developments, including \cite{Bru2020, Crippa2008, Crippa2022, DeLellis2008,Modena2018}, study how the smoothness of the transport field affects the stability and regularity of solutions, and demonstrate that, at critical levels of regularity, the solutions may lose uniqueness and smoothness. In fact, for Lipschitz continuous $\hat \varphi$ and $\hat f$ together with initial condition $u_0\in H^2(0,1)$ one can show, following \cite[Theorem 1.2]{Bru2020}, that $u\in W^{1,2,2}(0,T; H^1(0,1))$. Thus, the regularity assumption on the admissible state above can be interpreted as a regularity assumption on  the parameters $\hat \varphi$, $\hat f$ and $\hat u_0$. A second restriction, independent of the above, is the regularity $f^\dagger\in\mathcal{C}^3(U)$ of the limiting hidden physics. It does not originate in our general theory, which only requires Assumption \ref{ass_uniqueness}, iv), but in the approximation result used to verify the latter. Approximation results available under weaker regularity assumptions, such as the one underlying Proposition \ref{prop:raslan}, lower this requirement accordingly.

\medskip

\noindent For orientation, Table \ref{tab:main_results} summarizes side by side the two main results of this work, the problems they concern, the assumptions they require and their implications.

\begin{table}[htbp]
	\small
	\centering
	\begin{tabular}{@{}>{\raggedright\arraybackslash}p{0.13\linewidth} >{\raggedright\arraybackslash}p{0.405\linewidth} >{\raggedright\arraybackslash}p{0.405\linewidth}@{}}
		\toprule
		& \textbf{Proposition \ref{prop:p_dagger} (Uniqueness)} & \textbf{Theorem \ref{thm:uniqueness} (Consistency)}\\
		\midrule
		Question &
		Do full, noiseless measurements together with the regularization uniquely determine the unknowns? &
		Do reconstructions from incomplete, noisy measurements approach the unique solutions of \eqref{eq:limit_problem}?\\[4pt]
		Problem &
		The limit problem \eqref{eq:limit_problem}, with full data $(\hat y^l)_l$ and $f\in W^{1,\infty}(U)^N$ &
		The learning problems \eqref{min_prob_uniqueness} for $m\in\mathbb{N}$, with operators $K^m$, data $(y^{m,l})_l$ of noise level $\delta(m)$ and $f_\theta\in\mathcal{F}^m$\\[4pt]
		Assumptions &
		Assumptions \ref{ass_init_set}-\ref{ass_uniqueness}, \emph{without} the approximation capacity condition in Assumption \ref{ass_uniqueness}, iv) &
		Assumptions \ref{ass_init_set}-\ref{ass_uniqueness}, with Assumption \ref{ass_uniqueness}, iv) satisfied for $f^\dagger$, and the parameters coupled by $\lambda^m m^{-\beta q}, \mu^m\delta(m)^r, \nu^m\psi(m) = o(1)$\\[4pt]
		Mechanism &
		Injectivity of $K^\dagger$ (Assumption \ref{ass_uniqueness}, v)) identifies the state, strict convexity of $\mathcal{R}_0$ and of $\Vert\cdot\Vert_{L^\rho(U)}^\rho$ together with $F$ affine in $\varphi$ (Assumption \ref{ass_uniqueness}, vi), viii)) identify $\varphi$ and $f$ &
		The gradient penalty yields a uniform $W^{1,\infty}(U)$-bound, hence uniform convergence by $W^{1,\infty}(U)\hookdoubleheadrightarrow\mathcal{C}(\overline{U})$, combined with strong convergence of the states (Aubin-Lions) and the approximation capacity Assumption \ref{ass_uniqueness}, iv)\\[4pt]
		Assertion &
		Existence of a unique solution $(\varphi^\dagger, u^\dagger, u_0^\dagger, g^\dagger, f^\dagger)$ &
		$\varphi^m\rightharpoonup\varphi^\dagger$, $u^m\rightharpoonup u^\dagger$, $u_0^m\rightharpoonup u_0^\dagger$, $g^m\rightharpoonup g^\dagger$ and $f_{\theta^m}\overset{*}{\rightharpoonup}f^\dagger$\\[4pt]
		Interpretation &
		The state is identified exactly. The nonlinearity and the physical parameters are identified as regularization-minimizing solutions, reproducing the data-generating dynamics along the visited states (see the discussion following Proposition \ref{prop:p_dagger}) &
		The learning approach is consistent, i.e., solutions of the parameterized problems converge to the unique regularization-minimizing solution of the limit problem\\[4pt]
		Instance &
		Proposition \ref{prop:demo}, i) (exemplary case), Proposition \ref{prop:p_dagger} (full result) & 
		Proposition \ref{prop:demo}, ii) (exemplary case), Theorem \ref{thm:uniqueness} (full result) \\
		\midrule
		\multicolumn{3}{@{}p{0.96\linewidth}@{}}{Complementarily, for fixed $m$: Well-posedness of the learning problem \eqref{min_prob} under Assumptions \ref{ass_init_set}-\ref{ass_phys_term} alone, see Appendix \ref{app:existence}.}\\
		\bottomrule
	\end{tabular}
	\caption{Overview of the two main results of this work.}
	\label{tab:main_results}
\end{table}

\noindent \textbf{Related work.} This work is mainly motivated by \cite{AHN23} on data-driven structured model learning which proposes an \textit{all-at-once} approach for learning-informed parameter identification, i.e., determining the state simultaneously with the nonlinearity and the input parameters. Note that \cite{AHN23} considers single PDEs, while our work generalizes to PDE systems where the unknown term may additionally depend on higher-order derivatives of the state variable. Besides this fundamental difference, we derive well-posedness of the learning problem under slightly different conditions, where higher regularity assumptions on the state space stated in \cite{AHN23} can be omitted if the activation function of the neural networks approximating the nonlinearities is globally Lipschitz continuous. Moreover, we treat the cases of linear and nonlinear physical terms separately. Finally, the main difference of our work to \cite{AHN23} is that we focus on unique reconstructability, whereas \cite{AHN23} is mostly focused on well-posedness of the learning problem and the resulting PDE. The main reason for choosing an \textit{all-at-once} approach (see e.g. \cite{Kaltenbacher16, KaltNguy22}) in general is the possibility to account for practically realistic, incomplete and indirectly measured state data, which may be polluted by noise. It also circumvents the use of the parameter-to-state map, which requires regularity conditions that may not be feasible in practice (see e.g. \cite{Haber01, Kaltenbacher17, Kaltenbacher14, Nguyen19}).

In contrast to the \textit{all-at-once} setting pursued here, works that use a learning-informed control-to-state map to study the optimal control of certain PDEs are \cite{Hintermuller22,Papafits22, Dong22}. There it is assumed that the nonlinear constituents are only accessible through data-driven techniques e.g. arising from neural networks. Another related work in the field of optimal control is \cite{Court22} on nonlinearity identification in the monodomain model via neural network parameterization. We also mention the recent paper \cite{Kowalczyk24} which deals with the identification of semilinear elliptic PDEs in a low-regularity control regime. In the context of approximating nonlinearities for elliptic state equations see \cite{Spiliopoulos23}. We also mention the recent work \cite{riedl2025} which establishes global convergence guarantees for adjoint-based training of infinite-width neural networks embedded in nonlinear parabolic PDEs. In contrast to such results, our second main result concerns the reconstruction procedure rather than the training algorithm. It is a consistency result in the sense of the convergence of regularization methods, see e.g. \cite[Chapter 3]{Scherzer2008}, where the regularization parameters are coupled to the vanishing noise level, and the convergence of algorithms used to solve the learning problems is not addressed here. For structured model learning for ODEs we refer to \cite{Ebers24, Goyal21}. See also \cite{NgocNguyen2025} on regularized inversion for hidden reaction law discovery.
 
 Recent work incorporates conservation laws and symmetries into machine-learning models to improve physical fidelity and data efficiency. Soft-constraint methods enforce approximate conservation through regularized loss terms \cite{jagtap20, Li2024, Wu2022}, whereas exact-conservation approaches leverage integral forms \cite{Hansen2024}, or apply adaptive correction mechanisms \cite{CardosoBihlo2025,Geng2024, Schoenlieb2025}. Complementary strategies embed conservation directly into the architecture, via hard constraints in output layers \cite{Sturm21}, projection onto admissible solution spaces \cite{Negiar23}, or symmetry encoding \cite{Mueller23}, with additional architectural designs explored in \cite{morina2025b,Liu23, Yu2023, Richterpowell2022}. An important aspect of model learning is interpretability, which seeks representations that are accurate, parsimonious, and physically consistent. A primary goal is to recover simple laws that faithfully describe the underlying data \cite{Angelis2023, Quade2016, Brunton2017, Schmidt2009}. For a comprehensive overview of related methods, see \cite{cava2021}. From the perspective of inverse problems, model learning is, at its core, an identification problem, requiring that the inferred model is uniquely determined by the data to represent the true system dynamics rather than an equivalent alternative. Foundational contributions \cite{Bellman1970,Cobelli1980,DiStefano1980,Miao2011} formalize structural and parameter identifiability in dynamical systems. For the motivation of uniqueness results for parameter identification, we refer to the works \cite{Cannon1980,Egger15,Roesch96}, which derive uniqueness from stability estimates. Uniqueness has also been established for semilinear parabolic equations \cite{Isakov1993} and for the recovery of nonlinear diffusion coefficients \cite{Kaltenbacher2021b}. Further results on parameter identification for elliptic equations include \cite{Acar1993,Alessandrini1986,Knowles2001}. Foundational overviews of parameter identification and related inverse problems are provided in \cite{Engl1996} on deterministic regularization theory, \cite{Klibanov2013} on coefficient inverse problems, \cite{Banks1989} on PDE parameter estimation, and \cite{Kaipio2005} on Bayesian inverse problems. Beyond coefficients, model identification targets entirely unknown PDE components, with applications to reaction-diffusion systems \cite{DuChateau1985,Kaltenbacher2020,Kaltenbacher2020b,Kaltenbacher2025}, semilinear equations \cite{Feizmohammadi2024,Kian2023b,Kian2023} and hyperbolic inverse sources \cite{Jiang2017,Yamamoto1995}.

 Nonetheless, there is little hope to obtain results of this kind for the general system \eqref{introsystem}, even if the known physical term is linear in its physical input parameters due to the ambiguity of shift perturbations. In this respect, it seems indispensable to exploit the structural/regularity properties of the unknown term $f$ and the input parameter $\varphi$, as it is in this work and in \cite{kutyniok23}, which was already discussed above. For the sake of completeness we also mention the recent preprint \cite{Kut24}, extending the results of \cite{kutyniok23} on identifiability for symbolic recovery of differential equations to the noisy regime. Note that both works \cite{Kut24,kutyniok23} focus on unique identifiability per se, i.e. the classification of uniqueness, whereas our work provides an analysis-based guideline guaranteeing unique reconstructability in the limit of a practical PDE-based model learning setup.

\paragraph{Structure of the paper.} In Section \ref{sec:problem_setting} we present the problem setting under consideration. The necessary assumptions are outlined in detail in Subsection \ref{subsec:assumptions}. In Subsection \ref{subsec:neuralnet}, applicability of our general assumptions for $\mathcal{F}$ being a certain class of neural networks are discussed. Applicability of the assumptions on the known physical term are discussed in Subsection \ref{subsec:physicalterm}, with examples both for the linear and nonlinear case. Our main results, on unique reconstructability in the limit problem and on consistency of the corresponding reconstructions, are presented in Section \ref{sec:uniqueness}. To ensure a concise presentation of our results, most proofs are covered in the appendix. The results of Subsection \ref{subsec:neuralnet} are proven in Appendix \ref{app:neural_networks} and those of Subsection \ref{subsec:physicalterm} are given in Appendix \ref{app:physical_term}.  In Appendix \ref{app:existence} well-posedness of the main minimization problem is verified under our general assumptions. In Appendix \ref{app:demo} a proof of Proposition \ref{prop:demo}, showcasing our main results for a specific, linear example, is sketched. Finally, in Appendix \ref{app:analytic_example} an example is presented covering the necessity of the proposed regularization for unique recovery in the limit. For a first reading, the main line of the paper is given by the problem formulation in Section \ref{sec:problem_setting}, the assumptions and their discussion in Subsection \ref{subsec:assumptions}, and the main results in Section~\ref{sec:uniqueness}, which are summarized in Table \ref{tab:main_results} above. Subsections \ref{subsec:neuralnet} and \ref{subsec:physicalterm}, which verify the assumptions for concrete classes, as well as the appendices, may be consulted as needed.
\section{Problem setting}
\label{sec:problem_setting} In the general case, we are interested in obtaining nonlinearities $(f_n)_n$, states $(u^l_n)_{n,l}$, parameters $(\varphi^l_n)_{n,l}$, initial conditions $(u_{0,n}^l)_{n,l}$ and boundary conditions $(g^l_n)_{n,l}$ as solutions of the following system of nonlinear PDEs:
\begin{equation}
	\label{rdsystem}
	\begin{aligned}
		\frac{\partial}{\partial t}u_n^l & = F_n(t,u^l_1,\dots, u^l_N,\varphi_n^l)+f_n(t,\mathcal{J}_\kappa u^l_1,\dots, \mathcal{J}_\kappa u^l_N), \\
		u^l_n(0) &= u^l_{0,n},\\
		\gamma (u^l _n) &= g_n^l
	\end{aligned}\tag{$S$}
\end{equation}
Here, $n=1,\ldots,N$ denotes the number of PDEs and $l=1,\ldots,L$ the number of measurements of different states (with different parameters) that we will have at our disposal for obtaining the $f_n$.

In the above system, the states $u^l_n \in \mathcal{V}$ are given as $u^l_n:(0,T) \rightarrow V$ with $T>0$ and $V$ a static state space of functions $v:\Omega \rightarrow \mathbb{R}$ with $d \in \mathbb{N}$ and $\Omega \subset \mathbb{R}^d$ a bounded Lipschitz domain, $ X_\varphi \ni \varphi_n^l$ is a static parameter space, $ H \ni u^l_n (0),u^l_{0,n}$ is a static initial trace space, and $\mathcal{B} \ni g_n^l $ is a boundary trace space with $g_n^l:(0,T) \rightarrow B$, $B$ the static boundary trace space and $\gamma: \mathcal{V} \rightarrow \mathcal{B}$ the boundary trace map. The (known) physical terms $F_n$ are given as Nemytskii operators of
\begin{equation}
	\label{functions_Fi}
	\begin{aligned}
		F_n: (0,T)\times V^N\times X_\varphi&\to W\\
		(t,u_1, \dots, u_N,\varphi)&\mapsto F_n(t,u_1,\dots, u_N,\varphi)
	\end{aligned}
\end{equation}
with $W$ a static image space and $\mathcal{W}$ the corresponding dynamic version. The $\mathcal{J}_\kappa$ are derivative operators
\begin{equation}
	\label{Jdifferential}
	\begin{aligned}
		\mathcal{J}_\kappa:V&\to \otimes_{k=0}^\kappa V_k^\times\\
		v&\mapsto (v, J^1v, \dots, J^\kappa v)
	\end{aligned}
\end{equation}
with the Jacobian mappings $J^k$ given as
\begin{align}
	\label{Jl_operator} J^k:V\to V_k^\times, ~ ~ v\mapsto (D^\beta v)_{\vert\beta\vert = k}.
\end{align}
Here, $\kappa \in \mathbb{N}_0$ is the maximal order of differentiation, $V_k $ with $V\hookrightarrow V_k$ are such that $D^\beta v\in V_k$ for $1\leq \vert\beta\vert=k\leq \kappa$ with $\beta\in\mathbb{N}_0^d$ and $\vert \beta\vert=\beta_1+\dots+\beta_d$. Here we use "$\hookrightarrow$" to denote a continuous embedding and "$\hookdoubleheadrightarrow$" to denote a compact embedding. Furthermore, with $V_0:=V$, we define $V_k^\times = \otimes_{i=1}^{p_k}V_k$ where $p_k = \binom{d+k-1}{k}$ for $0\leq k\leq \kappa$. The nonlinearities $f_n$ are given as Nemytskii operators of
\begin{align*}
	f_n: (0,T)\times (\otimes_{k=0}^\kappa V_k^\times)^N&\to W\\
	(t,(v_1^k)_{0\leq k\leq\kappa},\dots, (v_N^k)_{0\leq k\leq \kappa})&\mapsto f_n(t,(v_1^k)_{0\leq k\leq\kappa},\dots, (v_N^k)_{0\leq k\leq \kappa})
\end{align*}
where $f_n: (0,T) \times(\otimes_{k=0}^\kappa \mathbb{R}^{p_k})^N \rightarrow \mathbb{R}$ is extended to $f_n: (0,T)\times (\otimes_{k=0}^\kappa V_k^\times)^N\to W$ via $f_n(t,v)(x):= f_n(t,v(x))$. We will approximate them with parameterized approximation classes
\begin{align}
	\label{param_approx_classes} \mathcal{F}_n^m = \{f_{\theta_n, n}:(0,T)\times (\otimes_{k=0}^\kappa \mathbb{R}^{p_k})^N\to \mathbb{R}~|~\theta_n\in\Theta_n^m\}
\end{align}
where $m \in \mathbb{N}$ is the scale of approximation and $\Theta_n^m$ are parameter sets. Here, we further define $\Theta^m=\otimes_{n=1}^N\Theta_n^m$ and $\mathcal{F}^m = \otimes_{n=1}^N\mathcal{F}_n^m$.

Approximation of the $f_n$ via the $f_{\theta_n,n}$ will be achieved on the basis of noisy measurements $y^l \approx K^m u^l$, with the $K^m: \mathcal{V}^{N} \rightarrow \mathcal{Y}$ being measurement operators (for scale $m \in \mathbb{N}$) and $\mathcal{Y}$ a space of functions $y:(0,T) \rightarrow Y$ with $Y$ a static measurement space. To this aim, we will analyze the minimization problem
\begin{align}
	\label{min_prob}
	\notag\min_{\substack{\varphi\in X_\varphi^{N\times L},\theta\in\Theta^m,\\ u\in\mathcal{V}^{N\times L}, u_0\in H^{N\times L},\\ g\in\mathcal{B}^{N\times L}}}
	&\sum_{1\leq l\leq L}\lambda\Vert \frac{\partial}{\partial t}u^l-F(t, u^l,\varphi^l)-f_{\theta}(t,\mathcal{J}_\kappa u^l)\Vert_{\mathcal{W}}^q+\mathcal{R}(\varphi,u,\theta, u_0, g)\\[-0.6cm]
	+&\sum_{1\leq l\leq L}\bigg[\lambda\Vert u^{l}(0)-u_{0}^l\Vert_{H} ^2+\lambda\mathcal{D}_{\BC}(\gamma(u^l)-g^l)+\mu\Vert K^mu^l-y^l\Vert_\mathcal{Y}^r\bigg]\tag{$\mathcal{P}$}
\end{align}
where $\mathcal{D}_{\BC}$ and $\mathcal{R}$ are suitable discrepancy and regularization functionals, respectively. Note that here, notation wise, we use a direct vectorial extension over $n=1,\ldots,N$ of all involved spaces and quantities, e.g., $F(t,u^l,\varphi^l)=(F_n(t,u^l,\varphi^l_n))_{n=1}^N$. We emphasize that in \eqref{rdsystem} and \eqref{min_prob} also the initial conditions $u_0$ and the boundary conditions $g$ are regarded as unknowns. The situation that they are known is covered by our setting and is addressed in Remark \ref{rem:known_ic_bc}.
\subsection{Assumptions}
\label{subsec:assumptions} The following assumptions, motivated by \cite[Assumption 1]{AHN23}, encompass all requirements necessary to tackle the goals of this work. Under Assumption \ref{ass_init_set}, \ref{ass_param_app_class} and \ref{ass_phys_term} we verify well-posedness of \eqref{min_prob}. Additionally, under Assumption \ref{ass_uniqueness}, we will establish our results on unique reconstructability in the limit $m \rightarrow \infty$. The assumptions are stated first and then discussed, grouped by their purpose, in the paragraphs following Assumption \ref{ass_uniqueness}. Table \ref{tab:ass_roles} there summarizes where each of them enters the analysis.
\begin{assumption}[Functional analytic setup]
	\label{ass_init_set}
	\hspace{1cm}\\
	\underline{Spaces/Embeddings:}
	\begin{enumerate}[label=\roman*)]
		\item ~For $\kappa \in \mathbb{N}$, suppose that the state space $V$, the spaces $V_k$ for $1\leq k\leq \kappa$, the image space $W$, the observation space $Y$, the initial trace space $H$, the boundary trace space $B$ and the space $\tilde{V}$ are separable, reflexive Banach spaces. Further assume that the parameter space $X_\varphi$ is a reflexive Banach space and let $\Theta^m_n$, for $n=1, \ldots,N$ and $m \in \mathbb{N}$ be closed parameter sets, each contained in a finite-dimensional space.\\[-4pt]
		\item ~Let $\Omega \subset \mathbb{R}^d $ with $d \in \mathbb{N}$ be a bounded Lipschitz domain and let the following embeddings hold:\vspace*{-0.1cm}
		\begin{gather*}
			H\hookrightarrow W, ~ ~ V\hookrightarrow H\hookrightarrow \tilde{V}\hookrightarrow W, ~ ~ V\hookdoubleheadrightarrow W^{\kappa,\hat{p}}(\Omega),\\
			L^{\hat{p}}(\Omega)\hookrightarrow V_k \hookrightarrow L^{\hat{q}}(\Omega) ~ \text{for} ~ 1\leq k\leq \kappa, ~ ~ V\hookrightarrow Y, ~ ~  L^{\hat{q}}(\Omega)\hookrightarrow W
		\end{gather*}
		and either $W^{\kappa,\hat{p}}(\Omega)\hookrightarrow \tilde{V}$ or $\tilde{V}\hookrightarrow W^{\kappa,\hat{p}}(\Omega)$ for some $1\leq \hat{q}\leq \hat{p} <\infty$.
		\item ~Let $T>0 $ and the extended spaces be defined by $\mathcal{W}=L^q(0,T;W),$\vspace*{-0.2cm}
		\begin{gather*}
			\mathcal{V} = L^p(0,T;V)\cap W^{1,p,p}(0,T;\tilde{V}), \mathcal{Y} = L^r(0,T;Y), \mathcal{B} = L^s(0,T;B),\\
			\mathcal{V}_0=\mathcal{V}_0^\times :=\mathcal{V}, \mathcal{V}_k = L^p(0,T;V_k), ~ ~ \mathcal{V}_k^\times = L^p(0,T;V_k^\times) ~ ~  \text{for} ~ 1\leq k\leq \kappa
		\end{gather*}
		for some $1\leq p,q,r,s<\infty$ with $p\geq q$, $p\geq s$. We refer to \cite[Chapter 7]{Roubíček2013} for the definition and properties of (Sobolev-)Bochner spaces.\\[-4pt]
		
		\hspace{-0.4cm}\underline{Trace map:}\\[-10pt]
		\item ~Assume that the boundary trace map $\gamma: \mathcal{V} \rightarrow \mathcal{B}$ is linear and continuous.\\[-4pt]
		
		\hspace{-0.4cm}\underline{Measurement operator:}\\[-10pt]
		\item ~Suppose that the operator $K^m:\mathcal{V}^{N}\to\mathcal{Y}$ is weak-weak continuous for $m\in \mathbb{N}$.\\[-4pt]
		
		\hspace{-0.4cm}\underline{Energy functionals:}\\[-10pt]
		\item ~Assume that the discrepancy term $\mathcal{D}_{\BC}: \mathcal{B}^{N}\to [0,\infty]$ is weakly lower semicontinuous, coercive and fulfills $\mathcal{D}_\BC(z)=0$ iff $z=0$. Suppose that the regularization functional $ \mathcal{R}: X_\varphi^{N\times L}\times \mathcal{V}^{N\times L}\times \Theta^m\times H^{N\times L}\times\mathcal{B}^{N\times L}\to [0,\infty]$ is coercive in its first three components and weakly lower semicontinuous. Further suppose that there exists $(\varphi, u, \theta,u_0,g)\in \mathbf{D}(\mathcal{R})$ with $(\gamma(u^l)-g^l)_l\subseteq \mathbf{D}(\mathcal{D}_\BC)$ where $\mathbf{D}(\mathcal{D}_{\BC})$ and $\mathbf{D}(\mathcal{R})$ denote the domains of the respective functionals.
	\end{enumerate}
\end{assumption}
The next assumption concerns properties on the parameterized nonlinearities.

\begin{assumption}[Parameterized approximation classes $(\mathcal{F}^m_n)_n$]
	\label{ass_param_app_class}
	\hspace{1cm}\\
	\underline{Nemytskii operators:}\\[-10pt]
	\begin{enumerate}[label=\roman*)]
		\item ~Assume that $f_{\theta_n,n}\in \mathcal{F}^m_n$ with $\mathcal{F}^m_n$ defined as in \eqref{param_approx_classes} induce well-defined Nemytskii operators $f_{\theta_n,n}:(\otimes_{k=0}^\kappa \mathcal{V}_k^\times)^N\to \mathcal{W}$ via
		\[ [f_{\theta_n,n}((v^k)_{0\leq k\leq\kappa})](t)(x)=f_{\theta_n,n}(t,(v^k(t,x))_{0\leq k\leq\kappa}).\vspace{0.2cm}
		\]
		
		\pagebreak[1]
		\hspace{-0.4cm}\underline{Strong-weak continuity:}\\[-10pt]
		\item~ Suppose that for each $f_{\theta_n,n}\in \mathcal{F}^m_n$ the map $$\Theta_n^m\times (\otimes_{k=0}^\kappa L^p(0,T;L^{\hat{p}}(\Omega)^{p_k}))^N \ni(\theta_n, v)\mapsto f_{\theta_n, n}(v)\in L^q(0,T;L^{\hat{q}}(\Omega))$$ is strongly-weakly continuous.
	\end{enumerate}
\end{assumption}
We require an analogous assumption for the physical PDE-term.	
\begin{assumption}[Known physical term]
	\label{ass_phys_term}
	\hspace{1cm}\\
	\underline{Nemytskii operators:}\\[-10pt]
	\begin{enumerate}[label=\roman*)]
		\item~ Assume that the $F_n$ induce well-defined Nemytskii operators
		\[ F_n:\mathcal{V}^N \times X_\varphi\to \mathcal{W} ~ ~ \text{with} ~ ~ [F_n(v,\varphi)](t)=F_n(t,v(t),\varphi).
		\]
		\hspace{-0.4cm}\underline{Weak-closedness:}\\[-10pt]
		\item ~Suppose that the $F_n:\mathcal{V}^N \times X_\varphi\to \mathcal{W}$ are weakly closed.
	\end{enumerate}
\end{assumption}
 Finally, to obtain our uniqueness results, we need to impose more regularity both on the state space and the approximation class. For that, recall the definition of the differential operator $\mathcal{J}_\kappa$ in \eqref{Jdifferential} and note that, as we will show in Lemma \ref{lem:Jbochner}, it follows from Assumption \ref{ass_init_set} that the $\mathcal{J}_\kappa$ induce suitable Nemytskii operators such that the following assumption makes sense notationally.\vspace*{-0.1cm}
\begin{assumption}[Uniqueness]
	\label{ass_uniqueness}
	\hspace{1cm}\\
	\vspace*{-0.5cm}\\
	\underline{Regularity:}\\[-10pt]
	\begin{enumerate}[label=\roman*)]
		\item ~ Assume that there exists a constant $c_{\mathcal{V}}>0$ such that $$\Vert \mathcal{J}_\kappa v\Vert_{L^\infty((0,T)\times\Omega)}\leq c_{\mathcal{V}}\Vert v\Vert_{\mathcal{V}} ~ ~ \text{for all} ~ v\in \mathcal{V}.$$
		\item ~For $D = 1+N\sum_{k=0}^\kappa p_k$, $1\leq n\leq N$, $m\in\mathbb{N}$ suppose that $\mathcal{F}_n^m\subseteq W^{1,\infty}_{\loc}(\mathbb{R}^{D})$.\\[-4pt]
		
		\item ~Suppose that for the full measurement data $\hat{y}\in \mathcal{Y}^L$ there exist admissible functions $\hat{f}\in W^{1,\infty}(\mathbb{R}^D)^N$, $\hat u \in \mathcal{V}^{N \times L}$, $\hat \varphi \in X_\varphi^{N\times L}$, $\hat u_0\in H^{N \times L}$ and $\hat g\in \mathcal{B}^{N \times L}$ solving \eqref{rdsystem} such that $K^\dagger \hat u^l = \hat y ^l$ for all $l=1,\ldots,L$.\\[-4pt]
		
		\hspace{-0.4cm}\underline{Approximation capacity of $\mathcal{F}^m$ for $f\in W^{1,\infty}_{\loc}(\mathbb{R}^{D})^N$:}
		\item ~The approximation capacity condition is considered to be satisfied for a fixed $f\in W^{1,\infty}_{\loc}(\mathbb{R}^{D})^N$ if for any bounded domain $U\subseteq \mathbb{R}^{D}$ there exist a monotonically increasing $\psi:\mathbb{N}\to \mathbb{R}$ and $c,\beta>0$ such that for $\Vert \cdot \Vert$ denoting some $l^p$-Norm for $1\leq p\leq \infty$ there exist parameters $\theta^m \in \Theta ^m$  with\vspace*{-0.1cm}
		\begin{align}
			\label{f_approximation1} \Vert f- f_{\theta^m}\Vert_{L^\infty(U)}\leq cm^{-\beta}, \qquad \Vert \theta^m\Vert\leq \psi(m)\vspace*{-0.1cm}
		\end{align}
		and $\Vert \nabla f_{\theta^m}\Vert_{L^\infty(U)}\to \Vert \nabla f\Vert_{L^\infty(U)}$ as $m\to \infty$.\\[-4pt]

		\hspace{-0.4cm}\underline{Measurement operator:}\\[-10pt]
		\item ~Suppose that for any weakly convergent sequence $(u^m)_m\subset \mathcal{V}^N$ it holds true that\vspace*{-0.1cm}
		\begin{equation}
			\label{operator_conv_weak} K^mu^m-K^\dagger u^m\to 0 \quad \text{in} ~ \mathcal{Y} \quad \text{as} ~ m \to \infty.\vspace*{-0.1cm}
		\end{equation}
	 Assume that $K^\dagger$ is injective and weak-strong continuous.\\[-4pt]

		\hspace{-0.4cm}\underline{Regularization functional:}\\[-10pt]
		\item ~Let $ \mathcal{R}_0: X_\varphi^{N\times L}\times \mathcal{V}^{N\times L}\times H^{N\times L}\times\mathcal{B}^{N\times L} \to [0,\infty]$ be strictly convex in its first component. Assume that there exists a monotonically increasing function $\pi:[0,\infty)\to [0,\infty)$ (e.g. the $p$-th root) such that
		\[ \Vert v\Vert_{\mathcal{V}}\leq \pi(\mathcal{R}_0(\cdot,v, \cdot,\cdot))
		\] for $v\in \mathcal{V}^{N\times L}$. Let $\mathcal{R}:X_\varphi^{N\times L}\times \mathcal{V}^{N\times L}\times \otimes_n \Theta_n^m\times H^{N\times L}\times\mathcal{B}^{N\times L} \rightarrow [0,\infty]$ be given as\vspace*{-0.1cm}
		\[ \mathcal{R}(\varphi, u,\theta,u_0,g)=\mathcal{R}_0(\varphi, u,u_0,g)+ \nu \|\theta\| + \Vert f_\theta\Vert_{L^\rho(U)}^\rho+\Vert\nabla f_\theta\Vert_{L^{\infty}(U)},\vspace*{-0.1cm}
		\]
		for $1<\rho<\infty$ and $U \subset \mathbb{R}^D$ a bounded Lipschitz domain.\\[-4pt]
		\item ~Assume that $U$ is sufficiently large such that it contains $\{z \in \mathbb {R}^D \,:\, \|z\|\leq \delta\}$ with $\delta = T+c_\mathcal{V}\pi(\hat{C})$
		and $\hat{C}\geq \mathcal{R}_0(\hat{\varphi}, \hat{u}, \hat{u}_0,\hat{g})+\Vert \hat{f}\Vert_{L^\rho(\mathbb{R}^D)}^\rho+\Vert \nabla\hat{f}\Vert_{L^\infty(\mathbb{R}^D)}+1$ an a-priori estimate on the admissible functions as in Assumption \ref{ass_uniqueness}, iii).\\[-4pt]

		\hspace{-.4cm}\underline{Physical term:}\\[-10pt]
		\item ~Suppose that $X_\varphi\ni\varphi\mapsto F(t, u,\varphi)\in W^N$ is affine for $u\in V^N$ and $t\in(0,T)$. Assume that $F:\mathcal{V}^N \times X_\varphi\to \mathcal{W}^N$ is weakly continuous.
	\end{enumerate}
\end{assumption}
\pagebreak[2]

\medskip

\noindent\textbf{Discussion of the assumptions.} While the above assumptions are formulated in a functional-analytic language, each of them plays a concrete and well-motivated role in the analysis presented in this work. We discuss them here grouped by their purpose. Broadly, Assumptions \ref{ass_init_set}, \ref{ass_param_app_class} and \ref{ass_phys_term} provide the infrastructure needed for \emph{well-posedness} of the learning problem \eqref{min_prob}, i.e., the existence of minimizers for each fixed $m\in\mathbb{N}$, whereas the additional conditions collected in Assumption \ref{ass_uniqueness} are required for the unique reconstructability result in the full-measurement, vanishing-noise limit $m\to\infty$.

\paragraph{Spaces and embeddings.} Well-posedness is established by the direct method of the calculus of variations (see Appendix \ref{app:existence}). That is, one selects a minimizing sequence, extracts a weakly convergent subsequence, and passes to the limit in the objective and the PDE constraint. This dictates the functional-analytic setting in Assumption \ref{ass_init_set}. Reflexivity and separability of the involved spaces, part i), guarantee that bounded sequences admit weakly, respectively weak-$*$, convergent subsequences. The continuous embeddings in part ii) arrange the state space $V$, the image space $W$, the trace and observation spaces, and the spaces $V_k$ carrying the derivatives into a chain of spaces. In this way every term of the problem is well-defined, i.e., the residual $\partial_t u-F-f$ lies in $\mathcal{W}$, the observation $K^m u$ in $\mathcal{Y}$, and the arguments of the pointwise nonlinearities in the $V_k$. The extended Sobolev-Bochner spaces in part iii), in particular $\mathcal{V}=L^p(0,T;V)\cap W^{1,p,p}(0,T;\tilde{V})$, lift the static setting to the time-dependent problem. Spatial regularity in $\mathcal{V}$ is encoded through $V$, and a well-defined temporal derivative through $\tilde{V}$. The definition of $\mathcal{V}$ yields the embedding $\mathcal{V}\hookrightarrow \mathcal{C}(0,T;\tilde{V})$ that gives meaning to the initial condition $u(0)$ (see Remark \ref{rem:regul_cond}). The exponent conditions $p\geq q$ and $p\geq s$ keep these extended embeddings consistent. Among the embeddings in part ii), the decisive one is the compact embedding $V \hookdoubleheadrightarrow W^{\kappa,\hat{p}}(\Omega)$. It is what upgrades weak convergence of the states to strong convergence of the arguments that enter the nonlinearities pointwise. Since these pointwise nonlinearities are in general not weakly continuous, this strong convergence is necessary for passing to the limit in the PDE constraint (see Lemma \ref{lem:strong_weak_cont}, the proof of Theorem \ref{thm:uniqueness}, and Remark \ref{rem:choice_V}). As $\kappa$ denotes the highest order of derivatives on which the unknown model may depend, this amounts to requiring the state space to carry more than $\kappa$ orders of regularity. Intuitively, one cannot expect to identify a model that depends on $\kappa$-th order derivatives without controlling those derivatives.

\paragraph{Trace map, measurement operator and energy functionals.} The remaining conditions in Assumption \ref{ass_init_set}, parts iv)-vi), concern the operators and functionals entering the problem. Linearity and continuity of the boundary trace map $\gamma$, part iv), ensure that the boundary condition enters as a weakly continuous constraint, so that the discrepancy term $\mathcal{D}_{\BC}(\gamma(u)-g)$ is well-defined and behaves well under weak limits. A concrete admissible choice of $\gamma$ is given in Remark \ref{rem:bdry}. Weak-weak continuity of the reduced measurement operators $K^m$, part v), makes the data-fidelity term weakly lower semicontinuous. We emphasize that at finite $m$ only this comparatively mild condition is required, whereas the stronger injectivity and weak-strong continuity needed for the reconstruction are imposed on the full measurement operator $K^\dagger$ in Assumption \ref{ass_uniqueness}, part v), discussed in the paragraph on measurements below. Finally, coercivity and weak lower semicontinuity of the discrepancy and regularization functionals, part vi), guarantee that a minimizing sequence stays bounded and that the objective does not increase in the limit, which closes the direct-method argument.

\paragraph{Parameterized class and physical term.} Assumptions \ref{ass_param_app_class} and \ref{ass_phys_term} guarantee that the nonlinear constraint in \eqref{min_prob} is compatible with this limiting procedure. The Nemytskii property ensures that the parameterized terms and the known physical term induce well-defined operators on the space-time (Bochner) spaces, while the strong-weak continuity of the parameterized nonlinearities and the weak closedness of the physical term are needed to preserve the constraint under weak limits in the direct method. Note that the parameterized term acts pointwise (through $\mathcal{J}_\kappa$, see Remark \ref{rem:jkappa}), whereas the physical term acts directly on the state.

\paragraph{Regularity and gradient control.}  For the unique reconstructability result, more regularity is required, both of the state and of the approximation class. Part i) in Assumption \ref{ass_uniqueness} imposes the uniform bound $\Vert \mathcal{J}_\kappa v\Vert_{L^\infty(\left]0,T\right[\times\Omega)}\leq c_{\mathcal{V}}\Vert v\Vert_{\mathcal{V}}$. It ensures that a sequence bounded in $\mathcal{V}$ has its derivatives up to order $\kappa$ uniformly bounded (see Remark \ref{rem:regul_cond}). Consequently, the arguments $(t,\mathcal{J}_\kappa u)$ of the nonlinearities stay in a fixed bounded set $U$. Only the behavior of the model on $U$ then matters, which is why it suffices to work with locally $W^{1,\infty}$-regular functions. An admissible model can even be taken globally $W^{1,\infty}$-regular by extension (see Remark \ref{remark:regularity_admiss}).

Let us make the scope of this condition more concrete, as it is the assumption that restricts the class of problems covered most directly. What it demands is an embedding of $\mathcal{V}$ into $W^{\kappa,\infty}(\Omega)$ in space, uniformly in time, where $\kappa$ is the differential order of the \emph{unknown} term and not the order of the equation. For $\kappa=0$, that is, when the term to be learned depends on the state itself and not on its derivatives, as for unknown reaction kinetics or source terms, the condition is a uniform bound on the state alone and requires no smoothness. It holds, for instance, whenever $\tilde{V}\hookrightarrow W^{2,2}(\Omega)$ with $d\leq 3$, by Remark \ref{rem:regul_cond} with $\tilde{\kappa}=\eta=2$, and it is not needed at all when the states are a priori bounded (Remark \ref{rem:apriori_bounded}).

For $\kappa\geq 1$ the requirement is stronger, since now the derivatives of the state up to order $\kappa$ must be bounded uniformly. By Remark \ref{rem:regul_cond} it suffices that the state be controlled in $W^{\kappa+\tilde{\kappa},\eta}(\Omega)$ for some $\tilde{\kappa}\in\mathbb{N}$ and some exponent $\eta$ with $\tilde{\kappa}\eta>d$. In words, the equation has to provide $\tilde{\kappa}$ orders of regularity beyond the order $\kappa$ of the unknown term, where $\tilde{\kappa}$ and the integrability exponent $\eta$ are balanced against the space dimension as Morrey's embedding requires. For $\tilde{\kappa}=1$, for instance, this amounts to $\kappa+1$ derivatives of the state in $L^\eta(\Omega)$ with $\eta>d$. Such regularity is available for parabolic problems with sufficiently regular coefficients and data. Problems whose solutions develop singularities consequently fall outside the scope of our uniqueness and consistency results. This concerns shocks in hyperbolic conservation laws, solutions exhibiting finite-time blow-up, as for reaction-diffusion equations with superlinear kinetics violating the growth conditions referred to in Remark \ref{rem:apriori_bounded}, and transport problems at critical regularity of the transport field, where already the forward problem may lose uniqueness and smoothness \cite{DiPerna1989,Ambrosio2017,Bru2020,Crippa2008,Crippa2022,DeLellis2008,Modena2018} (see Remark \ref{rem:admissible_state} and the discussion following Proposition \ref{prop:demo}). Here the distinction between $\kappa=0$ and $\kappa\geq 1$ matters. Phenomena in which the state remains bounded while its derivatives do not, as for the gradient blow-up occurring for the diffusive Hamilton-Jacobi equation $\partial_t u=\Delta u+\vert\nabla u\vert^p$ with a superquadratic exponent $p>2$, see, e.g., \cite{Arrieta2004}, leave the case $\kappa=0$ untouched but are outside the scope for $\kappa\geq 1$. The same applies to sharp fronts, as they arise at the free boundary of degenerate diffusion equations of porous medium type, where the state remains bounded but is not smooth across the interface \cite{Vazquez2007}. Working in favor of the condition, on the other hand, is the smoothing effect of nondegenerate parabolic problems. Their solutions gain regularity in space and time for positive times, even for rough initial data, see \cite[Subsection 7.1.3]{Evans2010} and \cite[Chapter 10]{Brezis2010}. Whether this suffices for the bound required here depends on the behavior at the initial time, and thus on the compatibility of the initial and boundary data. We emphasize that in all these situations well-posedness of the learning problems at a fixed scale $m$ is unaffected, as it requires Assumptions \ref{ass_init_set}-\ref{ass_phys_term} only.

Part ii) requires the approximation class to consist of locally $W^{1,\infty}$-regular functions. This regularity has two concrete uses in the analysis. First, a Lipschitz model induces a well-defined Nemytskii operator that is stable under strong convergence of the state, through an estimate of the form $\Vert f(\mathcal{J}_\kappa u^m)-f(\mathcal{J}_\kappa u^\dagger)\Vert_{\mathcal{W}}\leq c\Vert f\Vert_{W^{1,\infty}(U)}\Vert u^m-u^\dagger\Vert$. This is what allows passing to the limit in the nonlinearity in the proof of Theorem \ref{thm:uniqueness}. Second, together with the gradient penalty in the regularization (see below), it gives a uniform $W^{1,\infty}(U)$-bound, hence weak-$*$ compactness and, by the compact embedding $W^{1,\infty}(U)\hookdoubleheadrightarrow \mathcal{C}(\overline{U})$, uniform convergence to a limit that is again Lipschitz.

\paragraph{Approximation capacity.} The condition in Assumption \ref{ass_uniqueness}, iv), is central in the limit analysis. It requires that the growing classes $\mathcal{F}^m$ can approximate a sufficiently regular target model at a controlled rate $cm^{-\beta}$, with controlled parameter growth $\Vert\theta^m\Vert\leq\psi(m)$ and with convergence of the Lipschitz constants. Without such a condition reconstructability of the limit model by the parameterized problems cannot be expected in general. We verify this condition for suitable classes of neural networks in Subsection \ref{subsec:neuralnet}.

\paragraph{Measurements.} The measurement framework concentrates the restrictive requirements on the full measurement operator $K^\dagger$, while leaving the reduced operators $(K^m)_m$ flexible. Injectivity of $K^\dagger$ in Assumption \ref{ass_uniqueness}, v), ensures that the state to be reconstructed is uniquely determined by the data. Without injectivity, only a weaker notion of reconstruction would be possible. The operators $(K^m)_m$ need only be weak-weak continuous and are linked to $K^\dagger$ through the abstract convergence condition \eqref{operator_conv_weak}, which holds, for instance, for bounded linear operators converging in the operator norm (see Remark \ref{rem:conv_meas}).

\paragraph{Regularization.} The regularization in part vi) is needed for uniqueness of the reconstruction in the limit (Proposition \ref{prop:p_dagger}), and it has three ingredients. First, the term $\Vert f_\theta\Vert_{L^\rho(U)}^\rho$ with $1<\rho<\infty$ is strictly convex. In the limit, where $f$ is no longer tied to a finite-dimensional parameterization, this yields uniqueness of the model. Uniqueness of the physical parameter $\varphi$ follows in the same way, from strict convexity of $\mathcal{R}_0$ in $\varphi$ together with $F$ being affine in $\varphi$, part viii). Second, the gradient penalty $\Vert\nabla f_\theta\Vert_{L^\infty(U)}$ enforces the uniform $W^{1,\infty}$-bound used above, so that $f_\theta$ converges uniformly and not only in $L^\rho$. Third, part vi) requires the estimate $\Vert v\Vert_{\mathcal{V}}\leq \pi(\mathcal{R}_0(\cdot,v,\cdot,\cdot))$ for some monotonically increasing $\pi$. This bounds the state norm in terms of $\mathcal{R}_0$ and, combined with part i), controls $\Vert\mathcal{J}_\kappa u\Vert_{L^\infty}$ by $\mathcal{R}_0$. It is this estimate that guarantees the arguments $(t,\mathcal{J}_\kappa u)$ lie in $U$, provided $U$ is chosen large enough, part vii). The parameter penalty $\nu\Vert\theta\Vert$, finally, is needed for well-posedness at finite $m$ and vanishes as $m\to\infty$.

\paragraph{Choice of the exponent.} Let us finally comment on the exponent $\infty$ in the penalty $\Vert\nabla f_\theta\Vert_{L^\infty(U)}$, since $W^{1,\infty}$ is not a norm naturally associated with the PDE class under consideration. The reason it appears is that the regularization does not act on a state of the equation, but on the unknown nonlinearity, which enters the problem only through the composition $f(t,\mathcal{J}_\kappa u)$ with the state. What the analysis requires of a penalty for $f$ is accordingly not coercivity with respect to the equation, but that its bounded sublevel sets be relatively compact in $\mathcal{C}(\overline{U})^N$ and that the limit be again Lipschitz continuous, together with weak-$*$ lower semicontinuity. These are the properties that keep the above composition stable along a sequence whose states converge weakly only, as in the two uses described above. The $W^{1,\infty}$-type penalty provides them, but it is not the only one that does. Indeed, on the bounded Lipschitz domain $U\subseteq\mathbb{R}^D$ one has $W^{1,p}(U)\hookdoubleheadrightarrow \mathcal{C}(\overline{U})$ already for $p>D$, so that a $W^{1,p}$-type penalty with $p>D$ can be expected to serve the same purpose, at the price of an exponent depending on the dimension and of a weaker regularity of the limit. We do not carry out this variant. Our reasons for working with $p=\infty$ are that the resulting bound is dimension free, that the Lipschitz constant is the quantity which appears in the estimate controlling the composition, and that both Assumption \ref{ass_uniqueness}, iii) and the approximation results verifying Assumption \ref{ass_uniqueness}, iv) are available in $W^{1,\infty}$. What is not a matter of choice, on the other hand, is the presence of a first-order term as such. The example in Appendix \ref{app:analytic_example} presents problems for which the solutions regularized by an $L^\rho$-term alone converge, but not to the solution of the limit problem, and the gradients of the minimizers computed there grow unboundedly. The failure therefore occurs already at the level of the minimizers and cannot be repaired by a different proof technique.

\medskip

\noindent For orientation, we summarize in Table \ref{tab:ass_roles} where each assumption enters the analysis. Classes of functions satisfying these assumptions are presented in Subsections \ref{subsec:neuralnet} and \ref{subsec:physicalterm}. In particular, Assumption \ref{ass_param_app_class} together with ii) and iv) in Assumption \ref{ass_uniqueness} is verified there in case $\mathcal{F}_n^m$ is chosen as a suitable class of neural networks and $f$ in Assumption \ref{ass_uniqueness}, iv) has a suitable regularity, and Assumption \ref{ass_phys_term} is verified for affine linear and for nonlinear physical terms.

\medskip

\begin{table}[htbp]
	\small
	\centering
	\begin{tabular}{p{2cm} p{7cm} p{4.3cm}}
		\toprule
		\textbf{Assumption} & \textbf{Role in the analysis} & \textbf{Used in}\\
		\midrule
		\ref{ass_init_set} i), iii) & weak compactness of minimizing sequences & Appendix \ref{app:existence}\\
		\ref{ass_init_set} ii) & strong convergence of the nonlinearity arguments (\emph{Aubin-Lions}) & Lemma \ref{lem:strong_weak_cont}, Theorem \ref{thm:uniqueness}\\
		\ref{ass_init_set} vi) & coercivity/ weak lower semicontinuity for the direct method & Appendix \ref{app:existence}\\
		\ref{ass_param_app_class}, \ref{ass_phys_term} & PDE constraint stable under weak limits & Appendix \ref{app:existence}\\
		\ref{ass_uniqueness} i), ii) & gradient control for the model comparison & Theorem \ref{thm:uniqueness}\\
		\ref{ass_uniqueness} iv) & reachability of the limit model & Theorem \ref{thm:uniqueness}\\
		\ref{ass_uniqueness} v) & state uniquely determined by the data & Prop.\ \ref{prop:p_dagger}, Theorem \ref{thm:uniqueness}\\
		\ref{ass_uniqueness} vi), viii) & strict convexity/ identifiability limit solution & Prop.\ \ref{prop:p_dagger}\\
		\bottomrule
	\end{tabular}
	\caption{Role of the assumptions for well-posedness and unique reconstructability.}
	\label{tab:ass_roles}
\end{table}
\normalsize The following remarks discuss some aspects of the above assumptions.
\begin{remark}[Compact embedding of state space]
	\label{rem:choice_V}
	
	A possible choice of the space $V$ satisfying the compact embedding in Assumption \ref{ass_init_set} is $V=W^{\kappa+\tilde \kappa ,p_0}(\Omega)$ for $1<p_0<\infty, \tilde \kappa \in\mathbb{N}$ fulfilling either $\tilde \kappa p_0<d$ with $1\leq \hat{p}<\frac{dp_0}{d-\tilde \kappa p_0}$ or $\tilde \kappa p_0=d$ with $1\leq \hat{p}<\infty$ due to the Rellich-Kondrachov Theorem (see e.g. \cite[Theorem 6.3]{Adams2003} and \cite[§5.7]{Evans2010}). The spaces $V_k$ can be chosen as $V_k=L^{\hat{p}}(\Omega)$ for $1\leq k\leq \kappa$.
\end{remark}
\begin{remark}[Role of operator $\mathcal{J}_\kappa $]
	\label{rem:jkappa} As the nonlinearities $f_{\theta_n,n}$ operate pointwise in space and time, the operator $\mathcal{J}_\kappa $ is needed to allow for a dependence of $f_{\theta_n,n}$ also on derivatives of the state. For the physical term $F$ on the other hand, an explicit incorporation of derivatives is not necessary, as $F$ does not act pointwise in space but rather directly on $V $. As a consequence, the known physical term can in general not be absorbed into the learned one, e.g. when $F$ acts nonlocally or depends explicitly on the spatial variable through $\varphi$, as for instance for the transport term in Proposition \ref{prop:demo}. In that case the sum $F+f$ is not a Nemytskii operator of a function on $(0,T)\times(\otimes_{k=0}^\kappa \mathbb{R}^{p_k})^N$.
\end{remark}
\begin{remark}[Regularity condition extended state space]
	\label{rem:regul_cond} The regularity condition in Assumption \ref{ass_uniqueness}, i) ensures that a weakly convergent sequence in the extended state space attains uniformly bounded higher-order derivatives. This continuous embedding can be achieved by imposing additional regularity on the state space $V$ and thus, on its temporal extension $\mathcal{V}$. Indeed, as $\mathcal{V}=W^{1,p,p}(0,T;V,\tilde{V})$ by \cite[Lemma 7.1]{Roubíček2013} using $V\hookrightarrow\tilde{V}$ it follows that $\mathcal{V}\hookrightarrow \mathcal{C}(0,T;\tilde{V})$. If $\tilde{V}$ is sufficiently regular, e.g., fulfills some embedding of the form $\tilde{V}\hookrightarrow W^{\kappa+\tilde{\kappa}, \eta}(\Omega)$ with $\tilde{\kappa}\eta > d = \dim(\Omega)$, then $\mathcal{C}(0,T; W^{\tilde{\kappa},\eta}(\Omega))\hookrightarrow L^\infty((0,T)\times \Omega)$. Combining these embeddings, together with $D^\beta v(t)\in W^{\tilde{\kappa},\eta}(\Omega)$ for $v\in \mathcal{V}$ and $t\in(0,T)$ yields Assumption \ref{ass_uniqueness}, i).
\end{remark}
\begin{remark}[Convergence of measurement operators]
	\label{rem:conv_meas} Note that the required convergence in \eqref{operator_conv_weak} is rather weak in practice. In fact it holds for (potentially nonlinear) operators $(K^m)_m$ converging to $K^\dagger$ uniformly on bounded sets in $\mathcal{V}$ (since weakly convergent sequences are bounded by \cite[Proposition 3.5 (iii)]{Brezis2010}). This, in particular, encompasses bounded linear operators converging in the operator norm.
\end{remark}
\begin{remark}[Regularity of admissible function]
	\label{remark:regularity_admiss} The assumption $\hat{f}\in W^{1,\infty}(\mathbb{R}^D)^N$ in Assumption \ref{ass_uniqueness}, iii), seems to be restrictive. However, since an admissible state $\hat{u}$ attains uniformly bounded $\mathcal{J}_\kappa \hat{u}$ by Assumption \ref{ass_uniqueness}, i), the term $\hat{f}$ only acts on a compact subset of $\mathbb{R}^D$ in \eqref{rdsystem} due to its composition to $\mathcal{J}_\kappa \hat{u}$. Thus, whenever a $\hat{f}\in W^{1,\infty}_{loc}(\mathbb{R}^D)^N$ solving \eqref{rdsystem} exists, there exists w.l.o.g. a solution which is globally $W^{1,\infty}(\mathbb{R}^D)^N$-regular as a consequence of the following extension argument. For a sufficiently large and regular subset $U\subset \mathbb{R}^D$ meeting the regularity conditions in the references below and containing $(t,\mathcal{J}_\kappa \hat{u}(t,x))$ for a.e. $(t,x)\in (0,T)\times \Omega$ define $\hat{f}_0:\mathbb{R}^D\to \mathbb{R}^N$ with $\hat{f}_0=\hat{f}$ on $U$. The function $\hat{f}_0\in W^{1,\infty}(U)^N$ is then extendable to some $\hat{f}_0\in W^{1,\infty}(\mathbb{R}^D)^N$ due to regularity of $U$. The result in \cite[Chapter VI, Theorem 5]{stein} treats this in a more general framework that includes general Sobolev spaces and minimal smoothness conditions on the domain $U$. We refer to \cite[Theorem 5.24]{Adams2003} for an outline of the proof. A proof of the extension result, but for first-order Sobolev spaces and stricter smoothness assumptions on the underlying domain, can be found in \cite[Theorem 9.7]{Brezis2010}.
\end{remark}
\begin{remark}[Regularity of admissible state]
	\label{rem:admissible_state} The existence of an admissible solution to \eqref{rdsystem} as required in Assumption \ref{ass_uniqueness}, iii), with state regularity $\mathcal{V}$ may in practice be difficult to guarantee in view of the regularity typically expected from the underlying equation. For a concrete example we refer to the discussion of the transport equation following Proposition \ref{prop:demo} in the introduction. Nevertheless, it is important to note that this regularity requirement can be interpreted as an implicit assumption on the parameter space $X_\varphi$, as briefly highlighted for the transport equation, since the regularity of the state is generally inherited from the regularity of the model and that of the input parameters.
\end{remark}
\begin{remark}[Choice of $U$]
	In view of Assumption \ref{ass_uniqueness}, vii) one can circumvent choosing a sufficiently large radius $\delta$ depending on all admissible functions as in Assumption \ref{ass_uniqueness}, iii) as follows. For a closed and convex set $U\subset \mathbb{R}^D$ containing $(t, \mathcal{J}_\kappa \hat{u}(t,x))$ for a.e. $(t,x)\in (0,T)\times \Omega$ one can define $P_U$ to be the metric projector onto $U$. Now considering the term $f_\theta(P_U(t,\mathcal{J}_\kappa u^l))$ in \eqref{min_prob} instead of $f_\theta(t,\mathcal{J}_\kappa u^l)$ the results of this work still apply. Of course well-definedness of $f_\theta\circ P_U$ in function space has to be argued first. From a model perspective the composition $f_\theta\circ P_U$ can be interpreted as part of the parameterized approximation classes \eqref{param_approx_classes}. For the specific case of neural networks this generalizes to applying a sufficiently regular sigmoidal-type function as activation function in the first layer.
\end{remark}
\begin{remark}[A priori bounded states]
	\label{rem:apriori_bounded} It is possible to circumvent both the assumption $\hat{f}\in W^{1, \infty}(\mathbb{R}^D)^N$ and the regularity condition in Assumption \ref{ass_uniqueness}, i), if it is a priori known that the $\mathcal{J}_\kappa u$ are uniformly bounded. In case $\kappa =0$, this is the typical situation for systems whose solutions obey a maximum principle or remain in an invariant region, as for reaction-diffusion models of chemical concentrations, where the state $u$ is a priori bounded in the interval $[0,1]$, and for which global bounds and global classical solutions are available under control of mass and growth conditions on the kinetics \cite{Pierre2010,Fellner2020}. For maximum principles we refer to \cite[Subsection 7.1.4]{Evans2010}.
\end{remark}

\begin{remark}[Boundary trace map]
	\label{rem:bdry} In view of Assumption \ref{ass_init_set}, i) if $V\hookrightarrow W^{\kappa+1, \hat{p}}(\Omega)$, a possible choice of the trace map $\gamma: \mathcal{V}\to \mathcal{B}$ is the (pointwise in time) Dirichlet trace operator $\gamma_0:V \rightarrow B$ (see \cite[Chapter 5]{Adams2003}) with $B=L^b(\partial \Omega)$ for $b$ as follows. Following \cite[Theorem 5.36]{Adams2003} for instance, $\gamma_0 : W^{\kappa,\hat{p}}(\Omega)\to L^b(\partial\Omega)$ (and hence $\gamma$) is weak-weak continuous if $\kappa \hat{p}\leq d$ and $\hat{p}\leq b\leq \frac{(d-1)\hat{p}}{d-\kappa\hat{p}}$ (with $\hat{p}\leq b<\infty$ if $\kappa \hat{p} = d$). The choice of the (pointwise in time) Neumann trace operator (see \cite[Chapter 2]{Necas2011})) may be treated similarly with the same conditions on $b$. The discrepancy functional $\mathcal{D}_\BC$ can for instance be given as the indicator functional by $\mathcal{D}_\BC(w) = 0$ if $w=0$ and $\mathcal{D}_\BC(w) = \infty $ else, acting as a hard constraint, or as soft constraint via $\mathcal{D}_\BC(w)=\sum_{n}\Vert w_n\Vert_\mathcal{B}^s$ for $w\in \mathcal{B}^{N}$. In both cases $\mathcal{D}_\BC$ is weakly lower semicontinuous, coercive and fulfills $\mathcal{D}_\BC(z)=0$ iff $z=0$.
\end{remark}
\begin{remark}[Known initial and boundary conditions]
	\label{rem:known_ic_bc} In \eqref{rdsystem} and \eqref{min_prob} the initial conditions $u_0$ and the boundary conditions $g$ are considered unknown. This reflects the all-at-once viewpoint taken in this work that these data are themselves accessible only through incomplete and noisy measurements.
	
	The situation that these data are known is contained in our setting. In fact, prescribing $u_0=u_0^\circ$ and $g=g^\circ$ for given $u_0^\circ\in H^{N\times L}$, $g^\circ\in\mathcal{B}^{N\times L}$ amounts to adding the indicator functional of $\{(u_0^\circ,g^\circ)\}$ to the regularization functional $\mathcal{R}_0$, in complete analogy to the hard-constraint choice of $\mathcal{D}_\BC$ in Remark~\ref{rem:bdry}. This is admissible and the assumptions in Subsection \ref{subsec:assumptions} remain valid, such that the main results, Proposition \ref{prop:p_dagger} and Theorem \ref{thm:uniqueness}, hold verbatim, with $(u_0^m)_m$ and $(g^m)_m$ being constant sequences, and the other convergence assertions remaining unchanged. Note that the discrepancy terms in \eqref{min_prob} are not dropped in this case, but act as fidelity terms coupling the state to the known data.
	
\end{remark}

\subsection{Neural networks}
\label{subsec:nn}
\label{subsec:neuralnet} The present subsection and Subsection \ref{subsec:physicalterm} show that the above assumptions are met by concrete classes of approximating functions and physical terms. They are independent of the main results in Section~\ref{sec:uniqueness} and may be skipped on a first reading. In this section we discuss Assumption \ref{ass_param_app_class} together with ii) of Assumption \ref{ass_uniqueness} in case $(\mathcal{F}_n^m)_n$ are chosen as suitable classes of feed-forward neural networks. Furthermore, we provide results from literature that ensure Assumption \ref{ass_uniqueness}, iv) for specific network architectures and suitably regular $f$. Moreover, we address also Assumption \ref{ass_uniqueness}, vi).

\begin{definition}
	\label{def:nn} Let $L\in\mathbb{N}$, $(n_l)_{0\leq l\leq L}\subseteq \mathbb{N}$, $\sigma \in \mathcal{C}(\mathbb{R},\mathbb{R})$ and $\theta_l=(w^l, \beta^l)$ with $w^l\in \mathcal{L}(\mathbb{R}^{n_{l-1}},\mathbb{R}^{n_l})\simeq\mathbb{R}^{n_l\times n_{l-1}}$ and $\beta^l\in \mathbb{R}^{n_l}$ for $1\leq l\leq L$. Furthermore, let $L_{\theta_l}:\mathbb{R}^{n_{l-1}}\to\mathbb{R}^{n_l}$ via $L_{\theta_l}(z):=\sigma(w^lz+\beta^l)$ for $1\leq l\leq L-1$ together with $L_{\theta_L}(z):= w^Lz+\beta^L$. Then a fully connected feed-forward neural network $\mathcal{N}_\theta$ with activation function $\sigma$ is defined as $\mathcal{N}_\theta=L_{\theta_L}\circ\dots\circ L_{\theta_1}$. The input dimension of $\mathcal{N}_\theta$ is $n_0$ and the output dimension $n_L$. Moreover, we define the width of the network by $\mathcal{W}(\mathcal{N})=\max_l n_l$ and the depth by $\mathcal{D}(\mathcal{N})=L$.
\end{definition}
Note that $l$ and $L$ refer here to the layers of the network and are unrelated to the index $l$ of the individual measurement setups used elsewhere.
\begin{definition}[Model for $(\mathcal{F}_n^m)_n$]
	\label{def:model_fi} Let $\sigma:\mathbb{R}\to \mathbb{R}$ be locally Lipschitz continuous. Then we define for $L,(n_l)_l$ depending on $m\in \mathbb{N}$ and $\Theta_n^m\subseteq \otimes_{l=1}^L \mathbb{R}^{n_l\times n_{l-1}}\times\mathbb{R}^{n_l}$ for $1\leq n\leq N$ with $n_0=1+N\sum_{k=0}^\kappa p_k$ and $n_L=1$ the class of parameterized approximation functions of the unknown terms, $$\mathcal{F}_n^m = \left\{\mathcal{N}_\theta ~|~ \theta\in \Theta_n^m\right\},$$ for $n=1,\dots, N$ where each $\mathcal{N}_\theta: (0,T)\times(\otimes_{k=0}^\kappa \mathbb{R}^{p_k})^N\to \mathbb{R}$ is a fully connected feed-forward neural network with activation function $\sigma$.
\end{definition}
\begin{remark} \label{rem:activation_examples}
	Commonly used activation functions which are globally Lipschitz continuous include the softplus, saturated activation functions such as the sigmoid, hyperbolic tangent and Gaussian but also ReLU and some of its variations like the leaky ReLU and exponential linear unit amongst others. An example of a locally Lipschitz continuous activation function that is not globally Lipschitz continuous is the Rectified Quadratic Unit (ReQU).
\end{remark}

Now as first step, we focus on the induction of well-defined Nemytskii operators and strong-weak continuity as specified in Assumption \ref{ass_param_app_class}. Following \cite[Lemma 4, Lemma 5]{AHN23}, the former can be shown for general, continuous activation functions and the latter for locally Lipschitz continuous activation functions, both under the additional regularity assumption Assumption \ref{ass_uniqueness}, i). Here, we focus on a different strategy that does not require Assumption \ref{ass_uniqueness}, i), but assumes a globally Lipschitz continuous activation function. Note that in this section we write generically $\Theta$ instead of $\Theta_n^m$, as the results below on neural networks hold for general parameter sets as in Definition \ref{def:nn}. The following result, whose proof can be found in Appendix \ref{app:neural_networks}, shows that for $(\mathcal{F}_n^m)_n$ as in Definition \ref{def:model_fi} and Lipschitz continuous $\sigma$ the properties in Assumption \ref{ass_param_app_class} follow.
\begin{proposition}
		\label{prop:ass3_nn} Let Assumption \ref{ass_init_set} hold true. Suppose that $\sigma\in \mathcal{C}(\mathbb{R},\mathbb{R})$ is Lipschitz continuous with constant $L_\sigma$ (w.l.o.g. $L_\sigma\geq1$). Then $\mathcal{N}_\theta: (0,T)\times(\otimes_{k=0}^\kappa \mathbb{R}^{p_k})^N\to \mathbb{R}$ induces a well-defined Nemytskii operator $\mathcal{N}_\theta: (\otimes_{k=0}^\kappa \mathcal{V}_k^\times)^N\to L^p(0,T;L^{\hat{q}}(\Omega))$ via $[\mathcal{N}_\theta( u)](t)=\mathcal{N}_\theta(u(t,\cdot))$. The same applies to $\mathcal{N}_\theta:(\otimes_{k=0}^\kappa \mathcal{V}_k^\times)^N\to\mathcal{W}$. Furthermore, $$\mathcal{N}:\Theta\times (\otimes_{k=0}^\kappa L^p(0,T;L^{\hat{p}}(\Omega)^{p_k}))^N\to L^q(0,T;L^{\hat{q}}(\Omega)),\quad (\theta, v)\mapsto\mathcal{N}_\theta(v)$$ is strongly-strongly continuous.
\end{proposition}
\begin{proof}
	See Appendix \ref{app_subsec:nn_prop1}.
\end{proof}
Assuming a proper choice of the regularization functional $\mathcal{R}_0$, an important question is whether regularizing via
\begin{align}
	\label{regularization_functional} \mathcal{R}(\varphi, u,\theta,u_0,g)=\mathcal{R}_0(\varphi, u,u_0,g)+ \nu \|\theta\| + \Vert f_\theta\Vert_{L^\rho(U)}^\rho+\Vert\nabla f_\theta\Vert_{L^{\infty}(U)}
\end{align}
is justified for the class of parameterized approximation functions as introduced in Definition \ref{def:model_fi} based on feed-forward neural networks. This includes both $W^{1,\infty}_{loc}$-regularity of the classes $\mathcal{F}_n^m$ and weak lower semicontinuity of \eqref{regularization_functional} as required by Assumption \ref{ass_init_set}. For the latter, in turn, it suffices to verify for fixed $n=1, \dots, N$ weak lower semicontinuity of the map
\[ \Theta\ni\theta\mapsto \Vert \mathcal{N}_\theta\Vert_{L^\rho(U)}+\Vert \nabla \mathcal{N}_\theta\Vert_{L^\infty(U)},
\] again for a generic parameter set $\Theta$ in Definition \ref{def:nn}. By weak lower semicontinuity of the $L^\rho-$norm and strong-strong continuity of $\Theta\ni \theta\mapsto \mathcal{N}_\theta \in L^{\infty}(U)$ (as follows from \eqref{Ndiffestimation} in the proof of Proposition \ref{prop:ass3_nn}), for this, it remains to argue weak lower semicontinuity of
\begin{align}
	\label{map_for_lsc} \Theta\ni \theta\mapsto \Vert \nabla \mathcal{N}_\theta\Vert_{L^\infty(U)}.
\end{align}
The next result, which is also proved in Appendix \ref{app:neural_networks}, shows that ii) and vi) in Assumption \ref{ass_uniqueness} in fact hold true in this particular framework. In view of weak lower semicontinuity of \eqref{map_for_lsc} we restrict ourselves to locally Lipschitz continuous activation functions, which covers both $\mathcal{C}^1$-regular ones and the Rectified Linear Unit (ReLU).
\begin{proposition}
	\label{prop:ass5_nn} Assume that $\sigma\in \mathcal{C}(\mathbb{R},\mathbb{R})$ is locally Lipschitz continuous and let $(\mathcal{F}_n^m)_n$ be given as in Definition \ref{def:model_fi}. Then for $1\leq n\leq N$, $m\in \mathbb{N}$ it holds true that $$\mathcal{F}_n^m\subseteq W^{1,\infty}_\loc(\mathbb{R}^{D}).$$ Let $U\subseteq \mathbb{R}^D$ be bounded and open. Then for $(\theta^m)_m\subseteq \Theta$ with $\theta^m\to \theta\in \Theta$ as $m\to \infty$ it holds
	\[ \Vert\nabla\mathcal{N}_\theta\Vert_{L^\infty(U)}\leq \liminf_{m\to\infty}\Vert\nabla\mathcal{N}_{\theta^m}\Vert_{L^\infty(U)}.
	\]
\end{proposition}
\begin{proof}
	See Appendix \ref{app_subsec:nn_prop2}.
\end{proof}

Recall that a similar result as in Proposition \ref{prop:ass3_nn} holds true for locally Lipschitz continuous activations (e.g. the ReQU, which activates the networks in Propositions \ref{prop:belo} and \ref{prop:raslan} below) under the additional regularity Assumption~\ref{ass_uniqueness}, i) as discussed above directly after Remark \ref{rem:activation_examples}. Proposition \ref{prop:ass5_nn}, in turn, already covers such activations without any further assumption. 

We conclude the considerations on neural networks by discussing results from literature ensuring that  Assumption \ref{ass_uniqueness}, iv) holds true for suitably regular $f$. The corresponding estimate in \eqref{f_approximation1} is closely related to universal approximation theory for neural networks, an active field of research which is presented e.g. in \cite{devore21, elbraecther21, gribonval20} and the references therein. Determining suitable functions $\psi$ regarding \eqref{f_approximation1} for these approximation results is, however, not usually considered in works on neural network approximation theory and is in general not trivial. For an outline of state-of-the-art results dealing with suitable estimates on $\psi$ we refer to the comparative overview presented in \cite{morina_holler/online}. The result in \cite{morina_holler/online} shows that a slight modification of the nearly optimal uniform approximation result of piecewise smooth functions by ReLU networks in \cite{lu_main} grows polynomially and in general yields a better bound than the other results providing polynomial bounds except for \cite{belomestny23} which uses the ReQU activation function. As discussed in \cite{morina_holler/online}, the following (simplified) results hold true.
\begin{proposition}
	\label{prop:morina_holler} Let the parameterized classes in \eqref{param_approx_classes} be given by neural networks of the form in \cite[Theorem 3]{morina_holler/online} and $f\in \mathcal{C}^q(U)$ for some $q\geq 1$. Then \eqref{f_approximation1} in Assumption \ref{ass_uniqueness}, iv) holds true with $\beta=2q/D$ (with networks of constant depth and width of order $m\log m$) and $\psi(m) = \tilde{c} m^{\frac{6q-3}{D}}$ for some $\tilde{c}>0$.
\end{proposition}
\begin{proposition}
	\label{prop:belo} Let the parameterized classes in \eqref{param_approx_classes} be given by neural networks of the form in \cite[Theorem 1]{belomestny23} and $f\in \mathcal{C}^q(U)$ for some $q\in\mathbb{N}$ with $q\geq 3$. Then \eqref{f_approximation1} in Assumption \ref{ass_uniqueness}, iv) holds true with $\beta=q/D$ (with networks of constant depth and width of order $m$) and $\psi(m) = \tilde{c}$ for some $\tilde{c}>0$. In addition, for the same parameters $(\theta^m)_m$ realizing \eqref{f_approximation1} it holds that $\Vert \nabla f_{\theta^m}\Vert_{L^\infty(U)}\to\Vert\nabla f\Vert_{L^\infty(U)}$ as $m\to \infty$, such that Assumption \ref{ass_uniqueness}, iv) is satisfied in full.
\end{proposition}
The restriction to $q\geq 3$ and the fact that Assumption \ref{ass_uniqueness}, iv) is obtained in full have the following reason. The result in \cite[Theorem 1]{belomestny23} requires a Hölder smoothness index larger than $2$ and then approximates $f$ together with all its derivatives up to order $q-1$ simultaneously. As a consequence, for $f\in \mathcal{C}^q(U)$ with an integer $q\geq 3$, beyond the uniform estimate in \eqref{f_approximation1}, one derives first-order derivative approximation $\Vert \nabla f-\nabla f_{\theta^m}\Vert_{L^\infty(U)}=\mathcal{O}(m^{-(q-1)/D})$ and hence, by the reverse triangle inequality, the convergence $\Vert \nabla f_{\theta^m}\Vert_{L^\infty(U)}\to\Vert\nabla f\Vert_{L^\infty(U)}$ demanded in Assumption \ref{ass_uniqueness}, iv).
The regularity requirement in Proposition \ref{prop:belo} can be weakened by relying on a different approximation result. The work \cite{raslan21} provides approximations in Sobolev norms together with a bound on the weights, for activation functions admitting suitable partitions of unity, and yields the following.
\begin{proposition}
	\label{prop:raslan} Let the parameterized classes in \eqref{param_approx_classes} be given by neural networks of the form in \cite[Proposition 4.8]{raslan21} with the ReQU activation function, and let $f\in W^{n,\infty}(U)$ for some $n\in\mathbb{N}$ with $n\geq 2$. Then \eqref{f_approximation1} in Assumption \ref{ass_uniqueness}, iv) holds true with $\beta=(n-1)/D$ and $\psi(m)=\tilde{c}\,m^{\theta(n-1)/D}$ for some $\tilde{c},\theta>0$. In addition, for the same parameters $(\theta^m)_m$ realizing \eqref{f_approximation1} it holds that $\Vert \nabla f_{\theta^m}\Vert_{L^\infty(U)}\to\Vert\nabla f\Vert_{L^\infty(U)}$ as $m\to \infty$, such that Assumption \ref{ass_uniqueness}, iv) is satisfied in full.
\end{proposition}
Here, \cite[Proposition 4.8]{raslan21} provides, for $f$ in the unit ball of $W^{n,\infty}((0,1)^D)$, neural networks with $\mathcal{O}(\varepsilon^{-D/(n-1)})$ non-zero weights, bounded in absolute value by $c\varepsilon^{-\theta}$, and with approximation error at most $\varepsilon$ in the $W^{1,\infty}$-norm, the ReQU activation function generating exact partitions of unity in the sense required there. In particular, Assumption \ref{ass_uniqueness}, iv) is obtained already for $f\in W^{2,\infty}(U)$. For the networks in Proposition \ref{prop:morina_holler}, on the other hand, the convergence $\Vert \nabla f_{\theta^m}\Vert_{L^\infty(U)}\to \Vert \nabla f\Vert_{L^\infty(U)}$ as $m\to \infty$ does not follow from \cite[Theorem~3]{morina_holler/online}, which provides a uniform estimate only. Approximation results that do control derivatives are available in the literature. These include \cite{yang25} which considers approximation in $W^{m,p}$-Sobolev norms for integer $m\geq 2$, and \cite{guehring20} which covers $W^{1,\infty}$-approximation by ReLU neural networks, both being stronger than the convergence required here. A parameter estimation as stated in \eqref{f_approximation1}, however, is not provided in these works. Alternatively, for the result in \cite[Theorem 3]{morina_holler/online}, one might apply a lifting technique as outlined in Appendix \ref{app_subsec:lifting}, which is possible in case $f$ attains higher regularity such as $W^{2,\infty}$- or $\mathcal{C}^2$-regularity.
\subsection{Physical term}
\label{subsec:physicalterm} As announced in Subsection \ref{subsec:neuralnet}, this subsection likewise serves to verify the assumptions for concrete classes and is not needed for Section \ref{sec:uniqueness}. In the next subsections we verify Assumption \ref{ass_phys_term} in the setup of affine linear physical terms and in the general setup of nonlinear physical terms, and provide examples.
\subsubsection{Linear case} We assume that the physical term is given in linear form for some fixed $\omega\in \mathbb{N}_0$ by
\begin{align}
	\label{phys_term_linear}
	F(t,(u_n)_{1\leq n\leq N}, \varphi)&= \Psi(t,\varphi)+\sum_{n=1}^N\mathcal{J}_\omega u_n\cdot \Phi_n(t,\varphi)
\end{align}
\[ \text{with} ~ ~ ~ \mathcal{J}_\omega u_n\cdot \Phi_n(t,\varphi) :=\sum_{0\leq \vert\beta\vert\leq \omega} D^\beta u_n\cdot\Phi_{n,\beta}(t,\varphi)
\] for $t\in(0,T), (u_n)_{1\leq n\leq N}\in V^N, \varphi\in X_\varphi$, where we suppose that $(V\hookrightarrow)\tilde{V}\hookrightarrow W^{\omega,\hat{p}}(\Omega)$. The functionals $\Psi$ and $(\Phi_{n,\beta})_{n,\beta}$ are given as $\Psi:(0,T)\times X_\varphi \to L^{\hat{q}}(\Omega)$ and $\Phi_{n,\beta}:(0,T)\times X_\varphi \to L^{s_\beta}(\Omega)$ for $1\leq n\leq N$, $0\leq \vert\beta\vert\leq \omega$ and some suitable $1\leq s_{\beta}\leq \infty$ (to be determined below). Since $\Psi(t,\varphi) \in W$ due to  $L^{\hat{q}}(\Omega)\hookrightarrow W$, in order to show that $F(t,(u_n)_{1\leq n\leq N}, \varphi)\in W$ (i.e., that $F$ is well-defined) it suffices to choose the $s_\beta$ such that $ \mathcal{J}_\omega u_n\cdot \Phi_n(t,\varphi)\in W$. This can be done as follows. For $(u_n)_{1\leq n\leq N}\in V^N$ we have that $D^\beta u_n\in W^{\omega-\vert\beta\vert,\hat{p}}(\Omega)$ for $0\leq \vert \beta\vert\leq \omega$ and $D^\beta u_n\cdot \Phi_{n,\beta}(t,\varphi)\in L^{\hat{q}}(\Omega)\hookrightarrow W$, which shows welldefinedness of \eqref{phys_term_linear}, if one of the following conditions on $s_\beta$ is fulfilled:
\begin{align}
	\label{conditions_sbeta}
	\begin{cases}
		\sbullet \quad \frac{\omega-\vert\beta\vert}{d}>\frac{1}{\hat{p}}-\frac{1}{\hat{q}}+\frac{1}{s_\beta}\quad \text{and} \quad \hat{q}\leq s_\beta\leq \frac{\hat{p}\hat{q}}{\hat{p}-\hat{q}}, ~ s_\beta <\infty\\
		\sbullet \quad \frac{\omega-\vert\beta\vert}{d}=\frac{1}{\hat{p}}-\frac{1}{\hat{q}}+\frac{1}{s_\beta}\quad \text{and} \quad \hat{q}< s_\beta\leq \frac{\hat{p}\hat{q}}{\hat{p}-\hat{q}}, ~ s_\beta <\infty\\
		\sbullet \quad s_\beta = \frac{\hat{p}\hat{q}}{\hat{p}-\hat{q}}
	\end{cases}
\end{align}
In the first two cases welldefinedness of \eqref{phys_term_linear} is a consequence of \cite[Theorem 6.1]{behzadan2021} (see also \cite[Remark 6.2, Corollary 6.3]{behzadan2021} for the generalization to bounded Lipschitz domains). In the last case (with $\frac{\hat{p}\hat{q}}{\hat{p}-\hat{q}}=\infty$ for $\hat{p}=\hat{q}$ which is important as $s_\beta =\infty$ is excluded in the first two cases) welldefinedness follows by $V\hookrightarrow W^{\omega, \hat{p}}(\Omega)$ and Hölder's inequality. To achieve Assumption \ref{ass_phys_term}, ii) we need stricter conditions than \eqref{conditions_sbeta} in general as outlined below. The following result, proven in Appendix \ref{app:physical_term}, covers Assumption \ref{ass_phys_term} in the linear setup.
\begin{proposition}
	\label{prop:ass4_linear} Let Assumption \ref{ass_init_set} hold true and $\tilde{V}\hookrightarrow W^{\omega,\hat{p}}(\Omega)$. Suppose that $t\mapsto \Phi_n(t,\varphi)$ and $t\mapsto \Psi(t,\varphi)$ are measurable for all $\varphi\in X_\varphi$ and $s_\beta$ fulfill \eqref{conditions_sbeta}. Assume that there exist functions $\mathcal{B}_1, \mathcal{B}_2:\mathbb{R}_{\geq 0}\to \mathbb{R}_{\geq 0}$ that map bounded sets to bounded sets and $\phi\in L^{\frac{pq}{p-q}}(0,T)$ (with $\phi\in L^{\infty}(0,T)$ if $p=q$), $\psi\in L^q(0,T)$ such that
	\begin{align}
		\label{growth_cond_lin_terms1} \Vert \Phi_{n,\beta}(t,\varphi)\Vert_{L^{s_\beta}(\Omega)}\leq \phi(t)\mathcal{B}_1(\Vert\varphi\Vert_{X_\varphi}), ~ ~ ~ 	\Vert \Psi(t,\varphi)\Vert_{L^{\hat{q}}(\Omega)}\leq \psi(t)\mathcal{B}_2(\Vert\varphi\Vert_{X_\varphi}).
	\end{align}
	Then $F$ in \eqref{phys_term_linear} induces a well-defined Nemytskii operator $F: \mathcal{V}^N\times X_\varphi\to \mathcal{W}$ with $$[F((u_n)_{1\leq n\leq N}, \varphi)](t) = F(t,(u_n(t))_{1\leq n\leq N}, \varphi)$$ for $(u_n)_{1\leq n\leq N}\in \mathcal{V}^N, \varphi\in X_\varphi$ and $t\in (0,T)$. Now suppose that $\Psi(t,\cdot):X_\varphi\to L^{\hat{q}}(\Omega)$ and $\Phi_{n,\beta}(t,\cdot):X_\varphi\to L^{s_\beta}(\Omega)$ are weakly continuous for almost every $t\in (0,T)$, additionally with $\frac{\omega-\vert\beta\vert}{d}>\frac{1}{\hat{p}}-\frac{1}{\hat{q}}+\frac{1}{s_\beta}$ if $\hat{q}=1$ or $s_\beta = \frac{\hat{p}\hat{q}}{\hat{p}-\hat{q}}$. Furthermore, suppose that either $\omega\leq \kappa$ or otherwise in case $\omega>\kappa$ the following additional conditions hold: 
	\begin{itemize}
		\item ~For each $0\leq \vert\beta\vert<\omega$ assume that there exists some $\hat{q}\leq c_\beta\leq \infty$ such that $W^{\omega-\vert\beta\vert,\hat{p}}(\Omega)\hookdoubleheadrightarrow L^{c_\beta}(\Omega)$ and that we have the additional growth condition
		\[ \Vert \Phi_{n,\beta}(t,\varphi)\Vert_{\frac{c_\beta\hat{q}}{c_\beta-\hat{q}}}\leq \phi(t)\mathcal{B}_1(\Vert \varphi\Vert_{X_\varphi}).
		\]
		\item ~For $\vert\beta\vert=\omega$ assume that $\Phi_{n,\beta}(t,\cdot): X_\varphi\to L^{\frac{\hat{p}\hat{q}}{\hat{p}-\hat{q}}}(\Omega)$ is well-defined and weak-strong continuous for a.e. $t\in (0,T)$.
	\end{itemize}
	Then $\mathcal{V}^N\times X_\varphi\ni(u,\varphi)\mapsto F(u,\varphi)\in \mathcal{W}$ induced by \eqref{phys_term_linear} is weak-weak continuous.
\end{proposition}
\begin{proof}
	See Appendix \ref{app:linear}.
\end{proof}
To conclude this subsection we give the following example which is motivated by the parabolic problem considered in \cite[Chapter 4]{AHN23}. We restrict ourselves to a single equation which can be immediately generalized to general systems by introducing technical notation. Note that the space setup in the following example is consistent with Assumption \ref{ass_init_set}, but we do not discuss it in order to not distract from the central conditions on the parameters. In the example we suppose that the physical term is governed by a reaction-diffusion equation such that $\omega=2$ in \eqref{phys_term_linear}. Here, we assume that the general task consists in reconstructing additional unknown convection terms (which are of derivative order one) such that we can choose $\kappa=1$ in Assumption \ref{ass_init_set}. If only additional unknown reaction terms need to be reconstructed, one could choose $\kappa=0$.
\begin{example}
	Let $V=\tilde{V}=W^{2,\hat{p}}(\Omega)$, $W = L^{\hat{p}}(\Omega)$, $\mathcal{V}, \mathcal{W}$ as in Assumption \ref{ass_init_set}, $\kappa=1$, $\hat{p}=\hat{q}= 2$ and
	\[ F(t,u,\varphi)= \nabla\cdot(a\nabla u)+cu
	\] for $t\in (0,T), u\in \mathcal{V}$ and $\varphi=(a,c)$ with $a\in W^{1,\gamma}(\Omega)$ for $3=d<\gamma<\infty$ and $c\in L^2(\Omega)$. Note that $X_\varphi= W^{1,\gamma}(\Omega)\times L^2(\Omega)$. Thus, the physical term $F$ attains a representation of the form in \eqref{phys_term_linear} with $\omega = 2$, $\Psi\equiv 0$ and under abuse of notation
	\[ \Phi_{\bar{0}}(t,\varphi)=c, ~ \Phi_{e_k}(t,\varphi)=\partial_{x_k}a, ~ \Phi_{2e_k}(t,\varphi)=a
	\]
	for $1\leq k\leq 3$ with $e_k$ the $k$-th unit vector in $\mathbb{R}^3$ and $\bar{0}=(0,0,0)$. Furthermore, we set $\Phi_\beta\equiv 0$ for $\beta\notin\left\{0e_k,e_k,2e_k\right\}_{1\leq k\leq 3}$. We verify the requirements on $\Phi$ in Proposition \ref{prop:ass4_linear} based on the following case distinction for $0\leq \vert\beta\vert\leq 2$.\\
	
	\noindent\textbf{Case 1.} $\vert \beta\vert=0$: For $s_{\bar{0}}=2$ (fulfills $\frac{\omega-\vert\beta\vert}{d}>\frac{1}{\hat{p}}-\frac{1}{\hat{q}}+\frac{1}{s_{\bar{0}}}$ and $\hat{q}\leq s_{\bar{0}}<\infty$) it holds
	\[ \Vert \Phi_{\bar{0}}(t,\varphi)\Vert_{L^{s_{\bar{0}}}(\Omega)}=\Vert c\Vert_{L^{2}(\Omega)}\leq \Vert \varphi\Vert_{X_\varphi},
	\]
	proving a growth condition as in \eqref{growth_cond_lin_terms1} and weak continuity of $\Phi_{\bar{0}}(t,\cdot):X_\varphi\to L^{s_{\bar{0}}}(\Omega)$. Furthermore, since $W^{2,2}(\Omega)\hookdoubleheadrightarrow \mathcal{C}(\overline{\Omega})$ the additional conditions in Proposition \ref{prop:ass4_linear} apply with $c_{\bar{0}}=\infty$ by above considerations since $(\frac{1}{2}-\frac{1}{c_{\bar{0}}})^{-1}=2$.\\
	
	\noindent\textbf{Case 2.} $\vert \beta\vert=1$: For $\hat{q}\leq d< s_{e_k}\leq \gamma<\infty$ (fulfills $\frac{\omega-\vert\beta\vert}{d}>\frac{1}{\hat{p}}-\frac{1}{\hat{q}}+\frac{1}{s_\beta}$) it holds
	\[ \Vert \Phi_{e_k}(t,\varphi)\Vert_{L^{s_{e_k}}(\Omega)}=\Vert \partial_{x_k}a\Vert_{L^{s_{e_k}}(\Omega)}\leq \vert\Omega\vert^{\frac{\gamma-s_{e_k}}{\gamma s_{e_k}}}\Vert \partial_{x_k}a\Vert_{L^\gamma(\Omega)}\leq \vert\Omega\vert^{\frac{\gamma-s_{e_k}}{\gamma s_{e_k}}}\Vert \varphi\Vert_{X_\varphi},
	\] proving a growth condition of the form in \eqref{growth_cond_lin_terms1}. Weak continuity of $\Phi_{e_k}(t,\cdot):X_\varphi\to L^{s_{e_k}}(\Omega)$ follows by the continuous embedding $W^{1,\gamma}(\Omega)\hookrightarrow W^{1,s_{e_k}}(\Omega)$. Furthermore, the additional conditions in Proposition \ref{prop:ass4_linear} apply with $c_{e_k}=\frac{2\gamma}{\gamma-2}$. Indeed, $\gamma>d=3$ is equivalent to $c_{e_k}<6$, so that $\hat{q}=2<c_{e_k}<6=\frac{2d}{d-2}$ and hence $W^{1,2}(\Omega)\hookdoubleheadrightarrow L^{c_{e_k}}(\Omega)$ by the Rellich-Kondrachov Theorem. The remaining growth condition is immediate, since the exponent it requires is $\frac{c_{e_k}\hat{q}}{c_{e_k}-\hat{q}}=(\frac{1}{2}-\frac{1}{c_{e_k}})^{-1}=\gamma$, so that
	\[ \Vert \Phi_{e_k}(t,\varphi)\Vert_{L^{\gamma}(\Omega)}=\Vert \partial_{x_k}a\Vert_{L^{\gamma}(\Omega)}\leq \Vert \varphi\Vert_{X_\varphi}.
	\]
	
	\noindent\textbf{Case 3.} $\vert \beta\vert=2$: We may choose $s_{2e_k}=\infty$. As $\gamma>d$ there exists some constant $c_\gamma>0$ such that $\Vert \cdot\Vert_{\mathcal{C}(\overline{\Omega})} \leq c_\gamma \Vert \cdot\Vert_{W^{1,\gamma}(\Omega)}$ yielding
	\[ \Vert \Phi_{2e_k}(t,\varphi)\Vert_{L^{s_{2e_k}}(\Omega)}=\Vert a\Vert_{L^{\infty}(\Omega)}\leq c_\gamma \Vert a\Vert_{W^{1,\gamma}(\Omega)}\leq c_\gamma \Vert \varphi\Vert_{X_\varphi}
	\]
	and hence, a growth condition of the form in \eqref{growth_cond_lin_terms1}. As $W^{1,\gamma}(\Omega)\hookdoubleheadrightarrow \mathcal{C}(\overline{\Omega})$ by the Rellich-Kondrachov embedding, $\Phi_{2e_k}(t,\cdot):X_\varphi\to L^{s_{2e_k}}(\Omega)$ is weakly(-strongly) continuous, covering also the additional conditions in Proposition \ref{prop:ass4_linear}.\\
	
	\noindent Thus, the requirements on $\Phi$ and $\Psi$ in Proposition \ref{prop:ass4_linear} are fulfilled.
\end{example}
\subsubsection{Nonlinear case} The following result, proven in Appendix \ref{app:nonlinear}, verifies Assumption \ref{ass_phys_term} for general nonlinear physical terms under stronger conditions. Note that instead of weak closedness in Assumption \ref{ass_phys_term}, ii) we show weak-weak continuity which is stronger.
\begin{proposition}
	\label{prop:ass4_nonlinear} Let Assumption \ref{ass_init_set} and the extended state space embedding
	\[ \mathcal{V}\hookrightarrow\mathcal{C}(0,T;H)
	\] hold true. Suppose that the $F_n(\cdot, \cdot, \varphi):(0,T)\times V^N\to W$ satisfy the Carathéodory condition, i.e., $t\mapsto F_n(t,v, \varphi)$ is measurable for $v\in V^N$ and $v\mapsto F_n(t,v,\varphi)$ is continuous for a.e. $t\in(0,T)$. Further assume that the $F_n$ satisfy the growth condition
	\begin{align}
		\label{growth_condition1} \Vert F_n(t,(v_n)_{1\leq n\leq N}, \varphi)\Vert_W \leq \mathcal{B}_0(\Vert\varphi\Vert_{X_\varphi}, \sum_{n=1}^N\Vert v_n\Vert_H)(\Gamma(t)+\sum_{n=1}^N\Vert v_n \Vert_{V})
	\end{align}
	for some $\Gamma\in L^q(0,T)$ and $\mathcal{B}_0:\mathbb{R}^2\to \mathbb{R}$, increasing in the second entry and, for fixed second entry, mapping bounded sets to bounded sets. Then the $F_n:(0,T)\times V^N\times X_\varphi\to W$ induce well-defined Nemytskii operators $F_n: \mathcal{V}^N\times X_\varphi\to \mathcal{W}$ with
	\begin{align}
		\label{F_nonlinear_Nemytskii} [F_n(v,\varphi)](t) = F_n(t, v(t), \varphi)
	\end{align}
	for $v\in \mathcal{V}^N$ and $\varphi\in X_\varphi$. Now suppose weak-weak continuity of
	\begin{align*}
		F_n(t,\cdot): H^N\times X_\varphi&\to W\\
		(v_1,\dots, v_N,\varphi)&\mapsto F_n(t,v_1,\dots, v_N,\varphi)
	\end{align*}
	for a.e. $t\in(0,T)$. Further assume that the $F_n$ satisfy the stricter growth condition
	\begin{align}
		\label{growth_condition_strict1} \Vert F_n(t,(v_n)_{1\leq n\leq N}, \varphi)\Vert_W \leq \mathcal{B}_0(\Vert\varphi\Vert_{X_\varphi}, \sum_{n=1}^N\Vert v_n\Vert_H)(\Gamma(t)+\sum_{n=1}^N\Vert v_n \Vert_{H})
	\end{align}
	for some $\Gamma\in L^q(0,T)$ and $\mathcal{B}_0:\mathbb{R}^2\to \mathbb{R}$ as above. Then \eqref{F_nonlinear_Nemytskii} is weak-weak continuous.
\end{proposition}
\begin{proof}
	See Appendix \ref{app:nonlinear}.
\end{proof}
\begin{remark}
	\label{rem:nonlin_phys_remark} A possible application case of the previous proposition is the following. Assume that there exists a reflexive, separable Banach space $V'$ and $\lambda \in \mathbb{N}_0$ with
	\begin{align}
		\label{F_reduction} H\hookrightarrow W^{\lambda,\hat{p}}(\Omega)\hookrightarrow V'
	\end{align}
	with the property that $F:(0,T)\times (V')^N\times X_\varphi\to W$ is well-defined. One might think of physical terms which regarding the state space variable do not need all higher-order derivative information provided by the space $V$ (eventually given by $V=W^{\kappa+m,p_0}(\Omega)$ as outlined in Remark \ref{rem:choice_V}) but only $\lambda <\kappa+m$ many. Then the growth condition in \eqref{growth_condition1} with $\Vert \cdot \Vert_{V'}$ instead of $\Vert \cdot\Vert_V$ implies condition \eqref{growth_condition_strict1} due to \eqref{F_reduction}. Note that $H$ needs to be regular enough to be embeddable in $W^{\lambda,\hat{p}}(\Omega)$.
	
	The condition in \eqref{F_reduction} can be also understood the other way around. That is for given $H$ one might determine the maximal $\lambda\in \mathbb{N}$ such that $H\hookrightarrow W^{\lambda, \hat{p}}(\Omega)$. Then the previous considerations cover physical terms which are well-defined regarding state space variables with highest derivative order given by $\lambda$.

\end{remark}

To conclude this subsection we give the following example addressing the ideas in Remark \ref{rem:nonlin_phys_remark} more concretely. We restrict ourselves to a single equation which can be immediately generalized to general systems by introducing technical notation. Note that the space setup in the following example is consistent with Assumption \ref{ass_init_set}, but we do not discuss it in order to not distract from the central conditions on the parameters. For some preliminary ideas regarding the embedding $\mathcal{V}\hookrightarrow \mathcal{C}(0,T; H)$ see Remark \ref{rem:regul_cond} where one might have $\tilde{V}=H$.
\begin{example}
	We consider a simple three-dimensional transport problem where it is assumed that the known physics are governed by the inviscid Burgers' equation, i.e., we have $F(u)=-u\partial_x u-u\partial_y u-u\partial_z u$. Anticipating eventual viscosity effects we suppose that the unknown approximated term accounts for these effects. Let $V=W^{2,\hat{p}}(\Omega), \tilde{V}=H=W^{1,2}(\Omega)$, $W=L^{\hat{p}}(\Omega)$, $d=3$, $\kappa=1$ and $\hat{p}=\hat{q}=\frac{6-\epsilon}{4-\epsilon/2}$ for some small $0<\epsilon<1$. Then we have for $u\in V$ as $L^{3/2}(\Omega)\hookrightarrow W$ for some $c>0$ that
	\[ \Vert F(u)\Vert_{W}\leq c\Vert u(\partial_x u+\partial_y u+\partial_z u)\Vert_{L^{3/2}(\Omega)}\leq c\Vert u\Vert_{L^6(\Omega)}\Vert \nabla u\Vert_{L^2(\Omega)}
	\] where the last inequality follows by the generalized Hölder's inequality. Due to the embedding $W^{1,2}(\Omega)\hookrightarrow L^6(\Omega)$ (recall that $d=3$) we derive that $\Vert F(u)\Vert_W\leq c\Vert u\Vert_H^2$ and hence, a growth condition of the form in \eqref{growth_condition_strict1}.
	To see weak-weak continuity of $F:H\to W$ let $(u_n)_n\subseteq H$ with $u_n\rightharpoonup u\in H$ as $n\to \infty$. Then for $w\in L^{\hat{p}^*}(\Omega)$ we have that $$\langle u_n(\partial_x u_n+\partial_y u_n+\partial_z u_n)-u(\partial_x u+\partial_y u+\partial_z u), w\rangle_{L^{\hat{p}}(\Omega),L^{\hat{p}^*}(\Omega)}$$ can be rewritten for $ e = \begin{pmatrix} 1 & 1& 1\end{pmatrix}^T\in\mathbb{R}^3$ by
	\begin{align}
		\label{example_nonlin_weak_conv} \langle u~e\cdot(\nabla u_n-\nabla u), w\rangle_{L^{\hat{p}}(\Omega),L^{\hat{p}^*}(\Omega)}+\langle (u_n-u)~ e\cdot\nabla u_n, w\rangle_{L^{\hat{p}}(\Omega),L^{\hat{p}^*}(\Omega)}.
	\end{align}
	For the first term in \eqref{example_nonlin_weak_conv} note that $\nabla u_n\rightharpoonup \nabla u$ in $L^2(\Omega)$ as $n\to \infty$. As
	\[ \langle u~e\cdot (\nabla u_n-\nabla u), w\rangle_{L^{\hat{p}}(\Omega),L^{\hat{p}^*}(\Omega)} = \int_\Omega u(x)~e\cdot (\nabla u_n(x)-\nabla u(x)) w(x)\dx x
	\] it suffices to show that $uw \in L^2(\Omega)$ to obtain the convergence $u~e\cdot(\nabla u_n-\nabla u)\rightharpoonup 0$ in $L^{\hat{p}}(\Omega)$ as $n\to \infty$. This follows by $u\in W^{1,2}(\Omega)\hookrightarrow L^6(\Omega)$, Hölder's generalized inequality and $\hat{p}^*=\frac{6-\epsilon}{2-\epsilon/2}$ as
	\[ \left(\frac{1}{6}+\frac{2-\epsilon/2}{6-\epsilon}\right)^{-1}=\frac{6-\epsilon}{3-2\epsilon/3}\geq 2.
	\] It remains to show that the second term in \eqref{example_nonlin_weak_conv} approaches zero as $n\to \infty$. By
	\[ \langle (u_n-u)~e\cdot\nabla u_n, w\rangle_{L^{\hat{p}}(\Omega),L^{\hat{p}^*}(\Omega)} = \int_\Omega (u_n(x)-u(x))~e\cdot\nabla u_n(x) w(x)\dx x
	\] it suffices to show that $(\nabla u_n w)_n$ is uniformly bounded in $L^{\frac{6-\epsilon}{5-\epsilon}}(\Omega)$ as $u_n\to u$ in $L^{6-\epsilon}(\Omega)$ by the Rellich-Kondrachov Theorem. This follows by boundedness of $(\nabla u_n)_n$ in $L^2(\Omega)$ due to weak convergence and Hölder's generalized inequality concluding weak-weak continuity of $F$.
\end{example}
\begin{remark}
	The choice of $V=W^{2,\hat{p}}(\Omega)$ in the previous example might seem unnecessarily strong for modeling the inviscid Burgers' equation. The reason for this choice is that we suppose that the additional hidden physics that need to be reconstructed possibly include viscosity effects which are of order $\kappa =1$, i.e., the (un)known physical effects are of equal differential order. As a consequence, the state space $V$ needs to attain higher regularity to cover Assumption \ref{ass_init_set} (see also Remark \ref{rem:choice_V}). If one would anticipate only additional unknown reaction terms ($\kappa=0$) choosing $V$ of lower regularity would be possible. In any case, however, it is important to note that the stronger regularity requirement comes from the fact that we simultaneously recover both the state and a non-linear term acting on the state. During the learning process, this regularity can be enforced by using appropriate regularization. Afterwards, in applications of the learned model, the additional regularity is no longer required.
\end{remark}
\section{The uniqueness problem}
\label{sec:uniqueness} This section contains the two main results of this work, and it proceeds in two steps. First, Proposition \ref{prop:p_dagger} shows that the limit problem \eqref{eq:limit_problem}, in which the measurements are full and noiseless and the unknown nonlinearity is no longer tied to a finite-dimensional parameterization, admits a unique solution. Second, Theorem \ref{thm:uniqueness} shows that solutions of the learning problems \eqref{min_prob_uniqueness}, which are based on incomplete, noisy measurements and on the parameterized classes $\mathcal{F}^m$, approach this unique solution as $m\to\infty$. The two results, the assumptions they require and their implications are contrasted in Table \ref{tab:main_results}.
The starting point of our considerations on uniqueness is Assumption \ref{ass_uniqueness}, iii), where we assume for given full measurement data $(\hat y^l)_l \in \mathcal{Y}^L$ and $F:\mathcal{V}^N\times X_\varphi\to \mathcal{W}^N$, to be understood as in Section \ref{sec:problem_setting}, the existence of $\hat{f}:(\otimes_{k=0}^\kappa \mathcal{V}_k^\times)^N\to \mathcal{W}^N$, a state $(\hat{u}_n^l)_{\substack{1\leq n\leq N\\ 1\leq l\leq L}}\in \mathcal{V}^{N\times L}$, an initial condition $(\hat{u}_{0,n}^l)_{\substack{1\leq n\leq N\\ 1\leq l\leq L}}\in H^{N\times L}$, a boundary condition $(\hat{g}_n^l)_{\substack{1\leq n\leq N\\ 1\leq l\leq L}}\in \mathcal{B}^{N\times L}$ and a source term $(\hat{\varphi}_n^l)_{\substack{1\leq n\leq N\\ 1\leq l\leq L}}\in X_\varphi^{N\times L}$ solving the system of partial differential equations \eqref{rdsystem}, i.e.,
\begin{equation}
	\label{pdesystem} \tag{$S$}
	\begin{aligned}
		&\partial_t\hat{u}^l = F(t,\hat{u}^l,\hat{\varphi}^l)+\hat{f}(t,\mathcal{J}_\kappa\hat{u}^l)\\
		&\text{s.t.} ~ ~ \hat{u}^l(0)=\hat{u}_0^l, ~ \gamma(\hat{u}^l)=\hat{g}^l, ~ ~ 
	\end{aligned}
\end{equation}
together with the measurements
\begin{equation}
	\label{measurements} \tag{$M$} K^\dagger \hat{u}^l = \hat y^l
\end{equation}
for $l=1,\ldots,L$. The results of this section are developed based on Assumption \ref{ass_init_set} to \ref{ass_uniqueness}. Note that under these assumptions, due to \eqref{measurements} and injectivity of the full measurement operator $K^\dagger$ by Assumption \ref{ass_uniqueness}, v), the state $\hat{u}$ is uniquely given in system \eqref{pdesystem} even if the term $\hat{f}$ is not.

We recall that the bounded Lipschitz domain $U$ is chosen and fixed according to Assumption \ref{ass_uniqueness}, vii). Note that by Assumption \ref{ass_uniqueness}, ii), it holds that $\mathcal{F}_n^m\subseteq W^{1,\infty}(U)$ for $1\leq n\leq N$, $m\in\mathbb{N}$.

Before we move on to the limit problem and question of uniqueness let us justify the choice of regularization for $f_\theta\in W^{1,\infty}(U)^N$. The problem of using the $W^{1,\infty}(U)$-norm directly is that its powers are not strictly convex which is necessary for uniqueness issues later. This is overcome by the well known equivalence of the norms $\Vert \cdot\Vert_{W^{1,\infty}(U)}$ and $\Vert\cdot\Vert_{L^\rho(U)}+\vert\cdot\vert_{W^{1,\infty}(U)}$ on $W^{1,\infty}(U)$ for bounded domains $U$, which follows by \cite[6.12 A lemma of J.-L. Lions]{Brezis2010} and \cite[Theorem 9.16 (Rellich–Kondrachov)]{Brezis2010}. That is, the space $W^{1,\infty}(U)$ may be strictly convexified under the equivalent norm $\Vert\cdot\Vert_{L^\rho(U)}+\vert\cdot\vert_{W^{1,\infty}(U)}$ for $1<\rho<\infty$ with $\vert\cdot\vert_{W^{1,\infty}(U)}$ the seminorm in $W^{1,\infty}(U)$. For the role and necessity of the exponent $\infty$ in this penalty, we refer to the discussion following Assumption \ref{ass_uniqueness}.

The following proposition introduces the limit problem and shows uniqueness:
\begin{proposition}[Uniqueness]
	\label{prop:p_dagger} Let Assumptions \ref{ass_init_set} to \ref{ass_uniqueness} without Assumption \ref{ass_uniqueness}, iv) be satisfied. Then there exists a unique solution $(\varphi^\dagger, u^\dagger, u_0^\dagger, g^\dagger, f^\dagger)\in X_\varphi^{N\times L}\times \mathcal{V}^{N\times L}\times H^{N\times L}\times\mathcal{B}^{N\times L}\times W^{1,\infty}(U)^N$ to
	\begin{equation} \tag{$\mathcal{P}^\dagger$} \label{eq:limit_problem}
		\begin{aligned}
			\min_{\substack{\varphi\in X_\varphi^{N\times L},u\in\mathcal{V}^{N\times L},\\ u_0\in H^{N\times L}, g\in \mathcal{B}^{N\times L},\\ f\in W^{1,\infty}(U)^N}}
			& \mathcal{R}_0(\varphi, u,u_0,g)+\Vert f\Vert_{L^\rho(U)}^\rho+\Vert\nabla f\Vert_{L^{\infty}(U)}
			\\
			\text{s.t. } & \frac{\partial}{\partial t}u_n^l-F_n(t, u^l_1,\dots, u^l_N,\varphi^l_n)-f_{n}(t,\mathcal{J}_\kappa u^l_1,\dots, \mathcal{J}_\kappa u^l_N) = 0, \\
			& K^\dagger u^l = y^l, \,  u_n^{l}(0) = u_{0,n}^l , \,  \gamma(u^l) = g^l.
			& 
		\end{aligned}		
	\end{equation}
\end{proposition}
\begin{proof}
	First of all, the constraint set of problem \eqref{eq:limit_problem} is not empty by Assumption \ref{ass_uniqueness}, iii), i.e., there exist admissible functions solving system \eqref{pdesystem} such that \eqref{measurements} holds true. Due to injectivity of the full measurement operator $K^\dagger$, for any element satisfying the constraint set of \eqref{eq:limit_problem} the state is uniquely given by $u^\dagger =\hat{u}$. As a consequence, also the initial and boundary trace are uniquely determined by $u_0^\dagger =u^\dagger(0)=\hat{u}(0)=\hat{u}_0$ and $(g^{\dagger, l}_n)_{n,l}=(\gamma(u^{\dagger, l}_n))_{n,l}=(\gamma(\hat{u}^l_n))_{n,l}=(\hat{g}^l_n)_{n,l}$, respectively. By Assumption \ref{ass_uniqueness}, i), iii) and vi) it follows that
	\begin{align}
		\label{uhat_in_U} \Vert \mathcal{J}_\kappa u^\dagger\Vert_{L^\infty((0,T)\times \Omega)}= \Vert \mathcal{J}_\kappa \hat{u}\Vert_{L^\infty((0,T)\times \Omega)}\leq c_\mathcal{V}\Vert \hat{u}\Vert_{\mathcal{V}}\leq c_\mathcal{V}\pi(\mathcal{R}_0(\hat{\varphi}, \hat{u},\hat{u}_0,\hat{g}))
	\end{align}
	and hence, that $(t,\mathcal{J}_\kappa u^{\dagger,l}(t,x))\in U$ for $(t,x)\in(0,T)\times\Omega$ by Assumption \ref{ass_uniqueness}, vii). 
	
	Thus, problem \eqref{eq:limit_problem} may be rewritten equivalently by
	\begin{equation}
		\label{eq:reduced_limitproblem}
		\begin{aligned}
			\min_{\substack{\varphi\in X_\varphi^{N\times L},\\f\in W^{1,\infty}(U)^N}}
			& \mathcal{R}_0(\varphi, u^\dagger,u^\dagger_0,g^\dagger)+\Vert f\Vert_{L^\rho(U)}^\rho+\Vert\nabla f\Vert_{L^{\infty}(U)}
			\\
			\text{s.t. } & \frac{\partial}{\partial t}u^{\dagger, l}-F(t, u^{\dagger, l},\varphi^l)-f(t,\mathcal{J}_\kappa u^{\dagger,l}) = 0.
		\end{aligned}	
	\end{equation}
	The existence of a solution to \eqref{eq:reduced_limitproblem} follows by the direct method: In the following, w.l.o.g., we omit a relabelling of sequences to convergent subsequences. Using the norm equivalence of $\Vert\cdot\Vert_{W^{1,\infty}(U)}$, $\Vert\cdot\Vert_{L^\rho(U)}+\vert\cdot\vert_{W^{1,\infty}(U)}$ and coercivity of $\mathcal{R}_0$ a minimizing sequence $(\varphi^k, f^k)_k\subseteq X_\varphi^{N\times L}\times W^{1,\infty}(U)^N$ to \eqref{eq:reduced_limitproblem} is bounded. Thus, there exist $\varphi'\in X_\varphi^{N\times L}$ and $f'\in W^{1,\infty}(U)^N$ such that $\varphi^k\rightharpoonup \varphi'$ in $X_\varphi^{N\times L}$ and $f^k\overset{*}{\rightharpoonup} f'$ in $W^{1,\infty}(U)^N$ as $k\to \infty$ by reflexivity of $X_\varphi^{N\times L}$ and $W^{1,\infty}(U)^N$ being the dual of a separable space. By $f^k\overset{*}{\rightharpoonup} f'$ in $L^{\infty}(U)^N$ and $\nabla f^k\overset{*}{\rightharpoonup} \nabla f'$ in $L^{\infty}(U)^N$ as $k\to \infty$ together with $L^\infty(U)\hookrightarrow L^\rho(U)$, $1<\rho<\infty$ and weak lower semicontinuity of $\mathcal{R}_0$ it follows that $(\varphi',f')\in X_\varphi^{N\times L}\times W^{1,\infty}(U)^N$ minimizes the objective functional of \eqref{eq:reduced_limitproblem}. We argue that also 
	\begin{align}
		\label{eq:pde_in_limit} \partial_tu^{\dagger,l} = F(t,u^{\dagger,l},\varphi'^l)+f'(t,\mathcal{J}_\kappa u^{\dagger,l})
	\end{align}
	concluding that $(\varphi',f')$ is indeed a solution of \eqref{eq:reduced_limitproblem}. For that note that $f^k \to f'$ in $\mathcal{C}(\overline{U})^N$ as $k\to \infty$ by the Rellich-Kondrachov Theorem. Thus, by $L^q(0,T;L^{\hat{p}}(\Omega))\hookrightarrow \mathcal{W}$ and boundedness of $U$ together with \eqref{uhat_in_U} and $u^\dagger = \hat{u}$ we have for some $c>0$
	\[ \Vert f^k(\mathcal{J}_\kappa u^{\dagger,l})-f'(\mathcal{J}_\kappa u^{\dagger,l})\Vert_{\mathcal{W}}\leq c\Vert f^k-f'\Vert_{L^\infty(U)},
	\] and conclude that  $f^k(\mathcal{J}_\kappa u^{\dagger,l})\to f'(\mathcal{J}_\kappa u^{\dagger,l})$ in $\mathcal{W}^N$ as $k\to \infty$. Using this, as a consequence of boundedness of $\Vert \partial_t u^{\dagger,l}-f^k(\mathcal{J}_\kappa u^{\dagger,l})\Vert_{\mathcal{W}}$ for $k\in \mathbb{N}$ it follows by Assumption \ref{ass_phys_term}, ii) that $F(u^{\dagger,l}, \varphi^{k,l})\rightharpoonup F(u^{\dagger,l}, \varphi'^l)$ in $\mathcal{W}^N$ as $k\to \infty$ and we recover \eqref{eq:pde_in_limit}.
	
	Finally, uniqueness of $(\varphi^\dagger,f^\dagger)=(\varphi',f')$ as solution to \eqref{eq:reduced_limitproblem} follows from strict convexity of the objective functional in $(\varphi,f)\in X_\varphi^{N\times L}\times W^{1,\infty}(U)^N$ and from $F$ being affine with respect to $\varphi$.
\end{proof}
Let us comment on the relation between the solution $(\varphi^\dagger, u^\dagger, u_0^\dagger, g^\dagger, f^\dagger)$ of \eqref{eq:limit_problem} and the admissible, data-generating functions $(\hat{\varphi},\hat{u},\hat{u}_0,\hat{g},\hat{f})$ of Assumption \ref{ass_uniqueness}, iii). As shown in the proof of Proposition~\ref{prop:p_dagger}, the state is identified exactly, $u^\dagger =\hat{u}$, and with it the initial and boundary conditions, $u_0^\dagger=\hat{u}_0$ and $g^\dagger=\hat{g}$. For the remaining components, subtracting the two versions of the state equation along the common state $u^\dagger=\hat{u}$ yields, for $1\leq l\leq L$ and a.e. $(t,x)\in(0,T)\times \Omega$,
\[ (f^\dagger-\hat{f})(t,\mathcal{J}_\kappa \hat{u}^l(t,x)) = \big[F(t,\hat{u}^l,\hat{\varphi}^l)-F(t,\hat{u}^l,\varphi^{\dagger,l})\big](t,x).
\] In particular, the learned model reproduces the data-generating dynamics along the visited states
\[ \mathcal{S}:=\{(t,\mathcal{J}_\kappa \hat{u}^l(t,x))~\vert~ (t,x)\in (0,T)\times\Omega,~ 1\leq l\leq L\},
\] so that predictions along these states are exact in the limit. If moreover the parameters of the physical term are identified in the sense that $F(t,\hat{u}^l,\varphi^{\dagger,l})=F(t,\hat{u}^l,\hat{\varphi}^l)$, which is in particular the case if $F$ does not depend on unknown parameters, the above identity shows that $f^\dagger=\hat{f}$ on $\mathcal{S}$. That is, the unknown nonlinearity is then identified exactly where the data probe it, and the regularization only determines its extension to the remainder of $U$. In general, however, the regularization also selects the splitting of the dynamics into a known and an unknown part. Already for $F(t,u,\varphi)=\varphi u$ with a scalar parameter, the pairs $(\hat{\varphi}-c,\hat{f}+c\,\mathrm{id})$ for $c\in \mathbb{R}$ generate identical dynamics for every state, such that the constraints in \eqref{eq:limit_problem} cannot distinguish them and the objective functional selects one of them. Finally, since the data-generating functions are admissible for \eqref{eq:limit_problem} and the latter attains a unique solution, it holds that $(\varphi^\dagger,f^\dagger)=(\hat{\varphi},\hat{f})$ if and only if the data-generating functions minimize the objective functional of \eqref{eq:limit_problem}, which in particular covers the case that the measurements alone determine the pair.

Recall that, under Assumption \ref{ass_uniqueness}, the minimization problem \eqref{min_prob} reduces to the following specific case:
\begin{equation}
	\label{min_prob_uniqueness} \tag{$\mathcal{P}^m$}
	\begin{aligned}
		&\min_{\substack{\varphi\in X_\varphi^{N\times L},\theta\in\otimes_n\Theta^m_n,\\ u\in\mathcal{V}^{N\times L}, u_0\in H^{N\times L},\\ g\in \mathcal{B}^{N\times L}}}\sum_{1\leq l\leq L}\bigg[\lambda^m \bigg(\Vert \frac{\partial}{\partial t}u^l-F(t, u^l,\varphi^l)-f_{\theta}(t,\mathcal{J}_\kappa u^l)\Vert_{\mathcal{W}}^q\\
		&+\Vert u^{l}(0)-u_{0}^l\Vert_{H} ^2+\mathcal{D}_{\BC}(\gamma(u^l)-g^l)\bigg) + \mu^m\Vert K^mu^l-y^{m,l}\Vert_\mathcal{Y}^r\bigg] \\ 
		& +
		\mathcal{R}_0(\varphi, u,u_0,g)+ \nu^m \|\theta\| + \Vert f_\theta\Vert_{L^\rho(U)}^\rho+\Vert\nabla f_\theta\Vert_{L^{\infty}(U)}	
	\end{aligned}
\end{equation}
for a sequence of measured data $\mathcal{Y}\ni y^{m,l}\approx K^m u^{\dagger,l}$ for  $m\in\mathbb{N}$ and $1\leq l\leq L$ with $u^\dagger$ as in Proposition \ref{prop:p_dagger}. More concretely the measured data $(y^{m,l})_m\subseteq \mathcal{Y}$ is supposed to be given under some noise estimation
\begin{align}
	\label{estim:noise_estimator} \Vert y^{m,l}-K^mu^{\dagger,l}\Vert_{\mathcal{Y}}\leq \delta(m)
\end{align}
for $1\leq l\leq L$ with $\delta:\mathbb{N}\to \mathbb{R}_{\geq 0}$ the noise estimator such that $\lim_{m\to \infty}\delta(m)=0$. We again emphasize the role of the index $m$. It serves as a scale parameter governing all quantities refined in the limit:
\begin{itemize}
	\item~ The approximation classes $\mathcal{F}^m$ in \eqref{param_approx_classes} together with their parameter sets $\Theta^m$.
	\item~ The measurement operators $K^m$, the measured data $y^{m,l}$ with their noise level $\delta(m)$ in \eqref{estim:noise_estimator}.
	\item~ The weights $\lambda^m,\mu^m,\nu^m$ in \eqref{min_prob_uniqueness}.
\end{itemize}
Note that no relation between these quantities is assumed a priori, in particular for the well-posedness of \eqref{min_prob}, which is studied in Section \ref{sec:problem_setting} for a fixed $m$. Their interplay becomes relevant for the convergence result in the limit $m\to \infty$ as considered next.

Our main result on approximating the unique solution of \eqref{eq:limit_problem} is now the following:
\begin{theorem}[Consistency]
	\label{thm:uniqueness} Let Assumptions \ref{ass_init_set}-\ref{ass_uniqueness} hold true with the approximation capacity condition in Assumption \ref{ass_uniqueness}, iv) being satisfied for $f^\dagger$ where $(\varphi^\dagger, u^\dagger, u_0^\dagger, g^\dagger, f^\dagger)$ is the unique solution to \eqref{eq:limit_problem}.
	\begin{itemize}[leftmargin=*]
		\item ~Let $(\varphi^m,\theta^m,u^m,u_0^m,g^m)$ be a solution to \eqref{min_prob_uniqueness} for each $m \in \mathbb{N}$.\\[-4pt]
		\item ~Let further the parameters $\lambda^m,\mu^m,\nu^m>0$ be chosen such that $\lambda^m\to \infty$, $\mu^m\to \infty$ and $\nu^m\to 0$ with $\lambda^m m^{-\beta q} = o(1)$, $\mu^m\delta(m)^r=o(1)$ and $\nu^m\psi(m)=o(1)$ as $m\to \infty$.
	\end{itemize}
	Then $\varphi^m \rightharpoonup \varphi^\dagger$ in $X_\varphi^{N\times L}$, $u^m \rightharpoonup u^\dagger $ in $\mathcal{V}^{N\times L}$, $u_0^m \rightharpoonup u_0^\dagger$ in $H^{N\times L}$, $g^m\rightharpoonup g^\dagger$ in $\mathcal{B}^{N\times L}$ and $f_{\theta^m} \overset{*}{\rightharpoonup} f^\dagger $ in $W^{1,\infty}(U)^N$.
	
\end{theorem}
\begin{proof}
	First of all, the existence of solutions $(\varphi^m,\theta^m,u^m,u_0^m,g^m)$ to \eqref{min_prob_uniqueness} for each $m \in \mathbb{N}$ follows by the direct method which is discussed in all details in Appendix \ref{app:existence} on the existence of minimizers. Let now $c>0$ be a generic constant used throughout the following estimations. By Assumption \ref{ass_uniqueness}, iv) being satisfied for $f^\dagger$, there exist $\tilde{\theta}^m\in \otimes_{n=1}^N \Theta_n^m$  such that $\Vert f^\dagger-f_{\tilde{\theta}^m}\Vert_{L^\infty(U)} \leq c m^{-\beta}$ and $\Vert \tilde{\theta}^m\Vert \leq \psi(m)$ for $m\in \mathbb{N}$ together with $\Vert\nabla f_{\tilde{\theta}^m}\Vert_{L^\infty(U)}\to \Vert \nabla f^\dagger\Vert_{L^\infty(U)}$ as $m\to \infty$. As $(\varphi^m, \theta^m,u^m, u_0^m, g^m)$ is a solution to Problem \eqref{min_prob_uniqueness} we may estimate its objective functional value using the noise estimate \eqref{estim:noise_estimator} by
	\begin{align}
		\label{first_comparison}
		\notag&\sum_{1\leq l\leq L}\bigg[\lambda^m \bigg(\Vert \frac{\partial}{\partial t}u^{m,l}-F(t, u^{m,l},\varphi^{m,l})-f_{\theta^m}(t,\mathcal{J}_\kappa u^{m,l})\Vert_{\mathcal{W}}^q\\
		\notag&~ ~ +\Vert u^{m,l}(0)-u_{0}^{m,l}\Vert_{H} ^2+\mathcal{D}_{\BC}(\gamma(u^{m,l})-g^{m,l})\bigg) + \mu^m\Vert K^mu^{m,l}-y^{m,l}\Vert_\mathcal{Y}^r\bigg]\\
		\notag& ~ ~+\mathcal{R}_0(\varphi^m, u^m,u_0^m,g^m) +\nu^m \|\theta^m\| 	+ \Vert f_{\theta^m}\Vert_{L^\rho(U)}^\rho+\Vert\nabla f_{\theta^m}\Vert_{L^{\infty}(U)}\\
		\notag\leq&\sum_{1\leq l\leq L}\left[\lambda^m\Vert \frac{\partial}{\partial t}u^{\dagger,l}-F(t, u^{\dagger,l},\varphi^{\dagger,l})-f_{\tilde{\theta}^m}(t,\mathcal{J}_\kappa u^{\dagger,l})\Vert_{\mathcal{W}}^q+\mu^m\delta(m)^r\right]\\
		& ~ ~ +\mathcal{R}_0(\varphi^\dagger, u^\dagger,u_0^\dagger,g^\dagger)+\nu^m \|\tilde{\theta}^m\| 	+ \Vert f_{\tilde{\theta}^m}\Vert_{L^\rho(U)}^\rho+\Vert\nabla f_{\tilde{\theta}^m}\Vert_{L^{\infty}(U)}.
	\end{align}
	We may further estimate one part of the sum on the right-hand side of \eqref{first_comparison} by
	\begin{multline*}
		\sum_{1\leq l\leq L}\Vert \frac{\partial}{\partial t}u^{\dagger,l}-F(t,u^{\dagger,l},\varphi^{\dagger,l})-f_{\tilde{\theta}^m}(t,\mathcal{J}_\kappa u^{\dagger,l})\Vert_{\mathcal{W}}^q\\
		= \sum_{1\leq l\leq L}\Vert f^\dagger(t,\mathcal{J}_\kappa u^{\dagger,l})-f_{\tilde{\theta}^m}(t,\mathcal{J}_\kappa u^{\dagger, l})\Vert_{\mathcal{W}}^q\leq c\Vert f^\dagger-f_{\tilde{\theta}^m}\Vert_{L^\infty(U)}^q\leq cm^{-\beta q}
	\end{multline*}
	where in the penultimate estimation we have used $(t,\mathcal{J}_\kappa u^{\dagger, l})\in U$ which follows by Proposition \ref{prop:p_dagger} together with \eqref{uhat_in_U}, and in the last step Assumption \ref{ass_uniqueness}, iv). By 
	\[ \lim_{m\to \infty}\Vert f_{\tilde{\theta}^m}\Vert_{L^\rho(U)}^\rho+\Vert\nabla f_{\tilde{\theta}^m}\Vert_{L^{\infty}(U)} = \Vert f^\dagger\Vert_{L^\rho(U)}^\rho+\Vert\nabla f^\dagger\Vert_{L^{\infty}(U)},
	\] due to Assumption \ref{ass_uniqueness}, iv), and the choice of the $\lambda^m, \mu^m, \nu^m$ the right-hand side of \eqref{first_comparison} converges to
	\begin{align}
		\label{limit_rhs} \mathcal{R}_0(\varphi^\dagger, u^\dagger,u_0^\dagger,g^\dagger)+\Vert f^\dagger\Vert_{L^\rho(U)}^\rho+\Vert\nabla f^\dagger\Vert_{L^{\infty}(U)}
	\end{align}
	as $m\to \infty$ which is exactly the objective functional of problem \eqref{eq:limit_problem}. Using that $(\varphi^\dagger, u^\dagger, u_0^\dagger, g^\dagger, f^\dagger)$ is the minimizer to \eqref{eq:limit_problem} we can estimate \eqref{limit_rhs} from above by
	\[ \mathcal{R}_0(\hat{\varphi}, \hat{u},\hat{u}_0,\hat{g})+\Vert \hat{f}\Vert_{L^\rho(U)}^\rho+\Vert\nabla \hat{f}\Vert_{L^{\infty}(U)}
	\] for admissible $\hat{f}\in W^{1,\infty}(\mathbb{R}^D)^N$, $\hat{u}\in \mathcal{V}^{N\times L}$, $\hat{\varphi}\in X_\varphi^{N\times L}$, $\hat{u}_0\in H^{N\times L}$, $\hat{g}\in \mathcal{B}^{N\times L}$ according to Assumption \ref{ass_uniqueness}, iii). As a consequence, for $m$ sufficiently large it follows by Assumption \ref{ass_uniqueness}, i) and vi), that
	\begin{align}
		\label{est:Jum}
		\notag\Vert \mathcal{J}_\kappa u^m\Vert_{L^\infty((0,T)\times\Omega)}&\leq c_\mathcal{V}\Vert u^m\Vert_{\mathcal{V}} \leq c_\mathcal{V} \pi(\mathcal{R}_0(\varphi^m, u^m,u_0^m,g^m))\\
		&\leq c_\mathcal{V} \pi(\mathcal{R}_0(\hat{\varphi}, \hat{u},\hat{u}_0,\hat{g})+\Vert \hat{f}\Vert_{L^\rho(U)}^\rho+\Vert\nabla \hat{f}\Vert_{L^{\infty}(U)}+1).
	\end{align}
	Hence, we derive that $(t,\mathcal{J}_\kappa u^{m,l})\in U$ for $m$ sufficiently large by monotonicity of $\pi$ and $\Vert \hat{f}\Vert_{L^\rho(U)}^\rho+\Vert\nabla \hat{f}\Vert_{L^{\infty}(U)}\leq \Vert \hat{f}\Vert_{L^\rho(\mathbb{R}^D)}^\rho+\Vert\nabla \hat{f}\Vert_{L^{\infty}(\mathbb{R}^D)}$. By convergence of the right-hand side of \eqref{first_comparison} the terms $\Vert\varphi^m\Vert_{X_\varphi}$ and $\Vert u^m\Vert_{\mathcal{V}}$ are bounded due to coercivity of $\mathcal{R}_0$. Similarly boundedness of $\Vert f_{\theta^m}\Vert_{W^{1,\infty}(U)}$ follows using the norm equivalence of $\Vert\cdot\Vert_{W^{1,\infty}(U)}$ and $\Vert\cdot\Vert_{L^\rho(U)}+\vert\cdot\vert_{W^{1,\infty}(U)}$. Boundedness of $\Vert u_0^m\Vert_{H}$ follows as $\lambda^m\to \infty$ as $m\to \infty$ together with boundedness of $\Vert u^m(0)\Vert_{H}$, which holds by boundedness of $\Vert u^m\Vert_{\mathcal{V}}$, and continuity of the initial condition map shown in Appendix \ref{app:existence}, II. Finally by $\lambda^m\to \infty$ as $m\to \infty$, coercivity of $\mathcal{D}_\BC$, boundedness of $\gamma$ and boundedness of the $u^m$, also boundedness of $\Vert g^m\Vert_{\mathcal{B}}$ can be inferred. As a consequence of reflexivity of $X_\varphi^{N\times L}$, $\mathcal{V}^{N\times L}, H^{N\times L}, \mathcal{B}^{N\times L}$ and the fact that  $W^{1,\infty}(U)^N$ is the dualspace of a separable space, we derive that there exist weakly convergent subsequences (w.l.o.g. the whole sequences as we will see subsequently that the limit is unique) and $\tilde{\varphi}\in X_\varphi^{N\times L}, \tilde{u}\in \mathcal{V}^{N\times L}, \tilde{u}_0\in H^{N\times L}, \tilde{g}\in\mathcal{B}^{N\times L}$ and similarly a weak-$*$ convergent subsequence and $\tilde{f}\in W^{1,\infty}(U)^N$ with $\varphi^m\rightharpoonup \tilde{\varphi}$, $u^m\rightharpoonup \tilde{u}$, $u_0^m\rightharpoonup \tilde{u}_0$, $g^m\rightharpoonup \tilde{g}$, $f_{\theta^m}\overset{*}{\rightharpoonup} \tilde{f}$ as $m\to \infty$ (by \cite[Theorem 3.18]{Brezis2010} and Banach-Alaoglu-Bourbaki e.g. in \cite[Theorem 3.16]{Brezis2010}). By weak lower semicontinuity and weak-$*$ lower semicontinuity together with the previous considerations we derive
	\begin{multline}
		\label{Restim}
		\mathcal{R}_0(\tilde{\varphi},\tilde{u},\tilde{u}_0,\tilde{g})+\Vert\tilde{f}\Vert_{L^\rho(U)}^\rho+\Vert\nabla\tilde{f}\Vert_{L^\infty(U)}\\
		\leq \liminf_{m\to \infty}\mathcal{R}_0(\varphi^m,u^m,u_0^m,g^m)+\Vert f_{\theta^m}\Vert_{L^\rho(U)}^\rho+\Vert\nabla f_{\theta^m}\Vert_{L^\infty(U)}\\
		\leq \mathcal{R}_0(\varphi^\dagger, u^\dagger,u_0^\dagger,g^\dagger)+\Vert f^\dagger\Vert_{L^\rho(U)}^\rho+\Vert\nabla f^\dagger\Vert_{L^{\infty}(U)}
	\end{multline}
	We argue that $\tilde{u}=u^\dagger$: As the right-hand side of \eqref{first_comparison} converges it holds true that 
	\begin{align}
		\label{kkconvergence1} K^mu^{m,l}-y^{m,l}\to 0 ~ ~ ~ \text{strongly in } ~ \mathcal{Y} ~ \text{ as } ~ m\to \infty
	\end{align}
	due to $\mu^m\to \infty$ as $m\to \infty$. The following estimation shows that $K^mu^{m,l}$ converges to $K^\dagger \tilde{u}^l$ as $m\to \infty$.
	Due to the triangle inequality we have that
	\begin{align*}
		\Vert K^mu^{m,l} - K^\dagger \tilde{u}^l\Vert_{\mathcal{Y}}&\leq \Vert K^m u^{m,l}-K^\dagger u^{m,l}\Vert_{\mathcal{Y}}+\Vert K^\dagger u^{m,l} -K^\dagger \tilde{u}^l\Vert_{\mathcal{Y}}.
	\end{align*}
	Employing the convergence condition in \eqref{operator_conv_weak}, ensuring that the first term on the right-hand side converges to zero, and weak-strong continuity of $K^\dagger$, implying $K^\dagger u^{m,l}\to K^\dagger \tilde{u}^l$ in $\mathcal{Y}$ as $m\to \infty$, we recover that indeed
	\begin{align}
		\label{kkconvergence2} K^mu^{m,l}\to K^\dagger \tilde{u}^l ~ ~ ~ \text{strongly in } ~ \mathcal{Y} ~ \text{ as } ~ m\to \infty.
	\end{align}
	Thus, by \eqref{estim:noise_estimator}, the convergences \eqref{kkconvergence1}, and \eqref{kkconvergence2}, together with Assumption \ref{ass_uniqueness}, v), and
	\begin{align*}
	\Vert K^\dagger \tilde{u}^l-K^\dagger u^{\dagger,l}\Vert_{\mathcal{Y}}\leq \Vert K^\dagger \tilde{u}^l-K^mu^{m,l}\Vert_{\mathcal{Y}}+\Vert K^mu^{m,l}-y^{m,l}\Vert_{\mathcal{Y}} +\Vert y^{m,l}-K^mu^{\dagger,l}\Vert_{\mathcal{Y}}+\Vert K^mu^{\dagger,l}-K^\dagger u^{\dagger,l}\Vert_{\mathcal{Y}}
	\end{align*}
	we derive $K^\dagger \tilde{u}^l=K^\dagger u^{\dagger,l}$. As a consequence of injectivity of $K^\dagger$ we finally derive that $\tilde{u}=u^\dagger$. We argue next that $\tilde{u}_0=u_0^\dagger$. For that, note once more that by convergence of the right-hand side of \eqref{first_comparison} and $\lambda^m\to \infty$ as $m\to \infty$ we obtain that $u^{m}(0)-u_0^{m}\to 0$ in $H^{N\times L}$ as $m\to \infty$. As $u_0^m\rightharpoonup \tilde{u}_0$ in $H^{N\times L}$ as $m\to \infty$ we recover that $u^m(0)\rightharpoonup \tilde{u}_0$ in $H^{N\times L}$ as $m\to \infty$. Together with $u^m\rightharpoonup u^\dagger$, by what we have just shown, and weak closedness of the initial condition evaluation verified in II. of Appendix \ref{app:existence}, we obtain that indeed $\tilde{u}_0=u^\dagger(0)=u_0^\dagger$. By similar arguments and the assumption that $\mathcal{D}_\BC(z)=0$ for $z\in \mathcal{B}^N$ iff $z=0$ we obtain that $\gamma(u^{m,l})-g^{m,l}\to 0$ in $\mathcal{B}^N$ as $m\to \infty$. As $g^{m}\rightharpoonup \tilde{g}$ and $\gamma(u^{m,l})\rightharpoonup \gamma(u^{\dagger,l})=g^{\dagger,l}$ by continuity of $\gamma$, both in $\mathcal{B}^N$ as $m\to \infty$, it also holds $\tilde{g}=g^\dagger$. It remains to show $\tilde{\varphi}=\varphi^\dagger$ and $\tilde{f}=f^\dagger$. Using the already discussed identities for $\tilde{u}, \tilde{u}_0$ and $\tilde{g}$, estimation \eqref{Restim} yields
	\begin{align*}
		\mathcal{R}_0(\tilde{\varphi},u^\dagger,u_0^\dagger,g^\dagger)+\Vert\tilde{f}\Vert_{L^\rho(U)}^\rho+\Vert\nabla\tilde{f}\Vert_{L^\infty(U)} \leq \mathcal{R}_0(\varphi^\dagger, u^\dagger,u_0^\dagger,g^\dagger)+\Vert f^\dagger\Vert_{L^\rho(U)}^\rho+\Vert\nabla f^\dagger\Vert_{L^{\infty}(U)}.
	\end{align*}
	Moreover, as the right-hand side of \eqref{first_comparison} converges as $m\to\infty$, it holds true that
	\begin{align}
		\label{limitconstraint} \lim_{m\to \infty}\sum_{1\leq l\leq L}\Vert \frac{\partial}{\partial t}u^{m,l}-F(t, u^{m,l},\varphi^{m,l})-f_{\theta^m}(t,\mathcal{J}_\kappa u^{m,l})\Vert_{\mathcal{W}}^q = 0
	\end{align}
	due to $\lambda^m\to \infty$ as $m\to \infty$. We argue that 
	\[ \frac{\partial}{\partial t}u^{m,l}-F(t, u^{m,l},\varphi^{m,l})-f_{\theta^m}(t,\mathcal{J}_\kappa u^{m,l})\rightharpoonup \frac{\partial}{\partial t}u^{\dagger,l}-F(t, u^{\dagger,l},\tilde{\varphi}^{l})-\tilde{f}(t,\mathcal{J}_\kappa u^{\dagger,l})
	\] as $m\to \infty$ in $\mathcal{W}^N$, which together with \eqref{limitconstraint} and weak lower semicontinuity of the $\Vert\cdot\Vert_{\mathcal{W}}$-norm implies
	\begin{align}
		\label{eq:hard_constraint} \frac{\partial}{\partial t}u^{\dagger,l}=F(t, u^{\dagger,l},\tilde{\varphi}^{l})+\tilde{f}(t,\mathcal{J}_\kappa u^{\dagger,l}).
	\end{align}
	By Assumption \ref{ass_uniqueness}, viii), and the considerations in Appendix \ref{app:existence}, I. showing weak continuity of the temporal derivative, it follows that
	\begin{align}
		\label{weak_conv_firstpart} \frac{\partial}{\partial t}u^{m,l}-F(t,u^{m,l},\varphi^{m,l})\rightharpoonup \frac{\partial}{\partial t} u^{\dagger,l}-F(t,u^{\dagger,l},\tilde{\varphi}^l)
	\end{align}
	as $m\to \infty$ in $\mathcal{W}^N$. It remains to argue that $f_{\theta^m}(t, \mathcal{J}_\kappa u^{m,l})\rightharpoonup \tilde{f}(t,\mathcal{J}_\kappa u^{\dagger,l})$ in $\mathcal{W}^N$ as $m\to \infty$. Using \eqref{limitconstraint} and \eqref{weak_conv_firstpart} we obtain that the $\Vert f_{\theta^m}(t,\mathcal{J}_\kappa u^{m,l})\Vert_{\mathcal{W}}$ are bounded for $m\in \mathbb{N}$ and thus, the $(f_{\theta^m}(t, \mathcal{J}_\kappa u^{m,l}))_m$ attain a weakly convergent subsequence in $\mathcal{W}^N$. We show that indeed $f_{\theta^m}(t,\mathcal{J}_\kappa u^{m,l})\to \tilde{f}(t,\mathcal{J}_\kappa u^{\dagger,l})$ in $\mathcal{W}^N$ as $m\to \infty$. As $U$ is bounded, open and has a Lipschitz-regular boundary we have that $W^{1,\infty}(U)^N\hookdoubleheadrightarrow \mathcal{C}(\overline{U})^N$ by Rellich-Kondrachov and consequently, the convergence $f_{\theta^m}\to \tilde{f}$ holds uniformly on $U$ as $m\to \infty$. Thus, in particular $f_{\theta^m}(t, \mathcal{J}_\kappa u^{m,l})- \tilde{f}(t, \mathcal{J}_\kappa u^{m,l})\to0$ in $\mathcal{W}^N$ as $m\to \infty$ as
	\[ \Vert f_{\theta^m}(t,\mathcal{J}_\kappa u^{m,l})-\tilde{f}(t,\mathcal{J}_\kappa u^{m,l})\Vert_{\mathcal{W}}\leq c\Vert f_{\theta^m}-\tilde{f}\Vert_{L^\infty(U)}
	\] for some $c=c(T,\Omega)>0$, for $m$ sufficiently large such that $(t,\mathcal{J}_\kappa u^{m,l})\in U$. The convergence $\tilde{f}(t,\mathcal{J}_\kappa u^{m,l})\to \tilde{f}(t,\mathcal{J}_\kappa u^{\dagger,l})$ in $\mathcal{W}^N$ as $m\to \infty$ can be seen as follows. For that, we require the compact embedding $\mathcal{V}\hookdoubleheadrightarrow L^p(0,T;W^{\kappa,\hat{p}}(\Omega))$ which in fact by $\mathcal{V}=L^p(0,T;V)\cap W^{1,p,p}(0,T;\tilde{V})\hookdoubleheadrightarrow L^p(0,T;W^{\kappa,\hat{p}}(\Omega))$ follows in case $W^{\kappa,\hat{p}}(\Omega)\hookrightarrow \tilde{V}$ by the Aubin-Lions Lemma \cite[Lemma 7.7]{Roubíček2013}. For its application recall that $V, W^{\kappa, \hat{p}}(\Omega)$ are Banach spaces, $V$ reflexive and separable, $\tilde{V}$ a metrizable Hausdorff space, $V\hookdoubleheadrightarrow W^{\kappa,\hat{p}}(\Omega),$ $ W^{\kappa,\hat{p}}(\Omega)\hookrightarrow \tilde{V}$ and $1<p<\infty$. Otherwise in case $\tilde{V}\hookrightarrow W^{\kappa,\hat{p}}(\Omega)$ then $\mathcal{V}\subseteq L^p(0,T;V)\cap W^{1,p,p}(0,T;W^{\kappa,\hat{p}}(\Omega))$ and we can apply again Aubin-Lions' Lemma to obtain $\mathcal{V}\hookdoubleheadrightarrow L^p(0,T;W^{\kappa,\hat{p}}(\Omega))$.
	
	As a consequence, since $u^m\rightharpoonup u^\dagger$ in $\mathcal{V}^{N\times L}$ as $m\to \infty$ we derive that $u^{m,l}_n\to u^{\dagger,l}_n$ in $L^p(0,T;W^{\kappa, \hat{p}}(\Omega))$ strongly (w.l.o.g. for the whole sequence). Thus, it suffices to show that $\tilde{f}(t,\mathcal{J}_\kappa u^{m,l})\to \tilde{f}(t,\mathcal{J}_\kappa u^{\dagger,l})$ in $L^q(0,T;L^{\hat{p}}(\Omega))^N\hookrightarrow \mathcal{W}^N$ as $m\to \infty$. Due to $\tilde{f}\in W^{1,\infty}(U)^N$, it induces a well-defined Nemytskii operator $\tilde{f}$ with $[\tilde{f}(\mathcal{J}_\kappa u)](t,x)= \tilde{f}(t,\mathcal{J}_\kappa u(t,x))$ for $u\in L^p(0,T;L^{\hat{p}}(\Omega))^N$ and a.e. $(t,x)\in (0,T)\times \Omega$. Hence, we derive for $m$ large enough such that $(t,\mathcal{J}_\kappa u^{m,l})\in U$,
	\begin{align*}
		\Vert \tilde{f}(t,\mathcal{J}_\kappa u^{m,l})-\tilde{f}(t,\mathcal{J}_\kappa u^{\dagger,l})\Vert_{L^q(0,T;L^{\hat{p}}(\Omega))}\leq c\Vert \tilde{f}\Vert_{W^{1,\infty}(U)}\Vert u^{m,l}-u^{\dagger,l}\Vert_{L^q(0,T;W^{\kappa,\hat{p}}(\Omega))}
	\end{align*}
	for some constant $c>0$ and thus, the left-hand side approaches zero as $m\to \infty$.
	
	With this, identity \eqref{eq:hard_constraint} follows and by \eqref{Restim} together with uniqueness of the solution of \eqref{eq:limit_problem} that also $\tilde{\varphi}=\varphi^\dagger$ and $\tilde{f}=f^\dagger$, which concludes the proof.
\end{proof}

\section{Conclusions and outlook} In this work, we have considered the problem of learning structured models from data in an all-at-once framework. That is, the state, the nonlinearity and physical parameters, constituting the unknowns of a PDE system, are identified simultaneously based on noisy measured data of the state. It is shown that the main identification problem is well-posed in a general setup. The main results of this work are i) unique reconstructability of the state, the unknown nonlinearity and the parameters of the known physical term as regularization-minimizing solutions of a limit problem with full measurements, and ii) consistency, in the sense that reconstructions of these quantities based on incomplete, noisy measurements approximate the unique regularization-minimizing solutions in the limit. Here, the state is identified exactly, whereas the nonlinearity and the physical parameters are identified as regularization-minimizing solutions, which reproduce the data-generating dynamics along the visited states and are determined by the regularization elsewhere. For that, the class of functions used to approximate the unknown nonlinearity must meet a regularity and approximation capacity condition. These conditions are discussed and ensured for the case of fully connected feed-forward neural networks.

The results of this work provide a general framework that guarantees unique reconstructability in the limit of a practically useful all-at-once formulation in learning PDE models. This is particularly interesting because uniqueness of the quantities of interest is not given in general, but rather under certain conditions on the class of approximating functions and for certain regularization functionals. This provides an analysis-based guideline on which minimal conditions need to be ensured by practical implementations of PDE-based model learning setups in order to expect unique recovery of regularization-minimizing solutions in the limit. At the same time, the present analysis leaves several questions open, and we conclude by discussing those directions for future research which we consider the most relevant. These include the passage from asymptotic to quantitative statements, the computational realization of the proposed formulation, and the scope of the framework itself.
\paragraph{From asymptotic to quantitative statements.} The results of this work are asymptotic. Our main result in Theorem \ref{thm:uniqueness} couples the approximation scale, the measurement operators, the noise level and the regularization parameters only through the qualitative conditions $\lambda^m m^{-\beta q}=o(1)$, $\mu^m\delta(m)^r=o(1)$ and $\nu^m\psi(m)=o(1)$, and asserts convergence without quantifying it. The quantities that a quantitative counterpart would have to balance are, however, already present in the setting. These are given by the approximation rate $cm^{-\beta}$ and the parameter growth $\psi(m)$ of Assumption \ref{ass_uniqueness}, iv), the noise level $\delta(m)$ of \eqref{estim:noise_estimator}, and the deviation of the reduced measurement operators $K^m$ from $K^\dagger$ in \eqref{operator_conv_weak}. Replacing the limit by an error estimate in these quantities, with the parameters $\lambda^m,\mu^m,\nu^m$ chosen accordingly, would answer how many and which measurements, which network size and which noise level are needed for a prescribed reconstruction accuracy. A natural framework for this is the analysis of finite-dimensional variational regularization in \cite{poeschl2010}, which accounts simultaneously for a discretization of the unknown, an approximation of the forward operator and noisy data, and which yields rates in terms of Bregman distances under a sourcewise representation of the solution together with a nonlinearity condition. It covers nonseparable Banach spaces, among them $L^\infty$, which is the situation of the penalty used here. Its discretization error is measured along a recovery sequence for the exact solution, a role played in our setting by the sequence $(f_{\theta^m})_m$ of Assumption \ref{ass_uniqueness}, iv). Two points would require work. First, that assumption controls $\Vert f^\dagger-f_{\theta^m}\Vert_{L^\infty(U)}$ at the rate $cm^{-\beta}$, but requires of the gradient term only the convergence $\Vert \nabla f_{\theta^m}\Vert_{L^\infty(U)}\to\Vert \nabla f^\dagger\Vert_{L^\infty(U)}$, without a rate. A quantitative theory would need a rate here as well, hence approximation rates with respect to the $W^{1,\infty}$-norm for the class at hand, as they are available for neural networks in \cite{guehring20,raslan21}. Second, the discretization is not by a nested family of subspaces, but by the classes $\mathcal{F}^m$, which are in general neither linear nor convex, so that the corresponding arguments do not transfer verbatim. Convergence for a discretization by neural networks in an all-at-once identification problem closely related to ours, for Tikhonov as well as for projected Landweber regularization and as the network approximation error and the noise level tend to zero, is obtained in \cite{KaltNguy22}. Independently of the discretization, it is a quantitative version of the injectivity of $K^\dagger$ in Assumption \ref{ass_uniqueness}, v), that is, a conditional stability estimate in the sense of \cite{Werner_2020}, which would let the type of the measurements enter such a bound explicitly. Two limitations should be kept in mind. The error is necessarily measured against the regularization-minimizing solution of \eqref{eq:limit_problem} and not against the data-generating functions, and away from the visited states $\mathcal{S}$ the discrepancy between the two is governed by the regularization alone (see the discussion following Proposition \ref{prop:p_dagger}). An estimate on all of $U$ can therefore not be expected without quantifying how densely the data probe the state space. Since the reconstructed model is Lipschitz, however, a covering property of $\mathcal{S}$ in $U$ translates directly into a uniform bound, which gives the $W^{1,\infty}$-type regularization a quantitative role. Beyond this, our noise model \eqref{estim:noise_estimator} is deterministic, so that a statistical counterpart, with error bounds in expectation or with high probability as in the analysis of statistical inverse learning problems \cite{Blanchard2018,Helin2023,Rastogi2024}, is a further direction in which the dependence on the noise could be considered.

\paragraph{Computational and practical aspects.} Two ingredients of our formulation deserve attention when it is realized numerically. The first is the penalty $\Vert \nabla f_\theta\Vert_{L^\infty(U)}$. Approaches to estimating and controlling Lipschitz constants of neural networks were discussed in the introduction \cite{Latorre2020Lipschitz,Jordan2020,Gouk2021,huang2023on,Fazlab2019,Bungert2021}. The criterion our analysis requires from a surrogate is that it must have the property that its bounded sublevel sets are relatively compact in $\mathcal{C}(\overline{U})^N$. Surrogates that bound the Lipschitz constant from above, such as those built from layerwise bounds, have this property, whereas an estimate from below, such as a maximum over sampled gradients, does not in general. Identifying penalties that are implementable at scale and provably preserve this compactness is an open problem. The second is discretization. Our analysis is carried out in the continuous space-time setting, whereas any implementation also discretizes the state equation, so that the spatial and temporal discretization enters as a further quantity to be refined jointly with $m$, and the interplay of these two discretizations is not covered here. Beyond this, our results concern minimizers of the learning problems and not the output of the algorithms used to compute them. The convergence of the latter at a fixed scale $m$ is a question of interest in its own right. We refer to \cite{riedl2025} for a recent result in a related setting, and to \cite{Kaltenbacher16,KaltNguy22,NgocNguyen2025} for iterative and bi-level regularization approaches in all-at-once formulations.

\paragraph{Scope of the framework.} On the side of the model, the regularization is what selects among the models that reproduce the observed dynamics, so that it is the natural place for application-specific prior knowledge to enter. Penalties other than the one considered here are therefore of interest, be it the $W^{1,p}$-type variant with $p>D$, discussed in Subsection \ref{subsec:assumptions} but not carried out, or structural priors such as sparsity in a dictionary, invariances, or conservation properties of the learned dynamics \cite{morina2025b,Hansen2024,Mueller23,Sturm21}. On the side of the approximating class, our assumptions are modular. Any parametrized family for which Assumptions \ref{ass_param_app_class} and \ref{ass_uniqueness}, ii) and iv), can be verified is admissible, so that classes other than the neural networks treated in Subsection \ref{subsec:nn}, for instance parametrizations from operator learning \cite{Kovachki2024}, can be adapted within the same framework.
\section*{Funding} This research was funded in whole or in part by the Austrian Science Fund (FWF) 10.55776/F100800.

\section*{Data availability} No new data were generated or analysed in support of this research.
\pagebreak[2]
\appendix
\section{Neural networks}
\label{app:neural_networks} In the following section we will provide proofs for Proposition \ref{prop:ass3_nn}, treating Assumption \ref{ass_param_app_class}, and Proposition \ref{prop:ass5_nn}, addressing Assumption \ref{ass_uniqueness}, ii) and vi), both results dealing with neural networks as introduced in Definition \ref{def:nn}.

\subsection{Proof of Proposition \ref{prop:ass3_nn} (Assumption \ref{ass_param_app_class} for neural networks)}
\label{app_subsec:nn_prop1} We start by proving the first part of Proposition \ref{prop:ass3_nn} on the induction of well-defined Nemytskii operators.
\begin{lemma}
	\label{lem:nemytski_u_a} Let Assumption \ref{ass_init_set} hold true. Suppose that $\sigma\in \mathcal{C}(\mathbb{R},\mathbb{R})$ is Lipschitz continuous with constant $L_\sigma$ (w.l.o.g. $L_\sigma\geq1$). Then $\mathcal{N}_\theta: (0,T)\times(\otimes_{k=0}^\kappa \mathbb{R}^{p_k})^N\to \mathbb{R}$ induces a well-defined Nemytskii operator $\mathcal{N}_\theta: (\otimes_{k=0}^\kappa \mathcal{V}_k^\times)^N\to L^p(0,T;L^{\hat{q}}(\Omega))$ via $[\mathcal{N}_\theta( u)](t)=\mathcal{N}_\theta(u(t,\cdot))$. The same applies to $\mathcal{N}_\theta:(\otimes_{k=0}^\kappa \mathcal{V}_k^\times)^N\to\mathcal{W}$.
\end{lemma}
\begin{proof} First note that $\mathcal{N}_\theta$ is Lipschitz continuous with some Lipschitz constant 
	\begin{align}
		\label{Lipschitzconstant_NN} L_\theta \leq L_\sigma^{L-1}\prod_{l=1}^{L}\vert w^l\vert_\infty.
	\end{align}
	Hereinafter for $1\leq \alpha\leq \infty$ we denote by $\alpha^*$ the corresponding dual exponent defined by $\alpha^*:= \frac{\alpha}{\alpha-1}$ if $\alpha\in (1,\infty)$, $\alpha^*:=1$ if $\alpha=\infty$ and $\alpha^*=\infty$ if $\alpha=1$. Now fixing some $c\geq \Vert \mathcal{N}_\theta(0,0)\Vert_{L^{\hat{q}}(\Omega)}$ we have for $u=((u_1^k)_k, \dots, (u_N^k)_k)\in (\otimes_{k=0}^\kappa \mathcal{V}_k^\times)^N$ and a.e. $t\in(0,T)$ that 
	\begin{align*}
		\Vert \mathcal{N}_\theta(t,u(t,\cdot))\Vert_{L^{\hat{q}}(\Omega)}&\leq \Vert \mathcal{N}_\theta(0,0)\Vert_{L^{\hat{q}}(\Omega)}+\Vert \mathcal{N}_\theta(t,u(t,\cdot))-\mathcal{N}_\theta(0,0)\Vert_{L^{\hat{q}}(\Omega)}\\
		&\leq c+\sup_{\substack{\varphi\in L^{\hat{q}^*}(\Omega),\\ \Vert \varphi\Vert_{L^{\hat{q}^*}(\Omega)}\leq 1}}\langle \mathcal{N}_\theta(t,u(t,\cdot))-\mathcal{N}_\theta(0,0),\varphi\rangle_{L^{\hat{q}}(\Omega),L^{\hat{q}^*}(\Omega)}\\
		&\leq c+\sup_{\substack{\varphi\in L^{\hat{q}^*}(\Omega),\\ \Vert \varphi\Vert_{L^{\hat{q}^*}(\Omega)}\leq 1}}\int_\Omega\vert\mathcal{N}_\theta(t,u(t,x))-\mathcal{N}_\theta(0,0)\vert\vert\varphi(x)\vert\dx x\\
		&\leq c +L_\theta\sup_{\substack{\varphi\in L^{\hat{q}^*}(\Omega),\\ \Vert \varphi\Vert_{L^{\hat{q}^*}(\Omega)}\leq 1}}\int_\Omega(T+\vert u(t,x)\vert_1)\vert\varphi(x)\vert\dx x\\
		&\leq c+L_\theta(T\vert\Omega\vert^{1/\hat{q}}+\sum_{\substack{1\leq n\leq N\\ 0\leq k\leq \kappa}}\Vert u_n^k(t)\Vert_{L^{\hat{q}}(\Omega)^{p_k}})
	\end{align*}
	where the product norms correspond to the respective $\Vert \cdot \Vert_1$-norm. As $V\hookrightarrow L^{\hat{p}}(\Omega)\hookrightarrow L^{\hat{q}}(\Omega)$ and $\Vert u_n^0(t)\Vert_{V}<\infty$ for a.e. $t\in (0,T)$ due to $(u_n^0)_n\in\mathcal{V}^N\subseteq L^p(0,T;V)^N$ it holds true that $\Vert u_n^0(t)\Vert_{L^{\hat{q}}(\Omega)}<\infty$ for $1\leq n\leq N$. The embedding $V_k\hookrightarrow L^{\hat{q}}(\Omega)$ implies $V_k^\times \hookrightarrow L^{\hat{q}}(\Omega)^{p_k}$ by which we may infer again that $\Vert u_n^k(t)\Vert_{L^{\hat{q}}(\Omega)^{p_k}}<\infty$ for a.e. $t\in (0,T)$ as $u_n^k\in \mathcal{V}^\times_k= L^p(0,T; V_k^\times)$ for $1\leq n\leq N$, $1\leq k\leq \kappa$. Thus, it holds for a.e. $t\in(0,T)$ that $\mathcal{N}_\theta(t, u(t,\cdot))\in L^{\hat{q}}(\Omega)$ which is separable. Now $t\mapsto \mathcal{N}_\theta(t,u(t,\cdot))$ is weakly measurable, i.e.,
	\[ t\mapsto \int_\Omega \mathcal{N}_\theta(t,u(t,x))w(x)\dx x
	\] is Lebesgue measurable for all $w\in L^{\hat{q}^*}(\Omega)$ which follows by standard arguments as $\mathcal{N}_\theta$ is continuous, $w, u(t,\cdot)$ Lebesgue measurable and measurability is preserved under integration. Employing Pettis Theorem (see \cite[Theorem 1.34]{Roubíček2013}) we obtain that $t\mapsto \mathcal{N}_\theta(t,u(t,\cdot))\in L^{\hat{q}}(\Omega)$ is Bochner measurable. Similarly as before one can show that for $u=((u_1^k)_k, \dots, (u_N^k)_k)\in (\otimes_{k=0}^\kappa \mathcal{V}_k^\times)^N$ it holds for some $\tilde{c}>0$,
	\begin{equation}
		\label{Bochnerinteg}
		\Vert\mathcal{N}_\theta( u)\Vert_{L^p(0,T;L^{\hat{q}}(\Omega))}\leq \tilde{c}(1+\sum_{\substack{1\leq n\leq N\\ 0\leq k\leq \kappa}}\Vert u_n^k\Vert_{L^p(0,T;L^{\hat{q}}(\Omega)^{p_k})})\leq \tilde{c}(1+\sum_{\substack{1\leq n\leq N\\ 0\leq k\leq \kappa}}\Vert u_n^k\Vert_{\mathcal{V}_k^\times})<\infty
	\end{equation}
	again by $V_k\hookrightarrow L^{\hat{q}}(\Omega)$ using the isomorphism $L^p(0,T;L^{\hat{q}}(\Omega))^{p_k}\cong L^p(0,T;L^{\hat{q}}(\Omega)^{p_k})$ for $0\leq k\leq \kappa$. Finally, we derive by separability of $L^{\hat{q}}(\Omega)$ that $\mathcal{N}_\theta(u)$ is Bochner integrable (see \cite[Section 1.5]{Roubíček2013}) and by $p\geq q$ together with $L^{\hat{q}}(\Omega)\hookrightarrow W$ that also the Nemytskii operator $\mathcal{N}_\theta:(\otimes_{k=0}^\kappa \mathcal{V}_k^\times)^N\to\mathcal{W}$ is well-defined.
\end{proof}
The next result addresses the remaining part of Proposition \ref{prop:ass3_nn} on continuity. 
\begin{lemma}[Strong-strong continuity of $\mathcal{N}$]
	\label{lem:weak_cont_n} Assume that $\sigma\in \mathcal{C}(\mathbb{R},\mathbb{R})$ is Lipschitz continuous with Lipschitz constant $L_\sigma$ (w.l.o.g. $L_\sigma\geq1$). Then under Assumption \ref{ass_init_set}, $\mathcal{N}:\Theta\times (\otimes_{k=0}^\kappa L^p(0,T;L^{\hat{p}}(\Omega)^{p_k}))^N\to L^q(0,T;L^{\hat{q}}(\Omega))$, $(\theta, v)\mapsto\mathcal{N}_\theta(v)$ is strongly-strongly continuous.
\end{lemma}
\begin{proof}
	By analogous reasoning as in Lemma \ref{lem:nemytski_u_a} the Nemytskii operator $\mathcal{N}$ in the assertions of this lemma is well-defined. 
	
	Let $(\theta^m, u^m)\to (\theta, u)$ in $\Theta\times (\otimes_{k=0}^\kappa L^p(0,T;L^{\hat{p}}(\Omega)^{p_k}))^N$ as $m\to\infty$. We aim to show that $\mathcal{N}(\theta^m,u^m)\to \mathcal{N}(\theta, u)$ strongly in $L^q(0,T;L^{\hat{q}}(\Omega))$ as $m\to \infty$.\\
	
	Note that for $z\in \mathbb{R}^{1+N\sum_{k=0}^{\kappa}p_k}$ it holds
	\begin{align*}
		\mathcal{N}(\theta,z) &= (L_{\theta_L}\circ\dots\circ L_{\theta_1})(z),\\
		\mathcal{N}(\theta^m, z)&=(L_{\theta^m_L}\circ\dots\circ L_{\theta^m_1})(z)
	\end{align*}
	and define for $1\leq s\leq L-1$ the feed-forward neural networks  $\mathcal{N}_s(\theta^m,\theta, z)$  by
	\begin{align*}
		\mathcal{N}_s(\theta^m,\theta, z)&=(L_{\theta_L}\circ\dots\circ L_{\theta_{L-s+1}}\circ L_{\theta^m_{L-s}}\circ \dots\circ L_{\theta_1^m})(z),\\
		\mathcal{N}_0(\theta^m,\theta, z)&=\mathcal{N}(\theta^m, z),\\
		\mathcal{N}_L(\theta^m, \theta, z)&=\mathcal{N}(\theta, z).
	\end{align*}
	By $\theta^m\to \theta$ as $m\to \infty$ and continuity of $\theta^m \mapsto (L_{\theta_s^m}\circ\dots\circ L_{\theta_1^m})(0)$ for all $s=1,\ldots,L$ there exists $C>0$, used generically in the estimations below, with
	\[ \vert \mathfrak{L}_s^m (0)\vert_\infty < C,~ ~ ~ \forall 1\leq s\leq L,
	\] for sufficiently large $m\in \mathbb{N}$, where we set
	\[ \mathfrak{L}_s^m = L_{\theta_s^m}\circ\dots\circ L_{\theta_1^m}
	\] for $1\leq s\leq L$ and $\mathfrak{L}_0^m=\id$ the identity map. Recall that we aim to estimate
	\[ \Vert \mathcal{N}(\theta^m,u^m)-\mathcal{N}(\theta, u)\Vert_{L^q(0,T;L^{\hat{q}}(\Omega))}.
	\] For $M>0$ such that $L_\sigma^{L-1}\prod_{l=1}^L(\Vert w^l\Vert_\infty+1)<M$, we have for a.e. $(t,x)\in (0,T)\times \Omega$ (under abuse of notation omitting the dependence of $u,u^m$ on $(t,x)$) that $\vert \mathcal{N}(\theta^m,t,u^m)-\mathcal{N}(\theta,t,u)\vert$ is bounded by
	\begin{multline}
		\label{nn_ineq1}
		\vert \mathcal{N}(\theta^m,t,u^m)-\mathcal{N}(\theta^m, t, u)\vert+\vert \mathcal{N}(\theta^m,t,u)-\mathcal{N}(\theta,t, u)\vert \\
		\leq M\vert u-u^m\vert_1 +\sum_{s=0}^{L-1}\vert\mathcal{N}_{s+1}(\theta^m,\theta,t,u)-\mathcal{N}_{s}(\theta^m, \theta,t,u)\vert.
	\end{multline}
	For the second term estimate first $\vert\mathcal{N}_{s+1}(\theta^m,\theta,t,u)-\mathcal{N}_{s}(\theta^m, \theta,t,u)\vert$ by
	\begin{align}
		\label{nn_ineq3}
		\notag&= \vert (L_{\theta_L}\circ\dots\circ L_{\theta_{L-s}}\circ \mathfrak{L}_{L-s-1}^m)(t,u)-(L_{\theta_L}\circ\dots\circ L_{\theta_{L-s+1}}\circ L_{\theta_{L-s}^m}\circ \mathfrak{L}_{L-s-1}^m)(t,u)\vert\\
		\notag&\leq \left(L_\sigma^{s-1}\prod_{l=L-s+1}^{L}\vert w^l\vert_\infty\right)\vert(L_{\theta_{L-s}}\circ \mathfrak{L}_{L-s-1}^m)(t,u)-(L_{\theta_{L-s}^m}\circ \mathfrak{L}_{L-s-1}^m)(t,u)\vert_\infty\\
		\notag&\leq \left(L_\sigma^{s}\prod_{l=L-s+1}^{L}\vert w^l\vert_\infty\right) \left[\vert w^{L-s}-w_m^{L-s}\vert_\infty\vert(\mathfrak{L}_{L-s-1}^m)(t,u)\vert_\infty+\vert \beta^{L-s}-\beta_m^{L-s}\vert_\infty\right]\\
		\notag&\leq \left(L_\sigma^s\prod_{l=L-s+1}^{L}\vert w^l\vert_\infty\right) \vert \theta^{L-s}-\theta_m^{L-s}\vert_\infty\left(\vert(\mathfrak{L}_{L-s-1}^m)(t,u)-(\mathfrak{L}_{L-s-1}^m)(0)\vert_\infty +C\right)\\
		\notag&\leq \left(L_\sigma^s\prod_{l=L-s+1}^{L}\vert w^l\vert_\infty\right) \vert \theta^{L-s}-\theta_m^{L-s}\vert_\infty\left(L_\sigma^{L-s-1}\prod_{l=1}^{L-s-1}\vert w_m^l\vert_\infty(T+\vert u\vert_1) +C\right)\\
		&\leq M\vert \theta^{L-s}-\theta_m^{L-s}\vert_\infty\left(\vert u\vert_1 +C\right).
	\end{align}
	Combining this with \eqref{nn_ineq1} it follows that
	\begin{align}
		\label{Ndiffestimation} \vert \mathcal{N}(\theta^m,t,u^m)-\mathcal{N}(\theta, t,u)\vert\leq M\vert u-u^m\vert_1
		+M(\vert u\vert_1+C)\sum_{s=1}^{L}\vert \theta^s-\theta^s_m\vert_\infty.
	\end{align}
	To estimate $\Vert \mathcal{N}(\theta^m,u^m)-\mathcal{N}(\theta, u)\Vert_{L^q(0,T;L^{\hat{q}}(\Omega))}$ note that for $w^*\in L^{q^*}(0,T;L^{\hat{q}^*}(\Omega))$ with $\Vert w^*\Vert_{L^{q^*}(0,T;L^{\hat{q}^*}(\Omega))}\leq 1$ it holds for some generic constant $\tilde{C}>0$ by successively employing the upper bound \eqref{Ndiffestimation}, Minkowski's inequality in $L^{\hat{q}}(\Omega)$ and Hölder's inequality in time with $p,p^*$ that
	\begin{align*}
		&\int_0^T\Vert \mathcal{N}(\theta^m,t,u^m(t,\cdot))-\mathcal{N}(\theta,t, u(t,\cdot))\Vert_{L^{\hat{q}}(\Omega)}\Vert w^*(t)\Vert_{L^{\hat{q}^*}(\Omega)}\dx t\\
		\leq &\tilde{C}\int_0^T\bigg\Vert\bigg[\vert u(t,\cdot)-u^m(t,\cdot)\vert_1+(\vert u(t,\cdot)\vert_1+C)\sum_{s=1}^{L}\vert \theta^s-\theta^s_m\vert_\infty\bigg]\bigg\Vert_{L^{\hat{q}}(\Omega)}\Vert w^*(t)\Vert_{L^{\hat{q}^*}(\Omega)}\dx t\\
		\leq& \tilde{C}\bigg[\Vert u-u^m\Vert_{(\otimes_{k=0}^\kappa L^p(0,T;L^{\hat{q}}(\Omega)^{p_k}))^N}+(\Vert u\Vert_{(\otimes_{k=0}^\kappa L^p(0,T;L^{\hat{q}}(\Omega)^{p_k}))^N}+C)\sum_{s=1}^{L}\vert \theta^s-\theta^s_m\vert_\infty\bigg]
	\end{align*}
	due to $\Vert w^*\Vert_{L^{p^*}(0,T;L^{\hat{q}^*}(\Omega))}\leq 1$ as $p\geq q$ and $L^{\hat{p}}(\Omega)\hookrightarrow L^{\hat{q}}(\Omega)$. As the right-hand side of the previous estimation is independent of $w^*$ we obtain that
	\begin{multline*}
		\Vert \mathcal{N}(\theta^m,u^m)-\mathcal{N}(\theta, u)\Vert_{L^q(0,T;L^{\hat{q}}(\Omega))}\\
		\leq \tilde{C}\bigg[\Vert u-u^m\Vert_{(\otimes_{k=0}^\kappa L^p(0,T;L^{\hat{p}}(\Omega)^{p_k}))^N} +(\Vert u\Vert_{(\otimes_{k=0}^\kappa L^p(0,T;L^{\hat{p}}(\Omega)^{p_k}))^N}+C)\sum_{s=1}^{L}\vert \theta^s-\theta^s_m\vert_\infty\bigg].
	\end{multline*}
	Now by $u_m\to u$ in $(\otimes_{k=0}^\kappa L^p(0,T;L^{\hat{p}}(\Omega)^{p_k}))^N$, $\Vert u\Vert_{(\otimes_{k=0}^\kappa L^p(0,T;L^{\hat{p}}(\Omega)^{p_k}))^N}<\infty$ and $\theta_m\to\theta$ as $m\to \infty$ we derive that the last argument converges to zero as $m\to \infty$. 
	
	Thus, it holds 
	\[ \mathcal{N}(\theta^m,u^m)\to\mathcal{N}(\theta, u) ~ ~ ~ \text{as} ~ ~ ~ m\to\infty ~ ~ ~ \text{in} ~ L^q(0,T;L^{\hat{q}}(\Omega))
	\] yielding strong-strong continuity of the joint operator $\mathcal{N}$ as claimed.
\end{proof}
Combining Lemma \ref{lem:nemytski_u_a} and Lemma \ref{lem:weak_cont_n} concludes the result in Proposition \ref{prop:ass3_nn}.
\subsection{Proof of Proposition \ref{prop:ass5_nn} (Assumption \ref{ass_uniqueness} for neural networks)}
\label{app_subsec:nn_prop2} First, we verify $W^{1,\infty}_{loc}$-regularity of the classes $\mathcal{F}_n^m$.
\begin{lemma}
	\label{lem:W1oolocreg} Assume that $\sigma\in \mathcal{C}(\mathbb{R},\mathbb{R})$ is locally Lipschitz continuous and let $(\mathcal{F}_n^m)_n$ be given as in Definition \ref{def:model_fi}. Then $\mathcal{F}_n^m\subseteq W^{1,\infty}_\loc(\mathbb{R}^{1+N\sum_{k=0}^\kappa p_k})$ for $1\leq n\leq N$, $m\in \mathbb{N}$.
\end{lemma}
\begin{proof}
	Let $f\in \mathcal{F}_n^m$ for some $1\leq n\leq N$ and $m\in \mathbb{N}$.
	Since the activation function $\sigma$ is supposed to be locally Lipschitz continuous, $f$ is also locally Lipschitz continuous.
	This follows from the fact that continuous functions map bounded sets to bounded sets and thus, recalling the layer-wise structure of $f$ in Definition \ref{def:nn}, for any bounded $U\subseteq \mathbb{R}^{1+N\sum_{k=0}^\kappa p_k}$ the instance $f$ is Lipschitz continuous on $U$ with a constant depending on local Lipschitz constants of $\sigma$ and norms of the weights. Rademacher's Theorem finally yields $f\in W^{1,\infty}(U)$ and thus the assertion of the lemma.
\end{proof}
The next result shows for bounded $U\subseteq \mathbb{R}^D$ strong-strong continuity of the map $\Theta\ni \theta\mapsto \nabla\mathcal{N}_\theta\in L^\infty(U)$ for $\mathcal{C}^1$-regular and Lipschitz continuous activation functions, in particular implying weak lower semicontinuity as claimed in Proposition \ref{prop:ass5_nn}.
\begin{lemma}
	\label{neural_lsc} Let $U\subseteq \mathbb{R}^D$ be bounded. Furthermore, let the activation function $\sigma$ of the class of parameterized approximation functions fulfill $\sigma \in \mathcal{C}^1(\mathbb{R},\mathbb{R})$. Then the map
	\[ \Theta\ni \theta\mapsto \nabla\mathcal{N}_\theta\in L^\infty(U)
	\] is strongly-strongly continuous.
\end{lemma}
\begin{proof}
	We consider first the case that $\sigma\in \mathcal{C}^1(\mathbb{R},\mathbb{R})$ is Lipschitz continuous with constant $L_\sigma$ (w.l.o.g. $L_\sigma\geq 1$). Let $(\theta^m)_m\subseteq \Theta$ such that $\theta^m\to \theta\in \Theta$ as $m\to \infty$. Maintaining the notation in the proof of Lemma \ref{lem:weak_cont_n} we further set for $1\leq k\leq l\leq L$
	\[ \mathfrak{L}_{k,l}=L_{\theta_l}\circ\dots\circ L_{\theta_k}
	\] with $\mathfrak{L}_{k,l}=\id$ the identity map for $k>l$. Then we obtain for fixed $z\in U$ that
	\begin{multline*}
		\vert \nabla\mathcal{N}_{\theta^m}(z)-\nabla\mathcal{N}_\theta(z)\vert_\infty\leq \sum_{s=0}^{L-1}\vert\nabla \mathcal{N}_{s+1}(\theta^m, \theta, z)-\nabla\mathcal{N}_s(\theta^m, \theta, z)\vert_\infty\\
		=\sum_{s=0}^{L-1}\vert \nabla[(\mathfrak{L}_{L-s,L}\circ\mathfrak{L}_{L-s-1}^m)(z)]-\nabla[(\mathfrak{L}_{L-s+1,L}\circ\mathfrak{L}_{L-s}^m)(z)]\vert_\infty.
	\end{multline*}
	We consider a summand of the last sum for fixed $0\leq s\leq L-1$ and show convergence to zero for $m\to \infty$. For that we introduce the following simplifying notation for products of matrices $C_0\cdot\hdots\cdot C_n$ for $n\in \mathbb{N}$ where the row and column dimensions fit for the product to make sense, by
	\[ \mathcal{P}_{l=0}^n C_l := C_0\cdot\hdots\cdot C_n.
	\] Furthermore, we set $\mathcal{P}_{l=k}^mC_l:=1$ for $k>m$. Defining
	\[ A_{l,s}^m(z) = \sigma'(w^{L-l-1}(\mathfrak{L}_{L-s,L-l-2}\circ \mathfrak{L}_{L-s-1}^m)(z)+b^{L-l-1})w^{L-l-1} ~ ~ \text{for } ~ 0\leq l\leq s-1,
	\]
	\[ B_{l,s}^m(z) = \sigma'(w^{L-l-1}(\mathfrak{L}_{L-s+1,L-l-2}\circ \mathfrak{L}_{L-s}^m)(z)+b^{L-l-1})w^{L-l-1} ~ ~ \text{for } ~ 0\leq l\leq s-2,
	\]
	\[ \text{and} ~ ~ B_{s-1,s}^m(z) = \sigma'(w_m^{L-s}\mathfrak{L}_{L-s-1}^m(z)+b_m^{L-s})w_m^{L-s}
	\] for $z\in U$, we derive by the chain rule that $$\vert\nabla[(\mathfrak{L}_{L-s,L}\circ\mathfrak{L}_{L-s-1}^m)(z)]-\nabla[(\mathfrak{L}_{L-s+1,L}\circ\mathfrak{L}_{L-s}^m)(z)]\vert_\infty$$ can be estimated by
	\begin{align}
		\label{estimation_AB}
		\notag&=\vert w^L(\mathcal{P}_{l=0}^{s-1}A_{l,s}^m(z)-\mathcal{P}_{l=0}^{s-1}B_{l,s}^m(z))\nabla[\mathfrak{L}_{L-s-1}^m(z)]\vert_\infty\\
		\notag&\leq \vert w^L\vert_\infty\vert\nabla[\mathfrak{L}_{L-s-1}^m(z)]\vert_\infty\sum_{r=0}^{s-1}\vert(\mathcal{P}_{l=0}^{r-1}B_{l,s}^m(z))(A_{r,s}^m(z)-B_{r,s}^m(z))(\mathcal{P}_{l=r+1}^{s-1}A_{l,s}^m(z))\vert_\infty\\
		&\leq \vert w^L\vert_\infty\vert\nabla[\mathfrak{L}_{L-s-1}^m(z)]\vert_\infty\sum_{r=0}^{s-1}(\prod_{l=0}^{r-1}\vert B_{l,s}^m(z)\vert_\infty)\vert A_{r,s}^m(z)-B_{r,s}^m(z)\vert_\infty(\prod_{l=r+1}^{s-1}\vert A_{l,s}^m(z)\vert_\infty).
	\end{align}
	Let $M>0$ such that $L_\sigma^{L-1}\prod_{l=1}^L(\vert w^l\vert_\infty+1)<M$ and $m\in\mathbb{N}$ sufficiently large such that $\vert w^l_m-w^l\vert_\infty<1$ for $1\leq l\leq L$ which is possible due to $\theta^m\to \theta$ as $m\to \infty$. As $\vert A_{l,s}^m(z)\vert_\infty,\vert B_{l,s}^m(z)\vert_\infty\leq L_\sigma M$ for $0\leq s\leq L-1, 1\leq l\leq s-1$, $\vert w^L\vert_\infty<M$ and 
	\[ \nabla[\mathfrak{L}_{L-s-1}^m(z)] = \mathcal{P}_{l=0}^{L-s-2}\sigma'(w^{L-s-l-1}_m\mathfrak{L}_{L-s-l-2}^m(z)+b^{L-s-l-1}_m)w^{L-s-l-1}_m
	\] by the chain rule, implying $\vert\nabla[\mathfrak{L}_{L-s-1}^m(z)]\vert_\infty\leq L_\sigma^{L-s-1} M$, it remains to show that
	\begin{align}
		\label{ABconvergence} \lim_{m\to \infty}\vert A_{r,s}^m(z)-B_{r,s}^m(z)\vert_\infty =0.
	\end{align}
	This follows as $\theta^m\to \theta$, $\mathfrak{L}_{L-s,L-l-2}\circ \mathfrak{L}_{L-s-1}^m\to \mathfrak{L}_{1, L-l-2}$ in $L^\infty(U)$ for $0\leq l\leq s-1$ and $\mathfrak{L}_{L-s+1,L-l-2}\circ \mathfrak{L}_{L-s}^m\to \mathfrak{L}_{1, L-l-2}$ in $L^\infty(U)$ for $0\leq l\leq s-2$ as $m\to \infty$ by similar considerations as in \eqref{nn_ineq3} due to continuity of $\sigma'$. As the convergence in \eqref{ABconvergence} holds uniformly for $z\in U$ we recover the assertion of the lemma that $\nabla \mathcal{N}_{\theta^m}\to \nabla\mathcal{N}_\theta\in L^\infty(U)$ as $m\to \infty$.\\
	
	In case $\sigma\in \mathcal{C}^1(\mathbb{R},\mathbb{R})$ is not Lipschitz continuous (such as ReQU), the result follows by a similar strategy as above. An adaption concerns uniform boundedness of the $\sigma'$ terms in $A_{l,s}^m(z), B_{l,s}^m(z)$ for $z\in U$, which follows from uniform convergence $\mathfrak{L}_{L-s+1,L-l-2}\circ \mathfrak{L}_{L-s}^m\to \mathfrak{L}_{1, L-l-2}$ in $L^\infty(U)$ as $m\to \infty$ and the fact that the latter map $U$ to bounded sets.
\end{proof}
Finally, for merely locally Lipschitz continuous activations, such as the Rectified Linear Unit, we recover for bounded and open $U\subseteq\mathbb{R}^D$ weak lower semicontinuity of the map $\Theta\ni \theta\mapsto \Vert\nabla\mathcal{N}_\theta\Vert_{L^\infty(U)}$ as claimed in Proposition \ref{prop:ass5_nn}. Note that the argument below does not differentiate the activation function and is therefore independent of the one used for Lemma \ref{neural_lsc}.

\begin{lemma}
	\label{neural_lsc2} Let $U\subseteq \mathbb{R}^D$ be bounded and open. Furthermore, let the activation function $\sigma$ of the class of parameterized approximation functions be locally Lipschitz continuous (e.g., ReLU). Then for $(\theta^m)_m\subseteq \Theta$ with $\theta^m\to \theta\in \Theta$ as $m\to \infty$ it holds
	\[ \Vert\nabla\mathcal{N}_\theta\Vert_{L^\infty(U)}\leq \liminf_{m\to\infty}\Vert\nabla\mathcal{N}_{\theta^m}\Vert_{L^\infty(U)}.
	\]
\end{lemma}
\begin{proof}
	Let $(\theta^m)_m\subseteq \Theta$ with $\theta^m\to \theta\in \Theta$ as $m\to \infty$ and set
	\[ \Lambda:=\liminf_{m\to \infty}\Vert\nabla\mathcal{N}_{\theta^m}\Vert_{L^\infty(U)},
	\] where we may assume $\Lambda<\infty$ as otherwise there is nothing to show. Throughout, $\vert x\vert_1=\sum_{i=1}^D\vert x_i\vert$ denotes the $\ell^1$-norm of $x\in\mathbb{R}^D$, which is the norm dual to the norm $\vert\cdot\vert_\infty$ used on gradients, whereas $\vert\cdot\vert$ denotes the Euclidean norm, with respect to which all balls below are understood. Choose $(\theta^{m_j})_j$ with $\Vert\nabla\mathcal{N}_{\theta^{m_j}}\Vert_{L^\infty(U)}\to \Lambda$ as $j\to \infty$. Let $U_0\subseteq \mathbb{R}^D$ be an open ball with center $z_0\in U$ and radius $r>0$ such that $\overline{U_0}\subseteq U$, and for $\epsilon>0$ denote by $U_0^\epsilon$ the concentric open ball of radius $r+\epsilon$. We first claim that
	\begin{align}
		\label{lipschitz_on_balls} \vert \mathcal{N}_\eta(x)-\mathcal{N}_\eta(y)\vert\leq \Vert \nabla\mathcal{N}_\eta\Vert_{L^\infty(U)}\vert x-y\vert_1\quad \text{for all}~ x,y\in U_0 ~\text{and}~ \eta\in \Theta.
	\end{align}
	To see this, let $\rho\in \mathcal{C}^\infty_c(\mathbb{R}^D)$ with $\rho\geq 0$, $\supp \rho\subseteq \left\{z\in\mathbb{R}^D~\vert~\vert z\vert\leq 1\right\}$ and $\int_{\mathbb{R}^D}\rho(z)\dx z=1$, and set $\rho_\epsilon(z):=\epsilon^{-D}\rho(z/\epsilon)$. As $\overline{U_0}$ is a compact subset of the open set $U$, we may fix $\epsilon_0>0$ such that $U_0^{\epsilon_0}\subseteq U$, and we consider $0<\epsilon\leq \epsilon_0$ in the following. Since $\sigma$ is locally Lipschitz continuous, Lemma \ref{lem:W1oolocreg} yields $\mathcal{N}_\eta\in W^{1,\infty}_\loc(\mathbb{R}^D)$, so that $\mathcal{N}_\eta*\rho_\epsilon$ is smooth with $\nabla(\mathcal{N}_\eta*\rho_\epsilon)=(\nabla\mathcal{N}_\eta)*\rho_\epsilon$ on $U_0$. Thus, as $z-w\in U_0^\epsilon\subseteq U$ for $z\in U_0$ and $\vert w\vert\leq \epsilon$,
	\begin{align}
		\label{mollified_gradient_bound} \vert \nabla(\mathcal{N}_\eta*\rho_\epsilon)(z)\vert_\infty\leq \int_{\vert w\vert\leq \epsilon}\rho_\epsilon(w)\vert \nabla\mathcal{N}_\eta(z-w)\vert_\infty\dx w\leq \Vert \nabla\mathcal{N}_\eta\Vert_{L^\infty(U)}
	\end{align}
	for $z\in U_0$. As $U_0$ is convex we have $y+t(x-y)\in U_0$ for $x,y\in U_0$ and $t\in[0,1]$. Hence, the fundamental theorem of calculus, \eqref{mollified_gradient_bound} and Hölder's inequality for the dual pair $\vert\cdot\vert_\infty, \vert\cdot\vert_1$ yield
	\begin{align}
		\label{estimate:convolved} \vert (\mathcal{N}_\eta*\rho_\epsilon)(x)-(\mathcal{N}_\eta*\rho_\epsilon)(y)\vert=\left\vert\int_0^1 \nabla(\mathcal{N}_\eta*\rho_\epsilon)(y+t(x-y))\cdot (x-y)\dx t\right\vert\leq \Vert \nabla\mathcal{N}_\eta\Vert_{L^\infty(U)}\vert x-y\vert_1.
	\end{align}
	Finally, denoting by $L_\eta$ a Lipschitz constant of $\mathcal{N}_\eta$ on the bounded set $U_0^{\epsilon_0}$, which exists by Lemma~\ref{lem:W1oolocreg} and which is independent of $\epsilon$, we obtain $$\vert (\mathcal{N}_\eta*\rho_\epsilon)(z)-\mathcal{N}_\eta(z)\vert\leq \int_{\vert w\vert\leq\epsilon}\rho_\epsilon(w)\vert \mathcal{N}_\eta(z-w)-\mathcal{N}_\eta(z)\vert\dx w\leq L_\eta\epsilon$$ for $z\in U_0$. Thus, $\mathcal{N}_\eta~*\rho_\epsilon\to \mathcal{N}_\eta$ uniformly on $U_0$ as $\epsilon\to 0$, and together with \eqref{estimate:convolved}, the claim in \eqref{lipschitz_on_balls} follows.

	Now fix $x,y\in U_0$. The map $\Theta\ni\eta\mapsto \mathcal{N}_\eta(x)\in\mathbb{R}$ is continuous, being a finite composition of maps that are affine in the parameters and of the continuous activation function $\sigma$. Applying \eqref{lipschitz_on_balls} to $\eta=\theta^{m_j}$ and letting $j\to \infty$ therefore yields
	\[ \vert \mathcal{N}_\theta(x)-\mathcal{N}_\theta(y)\vert=\lim_{j\to \infty}\vert \mathcal{N}_{\theta^{m_j}}(x)-\mathcal{N}_{\theta^{m_j}}(y)\vert\leq \Lambda\vert x-y\vert_1 .
	\] Hence $\mathcal{N}_\theta$ is Lipschitz continuous on $U_0$ with constant $\Lambda$ with respect to $\vert\cdot\vert_1$. Consequently, if $z\in U_0$ is a point of differentiability of $\mathcal{N}_\theta$, then for $1\leq i\leq D$ and $h\neq 0$ small enough that $z+he_i\in U_0$ it holds $\vert \mathcal{N}_\theta(z+he_i)-\mathcal{N}_\theta(z)\vert\leq \Lambda\vert h\vert$, since $\vert e_i\vert_1=1$. Letting $h\to 0$ gives $\vert\partial_i\mathcal{N}_\theta(z)\vert\leq \Lambda$ for $1\leq i\leq D$ and therefore $\vert\nabla\mathcal{N}_\theta(z)\vert_\infty\leq \Lambda$.

	Finally, $\mathcal{N}_\theta$ is locally Lipschitz continuous and hence differentiable almost everywhere by Rademacher's Theorem, and as $U$ is open every $z\in U$ is contained in some open ball $U_0$ with $\overline{U_0}\subseteq U$. Thus $\vert\nabla\mathcal{N}_\theta(z)\vert_\infty\leq \Lambda$ for almost every $z\in U$, and taking the essential supremum over $z\in U$ concludes the assertion of the lemma.
\end{proof}
Combining the Lemmata \ref{lem:W1oolocreg} and \ref{neural_lsc2} concludes the result in Proposition \ref{prop:ass5_nn}. Note that for $\mathcal{C}^1$-regular activation functions, Lemma \ref{neural_lsc} yields the stronger assertion that $\Theta\ni\theta\mapsto \nabla\mathcal{N}_\theta\in L^\infty(U)$ is even strongly-strongly continuous, and this for arbitrary bounded $U\subseteq\mathbb{R}^D$.
\subsection{Lifting technique}
\label{app_subsec:lifting} In this subsection we discuss a lifting approach that shows how approximation results, such as in Proposition \ref{prop:morina_holler}, covering \eqref{f_approximation1} can be used to infer Assumption \ref{ass_uniqueness}, iv), i.e., also achieve the convergence $\Vert \nabla f_{\theta^m}\Vert_{L^\infty(U)}\to \Vert \nabla f\Vert_{L^\infty(U)}$ as $m\to \infty$. For that, one needs to impose higher regularity on $f$, such as $W^{2,\infty}$- or $\mathcal{C}^2$-regularity and assume that the bounded domain $U\subseteq \mathbb{R}^D$ of functions in $\mathcal{F}_n^m$ is star-shaped with some center given by $x_0\in U$. The basic idea is to approximate the partial derivatives of $f$ by the approximation result at hand (such as in Proposition \ref{prop:morina_holler})  and lift the approximation property to the function. More concretely, let $g_{\tilde{\theta}^m}$ approximate $\nabla f$ uniformly on $U$ by rate $\beta>0$ and $f_{\tilde{\eta}^m}$ the function $f$ by rate $\gamma>0$. Then with $\diam U$ denoting the diameter of $U$ we have
\begin{align*}
	&\Vert f(x)-f_{\tilde{\eta}^m}(x_0)-\int_0^1 g_{\tilde{\theta}^m}(x_0+t(x-x_0))\cdot (x-x_0)\dx t\Vert_{L^\infty(U)}\\
	&\leq \vert f(x_0)-f_{\tilde{\eta}^m}(x_0)\vert+\esssup\limits_{x\in U}\vert \int_0^1((\nabla f-g_{\tilde{\theta}^m})(x_0+t(x-x_0)))\cdot (x-x_0)\dx t\vert\\
	&\leq cm^{-\gamma}+c m^{-\beta}\diam U.
\end{align*}
Furthermore, it holds true by the Leibniz integral rule that
\begin{align*}
	&\Vert \nabla_x f-\nabla_x(f_{\tilde{\eta}^m}(x_0)+\int_0^1 g_{\tilde{\theta}^m}(x_0+t(x-x_0))\cdot (x-x_0)\dx t)\Vert_{L^\infty(U)}\\
	&=\Vert \nabla_x f-\int_0^1 t \nabla g_{\tilde{\theta}^m}(x_0+t(x-x_0))\cdot(x-x_0)+g_{\tilde{\theta}^m}(x_0+t(x-x_0))\dx t\Vert_{L^\infty(U)}\\
	&=\Vert \nabla_x f-\int_0^1 \frac{\dx}{\dx t}(tg_{\tilde{\theta}^m}(x_0+t(x-x_0)))\dx t\Vert_{L^\infty(U)}\\
	&=\Vert \nabla_x f-g_{\tilde{\theta}^m}(x)\Vert_{L^\infty(U)}\\
	&\leq cm^{-\beta}.
\end{align*}
Note that the Leibniz integral rule is applicable as $\int_0^1 g_{\tilde{\theta}^m}(x_0+t(x-x_0))\cdot (x-x_0)\dx t$ is finite, $t \nabla g_{\tilde{\theta}^m}(x_0+t(x-x_0))\cdot(x-x_0)+g_{\tilde{\theta}^m}(x_0+t(x-x_0))$ exists and is majorizable by $\diam U\Vert g_{\tilde{\theta}^m}\Vert_{W^{1,\infty}(U)}$. This shows that $f$ is approximated by $f_{\tilde{\eta}^m}(x_0)+\int_0^1 g_{\tilde{\theta}^m}(x_0+t(x-x_0))\cdot (x-x_0)\dx t$ in $W^{1,\infty}(U)$ as $m\to \infty$ with rate given by $\min(\beta,\gamma)$.
\section{Physical term}
\label{app:physical_term} In the following section we will provide proofs for Proposition \ref{prop:ass4_linear} and Proposition \ref{prop:ass4_nonlinear}, addressing Assumption \ref{ass_phys_term} on the physical term both in the linear and nonlinear case.
\subsection{Linear case}
\label{app:linear} We prove Proposition \ref{prop:ass4_linear}, starting with the first part on Assumption \ref{ass_phys_term}, i), the induction of well-defined Nemytskii operators.
\begin{lemma}
	\label{lem:phys_lin_nemyt} Let Assumption \ref{ass_init_set} hold true and $\tilde{V}\hookrightarrow W^{\omega, \hat{p}}(\Omega)$. Suppose that $t\mapsto \Phi_n(t,\varphi)$ and $t\mapsto \Psi(t,\varphi)$ are measurable for all $\varphi\in X_\varphi$ and $s_\beta$ fulfill \eqref{conditions_sbeta}. Assume that there exist functions $\mathcal{B}_1, \mathcal{B}_2:\mathbb{R}_{\geq 0}\to \mathbb{R}_{\geq 0}$ that map bounded sets to bounded sets and $\phi\in L^{\frac{pq}{p-q}}(0,T)$ (with $\phi\in L^{\infty}(0,T)$ if $p=q$), $\psi\in L^q(0,T)$ such that
	\begin{align}
		\label{growth_cond_lin_terms} \Vert \Phi_{n,\beta}(t,\varphi)\Vert_{L^{s_\beta}(\Omega)}\leq \phi(t)\mathcal{B}_1(\Vert\varphi\Vert_{X_\varphi}), ~ ~ ~ 	\Vert \Psi(t,\varphi)\Vert_{L^{\hat{q}}(\Omega)}\leq \psi(t)\mathcal{B}_2(\Vert\varphi\Vert_{X_\varphi}).
	\end{align}
	Then $F$ in \eqref{phys_term_linear} induces a well-defined Nemytskii operator $F: \mathcal{V}^N\times X_\varphi\to \mathcal{W}$ with $$[F((u_n)_{1\leq n\leq N}, \varphi)](t) = F(t,(u_n(t))_{1\leq n\leq N}, \varphi)$$ for $(u_n)_{1\leq n\leq N}\in \mathcal{V}^N, \varphi\in X_\varphi$ and $t\in (0,T)$.
\end{lemma}
\begin{proof}
	Employing similar arguments as in the proof of Lemma \ref{lem:nemytski_u_a} together with measurability of $t\mapsto \Phi_n(t,\varphi)$ and $t\mapsto \Psi(t,\varphi)$ yields Bochner measurability of 
	\[ (0,T)\ni t\mapsto  \Psi(t,\varphi(\cdot))+\sum_{n=1}^N \mathcal{J}_\omega u_n(t,\cdot)\cdot \Phi_n(t, \varphi(\cdot))\in W.
	\] Welldefinedness follows by the following chain of estimations for $u=(u_n)_{1\leq n\leq N}\in \mathcal{V}^N$ and $\varphi\in X_\varphi$ for some generic constant $c>0$. By the embedding $L^{\hat{q}}(\Omega)\hookrightarrow W$ it holds $\Vert F(u,\varphi)\Vert_{\mathcal{W}}\leq c\Vert F(u,\varphi)\Vert_{L^q(0,T;L^{\hat{q}}(\Omega))}$ which by the definition of $F$ and the triangle inequality can be estimated by
	\begin{align*}
		c \left(\sum_{n=1}^N\left(\int_0^T\Vert \mathcal{J}_\omega u_n(t,\cdot)\cdot \Phi_n(t,\varphi(\cdot))\Vert_{L^{\hat{q}}(\Omega)}^q\dx t\right)^{1/q}+\left(\int_0^T \Vert\Psi(t,\varphi(\cdot)) \Vert_{L^{\hat{q}}(\Omega)}^q\dx t\right)^{1/q}\right).
	\end{align*}
	Due to the growth condition in \eqref{growth_cond_lin_terms} we may estimate the term
	\[ \left(\int_0^T \Vert\Psi(t,\varphi(\cdot)) \Vert_{L^{\hat{q}}(\Omega)}^q\dx t\right)^{1/q}\leq \mathcal{B}_2(\Vert \varphi\Vert_{X_\varphi})\Vert \psi\Vert_{L^q(0,T)}<\infty.
	\] For the remaining part note that by \cite[Theorem 6.1]{behzadan2021}, \cite[Corollary 6.3]{behzadan2021} and  the choice of $s_\beta$ it holds true that the pointwise multiplication of functions is a continuous bilinear map
	\[ W^{\omega-\vert\beta\vert,\hat{p}}(\Omega)\times L^{s_\beta}(\Omega)\to L^{\hat{q}}(\Omega).
	\] Thus, there exists some generic constant $c>0$ independent of $u_n, t, \varphi, \Phi_n$ with
	\begin{align*}
		\Vert \mathcal{J}_\omega u_n(t,\cdot)\cdot \Phi_n(t,\varphi(\cdot))\Vert_{L^{\hat{q}}(\Omega)}\leq c \sum_{0\leq \vert\beta\vert\leq \omega}\Vert D^\beta u_n(t,\cdot)\Vert_{W^{\omega-\vert\beta\vert,\hat{p}}(\Omega)}\Vert \Phi_{n,\beta}(t,\varphi(\cdot))\Vert_{L^{s_\beta}(\Omega)}.
	\end{align*}
	We employ \eqref{growth_cond_lin_terms} together with Hölder's inequality to obtain 
	\[ \left(\int_0^T\Vert \mathcal{J}_\omega u_n(t,\cdot)\cdot \Phi_n(t,\varphi(\cdot))\Vert_{L^{\hat{q}}(\Omega)}^q\dx t\right)^{1/q}\leq c\mathcal{B}_1(\Vert\varphi\Vert_{X_\varphi}) \left(\int_0^T \Vert u_n\Vert_{W^{\omega,\hat{p}}(\Omega)}^q\phi(t)^q\dx t\right)^{1/q}.
	\] Using Hölder's inequality once more and $\mathcal{V}\hookrightarrow L^p(0,T; W^{\omega,\hat{p}}(\Omega))$ yields that
	\[ \left(\int_0^T \Vert u_n\Vert_{W^{\omega,\hat{p}}(\Omega)}^q\phi(t)^q\dx t\right)^{1/q}\leq c\Vert u_n\Vert_{L^p(0,T; W^{\omega,\hat{p}}(\Omega))}\Vert \phi\Vert_{L^{\frac{pq}{p-q}}(0,T)}\leq c\Vert u_n\Vert_{\mathcal{V}}\Vert \phi\Vert_{L^{\frac{pq}{p-q}}(0,T)}
	\] which is again finite by assumption. The case $s_\beta = \frac{\hat{p}\hat{q}}{\hat{p}-\hat{q}}$ can be covered similarly using $V\hookrightarrow W^{\omega, \hat{p}}(\Omega)$ and employing Hölder's inequality. Finally, we derive that $\Vert F(u,\varphi)\Vert_{\mathcal{W}}<\infty$ which concludes the assertions of the lemma.
\end{proof}
The next result addresses the remaining part of Proposition \ref{prop:ass4_linear} on continuity.
\begin{lemma}
	\label{lem:phys_term_weak} Let the assumptions of Lemma \ref{lem:phys_lin_nemyt} hold true. Suppose that $\Psi(t,\cdot):X_\varphi\to L^{\hat{q}}(\Omega)$ is weakly continuous for almost every $t\in (0,T)$. Let $s_\beta$ be given as in Lemma \ref{lem:phys_lin_nemyt}, additionally with strict inequality $\frac{\omega-\vert\beta\vert}{d}>\frac{1}{\hat{p}}-\frac{1}{\hat{q}}+\frac{1}{s_\beta}$ if $\hat{q}=1$ or $s_\beta = \frac{\hat{p}\hat{q}}{\hat{p}-\hat{q}}$. Assume that $\Phi_{n,\beta}(t,\cdot):X_\varphi\to L^{s_\beta}(\Omega)$ is weakly continuous for a.e. $t\in (0,T)$. Furthermore, suppose that either $\omega\leq \kappa$ or otherwise in case $\omega>\kappa$ the following additional conditions hold:
	\begin{itemize}
		\item ~For each $0\leq \vert\beta\vert<\omega$ assume that there exists some $\hat{q}\leq c_\beta\leq \infty$ such that $W^{\omega-\vert\beta\vert,\hat{p}}(\Omega)\hookdoubleheadrightarrow L^{c_\beta}(\Omega)$ and that we have the additional growth condition
		\[ \Vert \Phi_{n,\beta}(t,\varphi)\Vert_{\frac{c_\beta\hat{q}}{c_\beta-\hat{q}}}\leq \phi(t)\mathcal{B}_1(\Vert \varphi\Vert_{X_\varphi}).
		\]
		\item ~For $\vert\beta\vert=\omega$ assume that $\Phi_{n,\beta}(t,\cdot): X_\varphi\to L^{\frac{\hat{p}\hat{q}}{\hat{p}-\hat{q}}}(\Omega)$ is well-defined and weak-strong continuous for a.e. $t\in (0,T)$.
	\end{itemize}
	Then $\mathcal{V}^N\times X_\varphi \ni (u,\varphi)\mapsto F(u,\varphi)\in \mathcal{W}$ induced by \eqref{phys_term_linear} is weak-weak continuous.
\end{lemma}
\begin{proof}
	Let $(u^k)_k\subseteq \mathcal{V}^N, (\varphi^k)_k\subseteq X_\varphi$ and $u\in \mathcal{V}^N, \varphi\in X_\varphi$ with $u^k\rightharpoonup u$ in $\mathcal{V}^N$ and $\varphi^k\rightharpoonup \varphi$ in $X_\varphi$ as $k\to \infty$. We verify that $F(u^k,\varphi^k)\rightharpoonup F(u,\varphi)$ in $\mathcal{W}$ as $k\to \infty$. First, by $L^{\hat{q}}(\Omega)\hookrightarrow W$ and the growth condition in \eqref{growth_cond_lin_terms} it holds true for $w^*\in \mathcal{W}^*$ and a.e. $t\in[0,T]$ that 
	\begin{align*}
		\langle \Psi(t,\varphi^k)-\Psi(t,\varphi), w^*(t)\rangle_{W,W^*}&\leq c(\Vert \Psi(t,\varphi^k)\Vert_{L^{\hat{q}}(\Omega)}+\Vert \Psi(t,\varphi)\Vert_{L^{\hat{q}}(\Omega)})\Vert w^*(t)\Vert_{W^*}\\
		& \leq c(\mathcal{B}_2(\Vert \varphi^k\Vert_{X_\varphi})+\mathcal{B}_2(\Vert\varphi\Vert_{X_\varphi}))\psi(t)\Vert w^*(t)\Vert_{W^*}.
	\end{align*}
	By $\varphi^k\rightharpoonup \varphi$ in $X_\varphi$ the $\Vert \varphi^k\Vert_{X_\varphi}$ are uniformly bounded for all $k$. Thus, as $\mathcal{B}_2$ maps bounded sets to bounded sets there exists some $\tilde{c}$ such that $\mathcal{B}_2(\Vert \varphi^k\Vert_{X_\varphi})+\mathcal{B}_2(\Vert\varphi\Vert_{X_\varphi})\leq \tilde{c}$ for all $k$ and we derive that $\langle \Psi(t,\varphi^k)-\Psi(t,\varphi), w^*(t)\rangle_{W,W^*}$ is majorized by the integrable function $t\mapsto\tilde{c}\psi(t)\Vert w^*(t)\Vert_{W^*}$ independently of $k$ with  
	\[ \int_0^T\langle \Psi(t,\varphi^k)-\Psi(t,\varphi), w^*(t)\rangle_{W,W^*}\dx t\leq \tilde{c}\Vert \psi\Vert_{L^q(0,T)}\Vert w^*\Vert_{\mathcal{W}^*}<\infty
	\] by Hölder's inequality. The Dominated Convergence Theorem and weak-weak continuity of $\Psi$ for a.e. $t\in(0,T)$ yield that $\langle \Psi(\cdot,\varphi^k)-\Psi(\cdot,\varphi),w^*\rangle_{\mathcal{W},\mathcal{W}^*}\to 0$ as $k\to \infty$ and hence, that $\Psi(\cdot, \varphi^k)\rightharpoonup \Psi(\cdot, \varphi)$ in $\mathcal{W}$. Thus, by \eqref{phys_term_linear} it remains to show that
	\begin{align}
		\label{claim:weak_conv_lin_part} \mathcal{J}_\omega u^k_n\cdot \Phi_n(\cdot,\varphi^k)\rightharpoonup\mathcal{J}_\omega u_n\cdot\Phi_n(\cdot,\varphi)\qquad \text{for} ~ 1\leq n\leq N
	\end{align}
	in $L^q(0,T;L^{\hat{q}}(\Omega))$ as $k\to \infty$ which is sufficient due to $L^q(0,T;L^{\hat{q}}(\Omega))\hookrightarrow \mathcal{W}$. Since
	\[ \mathcal{J}_\omega u^k_n\cdot \Phi_n(\cdot,\varphi^k)-\mathcal{J}_\omega u_n\cdot\Phi_n(\cdot,\varphi) =\sum_{0\leq \vert\beta\vert\leq \omega} \left(D^\beta u_n^k\cdot\Phi_{n,\beta}(\cdot,\varphi^k)-D^\beta u_n\cdot\Phi_{n,\beta}(\cdot,\varphi)\right)
	\] by \eqref{phys_term_linear} it suffices to prove that for any fixed $0\leq\vert\beta\vert\leq \omega$ and $1\leq n\leq N$
	\begin{align}
		\label{reduced_weak_conv} D^\beta u_n^k\cdot\Phi_{n,\beta}(\cdot,\varphi^k)\rightharpoonup D^\beta u_n\cdot\Phi_{n,\beta}(\cdot,\varphi)
	\end{align}
	in $L^q(0,T;L^{\hat{q}}(\Omega))$ as $k\to \infty$. We show first that
	\begin{align}
		\label{reduced_weak_conv1} D^\beta u_n\cdot\left(\Phi_{n,\beta}(\cdot,\varphi^k)-\Phi_{n,\beta}(\cdot,\varphi)\right)\rightharpoonup 0
	\end{align}
	in $L^q(0,T;L^{\hat{q}}(\Omega))$ as $k\to \infty$ and then
	\begin{align}
		\label{reduced_weak_conv2} \left(D^\beta u_n^k-D^\beta u_n\right)\cdot\Phi_{n,\beta}(\cdot,\varphi^k)\rightharpoonup 0
	\end{align}
	in $L^q(0,T;L^{\hat{q}}(\Omega))$ as $k\to \infty$, proving the weak convergence in \eqref{reduced_weak_conv}. For that, let $w^*\in L^{q^*}(0,T;L^{\hat{q}^*}(\Omega))$. Then for a.e. $t\in(0,T)$ it holds that $D^\beta u_n(t)\in W^{\omega-\vert\beta\vert,\hat{p}}(\Omega)$ (with $W^{0,\hat{p}}(\Omega)= L^{\hat{p}}(\Omega)$) and $w^*(t)\in L^{\hat{q}^*}(\Omega)$. By \cite[Theorem 6.1, Corollary 6.3]{behzadan2021} the inclusion $D^\beta u_n(t)w^*(t)\in L^{r_\beta}(\Omega)$ holds true with $\frac{\hat{p}\hat{q}}{\hat{q}-\hat{p}+\hat{p}\hat{q}}\leq r_\beta\leq \hat{q}^*$ and $r_\beta^{-1}\geq \frac{1}{\hat{p}}+\frac{1}{\hat{q}^*}-\frac{\omega-\vert \beta\vert}{d}$ (with strict inequality if $\hat{q}=1$). In particular by the requirements on $s_\beta$ we may choose $r_\beta = s_\beta^*$ (or equivalently $r_\beta^*=s_\beta$). Using that $D^\beta u_n\in L^p(0,T;W^{\omega-\vert\beta\vert,\hat{p}}(\Omega))$ and $w^*\in L^{q^*}(0,T;L^{\hat{q}^*}(\Omega))$ we derive
	\[ D^\beta u_n w^* \in L^{\frac{pq^*}{p+q^*}}(0,T; L^{r_\beta}(\Omega)).
	\] Thus, we obtain by the growth condition \eqref{growth_cond_lin_terms} that for $w^*\in L^{q^*}(0,T;L^{\hat{q}^*}(\Omega))$
	\begin{align*}
		\langle \Phi_{n,\beta}(t, \varphi^k)-\Phi_{n,\beta}(t,\varphi), &D^\beta u_n(t) w^*(t)\rangle_{L^{r_\beta^*}(\Omega), L^{r_\beta}(\Omega)}\\
		&\leq  \Vert \Phi_{n,\beta}(t, \varphi^k)-\Phi_{n,\beta}(t,\varphi)\Vert_{L^{r_\beta^*}(\Omega)}\Vert D^\beta u_n(t)w^*(t)\Vert_{L^{r_\beta}(\Omega)}\\
		&\leq c\phi(t)\Vert D^\beta u_n(t)w^*(t)\Vert_{L^{r_\beta}(\Omega)}
	\end{align*}
	for $0\leq \vert \beta\vert\leq \omega$ and a.e. $t\in[0,T]$. Hence, independently of $k$, the term $$\langle \Phi_{n,\beta}(t, \varphi^k)-\Phi_{n,\beta}(t,\varphi), D^\beta u_n(t) w^*(t)\rangle_{L^{r_\beta^*}(\Omega), L^{r_\beta}(\Omega)}$$ is majorized by the integrable function $t\mapsto c\phi(t)\Vert D^\beta u_n(t)w^*(t)\Vert_{L^{r_\beta}(\Omega)}$  with 
	\begin{multline*}
		\int_0^T \langle \Phi_{n,\beta}(t, \varphi^k)-\Phi_{n,\beta}(t,\varphi), D^\beta u_n(t) w^*(t)\rangle_{L^{r_\beta^*}(\Omega), L^{r_\beta}(\Omega)}\dx t\\ \leq c \Vert\phi\Vert_{L^{\frac{pq}{p-q}}(\Omega)}\Vert D^\beta u_n w^* \Vert_{L^{\frac{pq^*}{p+q^*}}(0,T; L^{r_\beta}(\Omega))}<\infty
	\end{multline*}
	as $(\frac{pq^*}{p+q^*})^*=\frac{pq}{p-q}$. Employing dominated convergence together with weak continuity of $\Phi_{n,\beta}(t,\cdot): X_\varphi\to L^{r_\beta^*}(\Omega)=L^{s_\beta}(\Omega)$ for a.e. $t\in (0,T)$ concludes \eqref{reduced_weak_conv1}. The case that $s_\beta = \frac{\hat{p}\hat{q}}{\hat{p}-\hat{q}}$ can be similarly dealt with as before using that $D^\beta u_n(t) w^*(t)\in L^{(\frac{\hat{p}\hat{q}}{\hat{p}-\hat{q}})^*}(\Omega)$ for $w^*(t)\in L^{\hat{q}^*}(\Omega)$ by Hölder's generalized inequality. Next we prove the weak convergence in \eqref{reduced_weak_conv2} which follows if we can show for $w^*\in L^{q^*}(0,T;L^{\hat{q}^*}(\Omega))$ that
	\begin{align}
		\label{remaining_integral} \int_0^T\langle \left(D^\beta u_n^k(t)-D^\beta u_n(t)\right)\cdot\Phi_{n,\beta}(t,\varphi^k), w^*(t)\rangle_{L^{\hat{q}}(\Omega),L^{\hat{q}^*}(\Omega)}\dx t \to 0
	\end{align}
	as $k\to \infty$. In fact due to Hölder's inequality, the growth condition in \eqref{growth_cond_lin_terms} and similar arguments regarding the multiplication operator as in Lemma \ref{lem:phys_lin_nemyt} we obtain that the integrand of \eqref{remaining_integral} can be bounded from above by
	\begin{align}
		\label{est:first_term} c\phi(t)\mathcal{B}_1(\Vert \varphi^k\Vert_{X_\varphi})\Vert u_n^k(t)- u_n(t)\Vert_{W^{\omega,\hat{p}}(\Omega)}\Vert w^*(t)\Vert_{L^{\hat{q}^*}(\Omega)}
	\end{align}
	for some generic constant $c>0$ for a.e. $t\in (0,T)$. Using that $\mathcal{V}\hookrightarrow \mathcal{C}(0,T; W^{\omega,\hat{p}}(\Omega))$ by \cite[Lemma 7.1]{Roubíček2013} and the assumption $\tilde{V}\hookrightarrow W^{\omega, \hat{p}}(\Omega)$, we derive that $u_n^k\rightharpoonup u_n$ in $\mathcal{C}(0,T; W^{\omega,\hat{p}}(\Omega))$ as $k\to \infty$ which by boundedness of weakly convergent sequences in Banach spaces (see \cite[Proposition 3.5 (iii)]{Brezis2010}) implies that $\Vert u_n^k(t)- u_n(t)\Vert_{W^{\omega,\hat{p}}(\Omega)}$ is bounded independently of $t\in (0,T)$ and $k\in\mathbb{N}$. Using uniform boundedness of $\Vert\varphi^k\Vert_{X_\varphi}$ for all $k\in \mathbb{N}$ implies that \eqref{est:first_term} and hence the integrand of \eqref{remaining_integral} is majorized by the function $t\mapsto c\phi(t)\Vert w^*(t)\Vert_{L^{\hat{q}^*}(\Omega)}$ which is integrable since
	\[ \int_0^T\phi(t)\Vert w^*(t)\Vert_{L^{\hat{q}^*}(\Omega)}\dx t \leq T^{1/p}\Vert\phi\Vert_{L^{\frac{pq}{p-q}}(\Omega)}\Vert w^*\Vert_{L^{q^*}(0,T;L^{\hat{q}^*}(\Omega))}
	\] by Hölder's inequality. This can be similarly shown to hold true in case $s_\beta = \frac{\hat{p}\hat{q}}{\hat{p}-\hat{q}}$ using Hölder's generalized inequality. We show that the integrand of \eqref{remaining_integral} converges to zero pointwise for a.e. $t\in (0,T)$ under the following case distinction.
	
	If $\omega \leq \kappa$ it follows by $V\hookdoubleheadrightarrow W^{\kappa,\hat{p}}(\Omega)\hookrightarrow W^{\omega,\hat{p}}(\Omega)$ (due to Assumption \ref{ass_init_set}, ii)) that \eqref{est:first_term} converges to zero as $k\to \infty$ since the term depending on $\varphi^k$ is bounded independently of $k\in\mathbb{N}$ by weak convergence and $u^k_n(t)\rightharpoonup u_n(t)$ in $V$ for a.e. $t\in(0,T)$ as $k\to \infty$. Otherwise it holds $\omega>\kappa$ and in case $0\leq \vert\beta\vert<\omega$ we can similarly estimate the integrand of \eqref{remaining_integral} using the assumptions of the lemma by
	\[ c\phi(t)\mathcal{B}_1(\Vert\varphi^k\Vert_{X_\varphi})\Vert D^\beta u_n^k(t)-D^\beta u_n(t)\Vert_{L^{c_\beta}(\Omega)}\Vert w^*(t)\Vert_{L^{\hat{q}^*}(\Omega)}
	\] which converges to zero due to $D^\beta u_n^k(t)\rightharpoonup D^\beta u_n(t)$ in $W^{\omega-\vert\beta\vert,\hat{p}}(\Omega)\hookdoubleheadrightarrow L^{c_\beta}(\Omega)$ as $k\to \infty$. It remains to consider the case $\vert \beta\vert=\omega$. Since $D^\beta u_n^k(t)\rightharpoonup D^\beta u_n(t)$ in $L^{\hat{p}}(\Omega)$ as $k\to \infty$ and $\Phi_{n,\beta}(t,\varphi^k)\cdot w^*(t)\to \Phi_{n,\beta}(t,\varphi)\cdot w^*(t)$ in $L^{\hat{p}^*}(\Omega)$ by Hölder's inequality and weak-strong continuity of $\Phi_{n,\beta}(t,\cdot):X_\varphi\to L^{\frac{\hat{p}\hat{q}}{\hat{p}-\hat{q}}}(\Omega)$, the integrand of \eqref{remaining_integral} converges to zero also in this case. Hence, also \eqref{reduced_weak_conv2} holds true.
	
	 Thus, we recover \eqref{reduced_weak_conv} and consequently \eqref{claim:weak_conv_lin_part}. Finally, this implies that $F(u^k,\varphi^k)\rightharpoonup F(u,\varphi)$ in $\mathcal{W}$ concluding weak continuity as stated in the assertion of the lemma.
\end{proof}
Combining Lemma \ref{lem:phys_lin_nemyt} and Lemma \ref{lem:phys_term_weak} concludes the result in Proposition \ref{prop:ass4_linear}.
\begin{remark}
	In case $\omega\leq \kappa$ the assumption that $\tilde{V}\hookrightarrow W^{\omega,\hat{p}}(\Omega)$ can be avoided and only $V\hookrightarrow W^{\omega,\hat{p}}(\Omega)$ is necessary, using the compact embedding of the extended state space $\mathcal{V}\hookdoubleheadrightarrow L^p(0,T;W^{\kappa,\hat{p}}(\Omega))$ discussed at the end of the proof of Theorem \ref{thm:uniqueness}. We further note that the additional assumptions in the previous lemma are necessary to guarantee the convergence of \eqref{remaining_integral} to zero. Here it is not sufficient to only have weak convergence of the terms depending on the state and the parameters, respectively. The reason is that continuous bilinear operators are not jointly weakly continuous in general. The latter holds under the Dunford-Pettis property which reflexive spaces only attain in finite dimensions.
\end{remark}
\subsection{Nonlinear case}
\label{app:nonlinear} We prove Proposition \ref{prop:ass4_nonlinear}, starting with the first part on Assumption \ref{ass_phys_term}, i), the induction of well-defined Nemytskii operators.
\begin{lemma}
	\label{lem:Fibochner} Let Assumption \ref{ass_init_set} and the extended state space embedding
	\[ \mathcal{V}\hookrightarrow\mathcal{C}(0,T;H)
	\] hold true. Suppose that the $F_n(\cdot, \cdot, \varphi):(0,T)\times V^N\to W$ satisfy the Carathéodory condition, i.e., $t\mapsto F_n(t,v, \varphi)$ is measurable for $v\in V^N$ and $v\mapsto F_n(t,v,\varphi)$ is continuous for a.e. $t\in(0,T)$. Further assume that the $F_n$ satisfy the growth condition
	\begin{align}
		\label{growth_condition} \Vert F_n(t,(v_n)_{1\leq n\leq N}, \varphi)\Vert_W \leq \mathcal{B}_0(\Vert\varphi\Vert_{X_\varphi}, \sum_{n=1}^N\Vert v_n\Vert_H)(\Gamma(t)+\sum_{n=1}^N\Vert v_n \Vert_{V})
	\end{align}
	for some $\Gamma\in L^q(0,T)$ and $\mathcal{B}_0:\mathbb{R}^2\to \mathbb{R}$, increasing in the second entry and, for fixed second entry, mapping bounded sets to bounded sets. Then the $F_n:(0,T)\times V^N\times X_\varphi\to W$ induce well-defined Nemytskii operators $F_n: \mathcal{V}^N\times X_\varphi\to \mathcal{W}$ with
	\begin{align}
		\label{nemytskii2} [F_n(v,\varphi)](t) = F_n(t, v(t), \varphi)
	\end{align}
	for $v\in \mathcal{V}^N$ and $\varphi\in X_\varphi$.
\end{lemma}
\begin{proof}
	The Carathéodory assumption ensures Bochner measurability of the map $t\mapsto F_n(t,v(t),\varphi)$ for $v\in \mathcal{V}^N$ and $\varphi\in X_\varphi$. Growth condition \eqref{growth_condition} and Hölder's inequality imply that for $v\in \mathcal{V}^N$ and $\varphi\in X_\varphi$ the term $\int_0^T\Vert F_n(t,v(t),\varphi)\Vert_W^q\dx t$ can be bounded, for $C>0$ some in the following generically used constant, by
	\begin{align*}
		C\int_0^T \mathcal{B}_0(\Vert\varphi\Vert_{X_\varphi}, \sum_{n=1}^N\Vert v_n(t)\Vert_H)^q(\vert\Gamma(t)\vert^q+\sum_{n=1}^N\Vert v_n(t)\Vert_{V}^q)\dx t
	\end{align*}
	which may be further estimated by
	\begin{align}
		\label{Fibochner} C\mathcal{B}_0(\Vert\varphi\Vert_{X_\varphi}, \sum_{n=1}^N\Vert v_n\Vert_{\mathcal{C}(0,T;H)})^q(\Vert\Gamma\Vert_{L^q(0,T)}^q+\sum_{n=1}^N\int_0^T \Vert v_n(t)\Vert_{V}^q\dx t).
	\end{align}
	Monotonicity of $\mathcal{B}_0$ in its second entry, $v_n\in \mathcal{V}\hookrightarrow \mathcal{C}(0,T;H)$, $\Gamma\in L^q(0,T)$ and
	\begin{align}
		\notag\int_0^T \Vert v_n(t)\Vert_{V}^q\dx t\leq T^{\frac{p-q}{p}}\Vert v_n\Vert^q_{L^p(0,T;V)}\leq T^{\frac{p-q}{p}}\Vert v_n\Vert^q_\mathcal{V}<\infty
	\end{align}
	yield that \eqref{Fibochner} is finite. As a consequence, we derive that $\int_0^T \Vert F_n(t, v(t), \varphi)\Vert_W^q\dx t<\infty$ and thus, that $\Vert F_n(v,\varphi)\Vert_\mathcal{W}<\infty$ which together with separability of $W$ implies Bochner integrability of $t\mapsto F_n(t,v(t),\varphi)$ and well-definedness of the Nemytskii operator $F_n: \mathcal{V}^N\times X_\varphi\to \mathcal{W}$ concluding the assertions of the lemma.
\end{proof}
The next result addresses the remaining part of Proposition \ref{prop:ass4_nonlinear}. The proof is essentially based on \cite[Lemma 5]{AHN23}, for which the requirements of Lemma \ref{lem:Fibochner} are extended by a stronger growth condition. 
\begin{lemma}
	\label{lem:weak_cont_Fi} Let Assumption \ref{ass_init_set} and the extended state space embedding
	\[ \mathcal{V}\hookrightarrow\mathcal{C}(0,T;H)
	\] hold true. Suppose that the $F_n(\cdot, \cdot, \varphi):(0,T)\times V^N\to W$ fulfill the Carathéodory condition as in Lemma \ref{lem:Fibochner} and weak-weak continuity of
	\begin{align*}
		F_n(t,\cdot): H^N\times X_\varphi&\to W\\
		(v_1,\dots, v_N,\varphi)&\mapsto F_n(t,v_1,\dots, v_N,\varphi)
	\end{align*}
	for a.e. $t\in(0,T)$. Further assume that the $F_n$ satisfy the stricter growth condition
	\begin{align}
		\label{growth_condition_strict} \Vert F_n(t,(v_n)_{1\leq n\leq N}, \varphi)\Vert_W \leq \mathcal{B}_0(\Vert\varphi\Vert_{X_\varphi}, \sum_{n=1}^N\Vert v_n\Vert_H)(\Gamma(t)+\sum_{n=1}^N\Vert v_n \Vert_{H})
	\end{align}
	for some $\Gamma\in L^q(0,T)$ and $\mathcal{B}_0:\mathbb{R}^2\to \mathbb{R}$, increasing in the second entry and, for fixed second entry, mapping bounded sets to bounded sets. Then the Nemytskii operator in \eqref{nemytskii2} is weak-weak continuous.
\end{lemma}
\begin{proof}
	First note that, for $(u_n)_n\in \mathcal{V}^N, \psi\in X_\varphi$ and $t\in(0,T)$, the growth condition \eqref{growth_condition_strict} together with $\mathcal{V}\hookrightarrow \mathcal{C}(0,T;H)$ and monotonicity of $\mathcal{B}_0$ yields
	\begin{align}
		\label{Fi_bound1} \Vert F_n(u_1, \dots, u_N, \psi)(t)\Vert_W\leq \mathcal{B}_0(\Vert \psi\Vert_{X_\varphi}, \sum_{n=1}^N\Vert u_n\Vert_{\mathcal{C}(0,T;H)})(\Gamma(t)+\sum_{n=1}^{N}\Vert u_n(t)\Vert_H).
	\end{align}
	Let $(v,\varphi)\in \mathcal{V}^N\times X_\varphi$ and $(v^m)_m\subseteq \mathcal{V}^N, (\varphi^m)_m\subseteq X_\varphi$ with $v^m\rightharpoonup v$ in $\mathcal{V}^N$ and $\varphi^m\rightharpoonup \varphi$ in $X_\varphi$. We show
	\begin{align}
		\label{F_weak_convergence} F_n(v_1^m, \dots, v_N^m, \varphi^m)\rightharpoonup F_n(v_1, \dots, v_N, \varphi) ~ ~ \text{ in } ~ \mathcal{W}.
	\end{align}
	Boundedness of weakly convergent sequences (see e.g. \cite[Proposition 3.5 (iii)]{Brezis2010}) and $\mathcal{V}\hookrightarrow\mathcal{C}(0,T;H)$ together with the assumptions on $\mathcal{B}_0$ ensure the existence of $c_\varphi, c_v>0$ such that both $\mathcal{B}_0(\Vert \varphi\Vert_{X_\varphi}, \sum_{n=1}^N\Vert v_n\Vert_{\mathcal{C}(0,T;H)})\leq \mathcal{B}_0(c_\varphi, c_v)$ and
	\begin{equation}
		\label{B_bound} \sup_{m\in\mathbb{N}}\mathcal{B}_0(\Vert \varphi^m\Vert_{X_\varphi}, \sum_{n=1}^N\Vert v_n^m\Vert_{\mathcal{C}(0,T;H)})\leq \mathcal{B}_0(c_\varphi, c_v)
	\end{equation}
	hold true. Fixing $w^*\in \mathcal{W}^*$ and using \eqref{Fi_bound1} and \eqref{B_bound} it follows for a.e. $t\in [0,T]$ that
	\begin{align}
		\label{estim_chain_Fi_weak_cont}
		\notag\langle F_n(v_1^m, &\dots, v_N^m, \varphi^m)(t)-F_n(v_1, \dots, v_N, \varphi)(t), w^*(t)\rangle_{W, W^*}\\
		\notag&\leq (\Vert F_n(v_1^m, \dots, v_N^m, \varphi^m)(t)\Vert_W+\Vert F_n(v_1, \dots, v_N, \varphi)(t)\Vert_W) \Vert w^*(t)\Vert_{W^*}\\
		\notag&\leq \mathcal{B}_0(c_\varphi, c_v)(\Gamma(t)+\sum_{n=1}^N\Vert v_n(t)\Vert_H+\sum_{n=1}^N\Vert v_n^m(t)\Vert_H)\Vert w^*(t)\Vert_{W^*}\\
		\notag&\leq \mathcal{B}_0(c_\varphi, c_v)(\vert\Gamma(t)\vert+2c_v)\Vert w^*(t)\Vert_{W^*}.
	\end{align}
	As a consequence, for $c\geq \mathcal{B}(c_\varphi, c_v)$ the function 
	\[ t\mapsto \langle F_n(v_1^m, \dots, v_N^m, \varphi^m)(t)-F_n(v_1, \dots, v_N, \varphi)(t), w^*(t)\rangle_{W, W^*}
	\] is majorized by the integrable function $t\mapsto c(\vert\Gamma(t)\vert+2c_v)\Vert w^*(t)\Vert_{W^*}$ with
	\begin{multline*}
		\langle F_n(v_1^m, \dots, v_N^m, \varphi^m)-F_n(v_1, \dots, v_N, \varphi), w^*\rangle_{\mathcal{W}, \mathcal{W}^*}\\
		\leq c\int_0^T(\vert\Gamma(t)\vert+2c_v)\Vert w^*(t)\Vert_{W^*}\dx t \leq c(\Vert\Gamma\Vert_{L^q(0,T)}+2c_v T^{1/q})\Vert w^*\Vert_{\mathcal{W}^*}<\infty
	\end{multline*}
	as $p\geq q$. Thus, once we argue weak convergence 
	\begin{align}
		\label{F_V_X_weak_continuity} F_n(t,v_1^m(t), \dots, v_N^m(t),\varphi^m)\rightharpoonup F_n(t,v_1(t), \dots, v_N(t),\varphi)
	\end{align}
	in $W$ for a.e. $t\in (0,T)$, weak convergence in \eqref{F_weak_convergence} follows by the Dominated Convergence Theorem. For the former, note that by $\mathcal{V}\hookrightarrow\mathcal{C}(0,T;H)$ the pointwise evaluation map realizing $u(t)\in H$ for $u\in \mathcal{V}$ is weakly closed due to 
	\[ \Vert u(t)\Vert_H\leq \Vert u\Vert_{\mathcal{C}(0,T;H)}\leq c\Vert u\Vert_{\mathcal{V}}
	\] for $t\in (0,T)$. By $v^m\rightharpoonup v$ in $\mathcal{C}(0,T;H)^N$ it holds true that $(\Vert v^m(t)\Vert_{H})_m$ is bounded for $t\in (0,T)$. Thus employing weak closedness of the evaluation map yields that every subsequence and hence, the whole sequence $v^m(t)$ converges weakly $v^m(t)\rightharpoonup v(t)$ in $H^N$. This together with weak-weak continuity of $F_n(t,\cdot): H^N\times X_\varphi\to W$ implies the convergence stated in \eqref{F_V_X_weak_continuity} and finally, the assertion.
\end{proof}
Combining Lemma \ref{lem:Fibochner} and Lemma \ref{lem:weak_cont_Fi} concludes the result in Proposition \ref{prop:ass4_nonlinear}.
\section{Existence of minimizers}
\label{app:existence} In this section we verify well-posedness of the minimization problem in \eqref{min_prob} under the Assumptions \ref{ass_init_set}, \ref{ass_param_app_class}, \ref{ass_phys_term}. As first step, we show that \eqref{min_prob} is indeed well-defined by proving that, for any $f_{\theta_n,n}\in\mathcal{F}_n^m$, the composed function $(t, u)\mapsto f_{\theta_n,n}(t, \mathcal{J}_\kappa u_1, \dots, \mathcal{J}_\kappa u_N)$ for $u\in V^N$ induces a well-defined Nemytskii operator on the dynamic space for $n=1,\dots, N$ and similarly the trace map $\gamma$. For that we consider first the differential operator introduced in \eqref{Jdifferential}.
\begin{lemma}
	\label{lem:Jbochner} Let Assumption \ref{ass_init_set} hold true. Then the function $\mathcal{J}_\kappa: W^{\kappa,\hat{p}}(\Omega)\to \otimes_{k=0}^\kappa L^{\hat{p}}(\Omega)^{p_k}$ induces a well-defined Nemytskii operator $\mathcal{J}_\kappa: L^p(0,T;W^{\kappa,\hat{p}}(\Omega))\to \otimes_{k=0}^\kappa L^p(0,T;L^{\hat{p}}(\Omega)^{p_k})$ with
	\[ [\mathcal{J}_\kappa v](t) = \mathcal{J}_\kappa v(t)
	\] for $v\in L^p(0,T;W^{\kappa,\hat{p}}(\Omega))$. Furthermore, it is weak-weak continuous.
\end{lemma}
\begin{proof}
	We show first that for fixed $\beta \in \mathbb{N}_0^d$ with $0\leq k:=\vert\beta\vert \leq \kappa$ the differential operator $D^\beta: W^{\kappa,\hat{p}}(\Omega)\to L^{\hat{p}}(\Omega)$ induces a well-defined Nemytskii operator $D^\beta: L^p(0,T;W^{\kappa,\hat{p}}(\Omega))\to L^p(0,T;L^{\hat{p}}(\Omega))$ with $[D^\beta v](t)=D^\beta v(t)$ for $v\in L^p(0,T;W^{\kappa,\hat{p}}(\Omega))$. To that end let $v\in L^p(0,T;W^{\kappa,\hat{p}}(\Omega))$. By Assumption \ref{ass_init_set} we derive that $v(t,\cdot)\in W^{\kappa,\hat{p}}(\Omega)$ for a.e. $t\in(0,T)$.
	Thus, it follows that 
	\begin{align}
		\label{differential_ineq_nemyt:1} \Vert D^\beta v(t,\cdot)\Vert_{L^{\hat{p}}(\Omega)}\leq \Vert v(t,\cdot)\Vert_{W^{\kappa,\hat{p}}(\Omega)}<\infty
	\end{align}
	for a.e. $t\in (0,T)$. As in particular $v\in L^1(0,T;W^{\kappa, \hat{p}}(\Omega))$ is Bochner measurable there exist temporal simple functions $v_k$ approximating $v$ pointwise a.e. in $(0,T)$ in the strong sense of $W^{\kappa, \hat{p}}(\Omega)$. Employing the embedding $W^{\kappa,\hat{p}}(\Omega)\hookrightarrow L^{\hat{p}}(\Omega)$ yields that the temporal simple functions $D^\beta v_k$ approximate $D^\beta v$ pointwise a.e. in $(0,T)$ in the strong sense of $L^{\hat{p}}(\Omega)$ and hence, Bochner measurability of
	\[ (0,T)\ni t \mapsto D^\beta v(t, \cdot)\in L^{\hat{p}}(\Omega).
	\]
	Similar to \eqref{differential_ineq_nemyt:1}  well-definedness of the Nemytskii operator $D^\beta: L^p(0,T;W^{\kappa,\hat{p}}(\Omega))\to L^p(0,T; L^{\hat{p}}(\Omega))$ with $[D^\beta v](t)=D^\beta v(t)$ for $v\in L^p(0,T;W^{\kappa,\hat{p}}(\Omega))$ follows.
	
	Weak-weak continuity of $D^\beta:L^p(0,T;W^{\kappa,\hat{p}}(\Omega))\to L^p(0,T; L^{\hat{p}}(\Omega))$ follows by boundedness and linearity where the latter follows immediately from linearity of the differential operator $D^\beta$. To see boundedness let $w\in L^{p^*}(0,T;L^{\hat{p}^*}(\Omega))$. Then by
	\eqref{differential_ineq_nemyt:1} we derive for some $c>0$ that
	\begin{align*}
		\langle D^\beta v, w\rangle_{L^{p}(0,T;L^{\hat{p}}(\Omega)), L^{p^*}(0,T;L^{\hat{p}^*}(\Omega))} &= \int_0^T \langle D^\beta v(t), w(t)\rangle_{L^{\hat{p}}(\Omega), L^{\hat{p}^*}(\Omega)}\dx t\\
		&\leq c\int_0^T \Vert v(t)\Vert_{W^{\kappa,\hat{p}}(\Omega)} \Vert w(t)\Vert_{L^{\hat{p}^*}(\Omega)}\dx t\\
		&\leq c\Vert v\Vert_{L^p(0,T;W^{\kappa,\hat{p}}(\Omega))}\Vert w\Vert_{L^{p^*}(0,T;L^{\hat{p}^*}(\Omega))}
	\end{align*}
	proving that $\Vert D^\beta v\Vert_{L^{p}(0,T;L^{\hat{p}}(\Omega))}\leq c\Vert v\Vert_{L^p(0,T;W^{\kappa,\hat{p}}(\Omega))}$.\\
	
	As a consequence, for fixed $0\leq k\leq \kappa$ the function $J^k:W^{\kappa,\hat{p}}(\Omega)\to L^{\hat{p}}(\Omega)^{p_k}$ in \eqref{Jl_operator} induces a well-defined Nemytskii operator $J^k: L^p(0,T;W^{\kappa,\hat{p}}(\Omega))\to L^p(0,T;L^{\hat{p}}(\Omega)^{p_k})$ with $[J^kv](t) = J^kv(t)$ for $v\in L^p(0,T;W^{\kappa,\hat{p}}(\Omega))$ which is linear and bounded and thus, weak-weak continuous. This is straightforward as $J^k$ is the Cartesian product of finitely many functions which by the previous considerations induce well-defined Nemytskii operators sharing the property of weak-weak continuity, respectively. The same arguments yield the assertion of the lemma that $\mathcal{J}_\kappa$ induces a well-defined Nemytskii operator $\mathcal{J}_\kappa: L^p(0,T;W^{\kappa,\hat{p}}(\Omega))\to \otimes_{k=0}^\kappa L^p(0,T;L^{\hat{p}}(\Omega)^{p_k})$ which is weak-weak continuous.
\end{proof}
By minor adaptions of the previous proof it is straightforward to show that indeed also the Nemytskii operator $\mathcal{J}_\kappa: \mathcal{V}\to \otimes_{k=0}^\kappa\mathcal{V}_k^\times$ is well-defined. Employing Assumption \ref{ass_param_app_class}, i) we obtain that $(t, u)\mapsto f_{\theta_n,n}(t, \mathcal{J}_\kappa u_1, \dots, \mathcal{J}_\kappa u_N)$ for $u\in V^N$ induces a well-defined Nemytskii operator with
\begin{align*}
	[f_{\theta_n,n}(\mathcal{J}_\kappa u_1,\dots, \mathcal{J}_\kappa u_N)](t)(x)= f_{\theta_n,n}(t, \mathcal{J}_\kappa u_1(t,x),\dots, \mathcal{J}_\kappa u_N(t,x))
\end{align*}
for $u\in \mathcal{V}^N$ and $t\in (0,T)$. On basis of the previous considerations we recover the next continuity result.

\begin{lemma}
	\label{lem:strong_weak_cont} In the setup of Assumption \ref{ass_init_set} and Assumption \ref{ass_param_app_class} it holds that
	\[ \Theta_n^m\times \mathcal{V}^N\ni(\theta_n, u)\mapsto f_n(\theta_n, u)=: f_{\theta_n,n}(\mathcal{J}_\kappa u_1, \dots, \mathcal{J}_\kappa u_N)\in \mathcal{W}
	\] is weak-weak continuous for $n=1,\dots, N$.
\end{lemma}
\begin{proof}
	Let $(\theta^j_n, u^j)\rightharpoonup (\theta_n, u)\in \Theta_n^m\times \mathcal{V}^N$ weakly as $j\to\infty$. We aim to show that $f_n(\theta^j_n,u^j)\rightharpoonup f_n(\theta_n, u)$ weakly in $\mathcal{W}$ as $j\to \infty$. First, as $\Theta_n^m$ is a subset of a finite-dimensional space, the convergence $\theta^j_n\to \theta_n$ holds in the strong sense. Regarding $(u^j)_j\subseteq \mathcal{V}^N$ we have that $u^j\to u$ strongly in $L^p(0,T;W^{\kappa,\hat{p}}(\Omega))^N$ as $j\to\infty$ by the compact embedding $\mathcal{V}\hookdoubleheadrightarrow L^p(0,T;W^{\kappa,\hat{p}}(\Omega))$, discussed at the end of the proof of Theorem \ref{thm:uniqueness}. Now as $u^j\to u$ strongly in $L^p(0,T;W^{\kappa,\hat{p}}(\Omega))^N$ as $j\to \infty$ it follows that $\mathcal{J}_\kappa u^j\to \mathcal{J}_\kappa u$ strongly in $(\otimes_{k=0}^\kappa L^p(0,T;L^{\hat{p}}(\Omega)^{p_k}))^N$ as $j\to \infty$ due to the definition of the operator $\mathcal{J}_\kappa$ and Lemma \ref{lem:Jbochner}. Together with Assumption \ref{ass_param_app_class}, ii), we derive that $f_n(\theta_n^j, u^j)\rightharpoonup f_n(\theta_n, u)$ weakly in $L^q(0,T;L^{\hat{q}}(\Omega))$ as $j\to \infty$. Finally, we conclude that indeed $f_n(\theta_n^j, u^j)\rightharpoonup f_n(\theta_n, u)$ weakly in $\mathcal{W}$ as $j\to \infty$ due to the embedding $L^q(0,T;L^{\hat{q}}(\Omega))\hookrightarrow \mathcal{W}$.
\end{proof}
Lastly, it remains to show that the trace map $\gamma$ induces a well-defined Nemytskii operator on the extended space.
\begin{lemma}
	\label{lem:gamma} Let Assumption \ref{ass_init_set} hold. Then the trace map $\gamma: V\to B$ induces a well-defined Nemytskii operator $\gamma: \mathcal{V}\to \mathcal{B}$ with $[\gamma(v)](t) = \gamma(v(t))$ for $v\in \mathcal{V}$. Furthermore, it is weak-weak continuous.
\end{lemma}
\begin{proof}
	By Assumption \ref{ass_init_set}, iv), the map $\gamma$ is continuous. Together with separability of the spaces $V, B$ and $p\geq s$ we derive by \cite[Theorem 1.43]{Necas2011} that $\gamma$ induces a well defined Nemytskii operator $\gamma:L^p(0,T;V)\to L^s(0,T;B)=\mathcal{B}$ which is continuous. Employing $\mathcal{V}\hookrightarrow L^p(0,T;V)$ and linearity of $\gamma$ concludes the proof.
\end{proof}
As a consequence together with the considerations in Section \ref{sec:problem_setting} the terms occurring in problem \eqref{min_prob} are well-defined. In view of well-posedness of the minimization problem \eqref{min_prob} we follow \cite{AHN23}. For that purpose define for $1\leq l\leq L$ the maps $G^l$ by
\begin{align*}
	G^l:X_\varphi^{N\times L}\times\mathcal{V}^{N\times L}\times\otimes_n\Theta_n^m\times H^{N\times L}\times\mathcal{B}^{N\times L}\to \mathcal{W}^N\times H^N\times\mathcal{B}^N\times \mathcal{Y}
\end{align*}
where $(\varphi, u,\theta, u_0,g)$ is mapped to 
\begin{align*}
	(\frac{\partial}{\partial t}u^l-F(t, u^l,\varphi^l)-f_{\theta}(t,\mathcal{J}_\kappa u^l),u^l(0)-u_{0}^l,\gamma(u^l)-g^l, K^mu^l)
\end{align*}
with $\varphi=(\varphi_n^l)_{\substack{1\leq n\leq N\\ 1\leq l\leq L}}\subseteq X_\varphi, u=(u_n^l)_{\substack{1\leq n\leq N\\1\leq l\leq L}}\subseteq\mathcal{V}, u_0=(u_{0,n}^l)_{\substack{1\leq n\leq N\\1\leq l\leq L}}\subseteq H$ and $\theta\in\otimes_n \Theta_n^m$. Recall that, notation wise, we use direct vectorial extensions over $n=1,\dots, N$. Furthermore, define for the domain of definition given by $\mathbf{D}(G):=X_\varphi^{N\times L}\times \mathcal{V}^{N\times L}\times\otimes_n\Theta_n\times H^{N\times L}\times\mathcal{B}^{N\times L}$ the operator
\begin{equation}
	\label{G_functional}
	\begin{aligned}
		G: \hspace{0.75cm}\mathbf{D}(G)\hspace{0.75cm}&\to \mathcal{W}^{N\times L}\times H^{N\times L}\times\mathcal{B}^{N\times L}\times \mathcal{Y}^L\\
		(\varphi, u,\theta, u_0,g) &\mapsto(G^l(\varphi, u,\theta,u_0,g))_{ 1\leq l\leq L}.
	\end{aligned}
\end{equation}
For $\lambda, \mu\in \mathbb{R}_+$ we define the map $\Vert\cdot\Vert_{\lambda, \mu}$ in $\mathcal{W}^{N\times L}\times H^{N\times L}\times\mathcal{B}^{N\times L}\times \mathcal{Y}^L$ by
\[ \Vert (w, h, b,y)\Vert_{\lambda, \mu} = \sum_{l=1}^{L}[\lambda(\Vert w^{l}\Vert^q_{\mathcal{W}}+\Vert h^{l}\Vert^2_{H}+\mathcal{D}_\BC(b^l))+\mu\Vert y^l\Vert^r_\mathcal{Y}]
\] for $(w,h,b,y)\in \mathcal{W}^{N\times L}\times H^{N\times L}\times \mathcal{B}^{N\times L} \times\mathcal{Y}^L$. Letting $\mathcal{R}$ as in Assumption \ref{ass_init_set}, vi), minimization problem \eqref{min_prob} may be equivalently rewritten by
\begin{align}
	\label{min_prob_equiv} \tag{$\mathcal{P}'$} \min_{(\varphi,u,\theta,u_0,g)\in \mathbf{D}(G)}\Vert G(\varphi,u,\theta,u_0,g)-(0,0,0,y)\Vert_{\lambda, \mu}+\mathcal{R}(\varphi,u,\theta,u_0,g).
\end{align}
Note that problem \eqref{min_prob_equiv} is in canonical form as the sum of a data-fidelity term and a regularization functional where $G$, given in \eqref{G_functional}, is the forward operator and $(0,0,0,y)\in \mathcal{W}^{N\times L}\times H^{N\times L}\times\mathcal{B}^{N\times L}\times\mathcal{Y}^L$ the measured data. We prove that problem \eqref{min_prob_equiv} admits a solution in $\mathbf{D}(G)$. If the forward operator $G$ is weakly closed then problem \eqref{min_prob_equiv} admits a minimizer due to the direct method (see e.g. \cite[Chapter 3]{Scherzer2008}) and Assumption \ref{ass_init_set}, vi). The idea is to choose a minimizing sequence, which certainly, for indices large enough is bounded by coercivity of the regularizer, the norm in $H$ and the discrepancy term (together with boundedness of the trace map), thus, attaining a weakly convergent subsequence. Employing weak closedness of $G$, weak lower semicontinuity of the norms, the regularizing term and the discrepancy term (due to Assumption \ref{ass_init_set}, i) and Lemma \ref{lem:gamma}) we derive that the limit of this subsequence is a solution of the minimization problem \eqref{min_prob_equiv}. Thus, it remains to verify weak closedness of the operator $G$. This is obviously equivalent and reduces to showing weak closedness of the operators $G^l$ for $1\leq l\leq L$. For weak closedness of $G^l$ it suffices to verify that

\begin{enumerate}[label=\Roman*., leftmargin=0.7cm]
	\item ~\((\varphi_n^l, (u_k^l)_{1\leq k\leq N}, \theta_n)\mapsto \frac{\partial}{\partial t}u_n^l-F_n(t,(u^l_k)_{1\leq k\leq N},\varphi^l_n)-f_{\theta_n, n}(t,(\mathcal{J}_\kappa u^l_k)_{1\leq k\leq N})\)
	\item ~\((u_n^l,u_{0,n}^l)\mapsto u_n^l(0)-u_{0,n}^l\)
	\item ~\(u^l=(u_n^l)_{1\leq n\leq N}\mapsto K^mu^l\)
	\item ~\((u^l,g^l)\mapsto \gamma(u^l)-g^l\)
\end{enumerate}
are weakly closed in $\mathbf{D}(G)$. The weak closedness in III. and IV. follows immediately by weak-weak continuity of $K^m$ and continuity of $\gamma$ assumed in Assumption \ref{ass_init_set}. In view of I. it suffices to verify weak closedness of the differential operator $\frac{\partial}{\partial t}: \mathcal{V}\to \mathcal{W}$ as the map $(\theta_n, v, \varphi)\mapsto  F_n(v_1,\dots, v_N,\varphi)+f_{\theta_n, n}(\mathcal{J}_\kappa v_1,\dots, \mathcal{J}_\kappa v_N)\in\mathcal{W}$ for $(\theta_n,v,\varphi)\in \Theta_n^m\times\mathcal{V}^N\times X_\varphi$ is weakly closed by Lemma \ref{lem:strong_weak_cont} and Assumption \ref{ass_phys_term}, ii). For weak closedness of $\frac{\partial}{\partial t}: \mathcal{V}\to \mathcal{W}$ recall Assumption \ref{ass_init_set}, ii) that $\tilde V\hookrightarrow W$, and iii) that $\mathcal{V}=L^p(0,T;V)\cap W^{1,p,p}(0,T;\tilde{V}), \mathcal{W} = L^q(0,T; W)$ with some $p\geq q$. Let $(u_m)_m\subseteq \mathcal{V}$ such that $u_m\rightharpoonup u \in \mathcal{V}$ and $\frac{\partial}{\partial t} u_m\rightharpoonup v\in \mathcal{W}$. As $\frac{\partial}{\partial t}u_m\rightharpoonup \frac{\partial}{\partial t}u\in L^p(0,T;\tilde{V})\hookrightarrow \mathcal{W}$ it follows immediately that $\frac{\partial}{\partial t}u=v$, concluding weak closedness of the temporal derivative. For II., employing the embedding $\mathcal{V}\hookrightarrow\mathcal{C}(0,T;H)$ we have that the map $(\cdot)_{t=0}:\mathcal{V}\to H$ with $u\mapsto u(0)$ is weakly closed due to \[\Vert u(0)\Vert_H\leq \sup_{0\leq t\leq T}\Vert u(t)\Vert_H\leq c\Vert u\Vert_\mathcal{V}.\] Thus, problem \eqref{min_prob_equiv} admits a solution in $\mathbf{D}(G)$ and we conclude well-posedness of problem \eqref{min_prob} under the Assumptions \ref{ass_init_set} to \ref{ass_phys_term}.
\section{Proof of Proposition \ref{prop:demo}}
\label{app:demo}
In this section we sketch the proof for the linear example in Proposition \ref{prop:demo} showcasing our main results for the sake of completeness. For that, once Assumption \ref{ass_init_set} - \ref{ass_uniqueness} are verified to hold true, it follows by application of Proposition \ref{prop:p_dagger} and Theorem \ref{thm:uniqueness} under suitable choice of regularization parameters depending on the noise of the measurement data and Proposition \ref{prop:belo}.\\

\noindent\textbf{Ad Assumption \ref{ass_init_set}:} The spaces $V=\tilde{V}=X_\varphi=H^{1}(\Omega)$, $W = Y = L^2(\Omega)$ are separable and reflexive Banach spaces. Note that neither initial nor boundary conditions are considered in the setup of Proposition \ref{prop:demo} such that no choice of $H, B, \gamma$ and $\mathcal{D}_{BC}$ is necessary. We have $\kappa=0$ (only eventually unknown reaction terms are learned) and choose $V_1=L^2(\Omega)$. The parameter sets $\Theta^m_n$ in \cite[Theorem 1]{belomestny23} are closed and contained in finite-dimensional spaces (note that the components of each parameter are contained in the interval $[-1,1]$). The embeddings in Assumption \ref{ass_init_set}, ii) are either trivial or follow by the compact embedding $H^{1}(\Omega)\hookdoubleheadrightarrow L^{2}(\Omega)$ where we choose $\hat{p}=\hat{q}=2$. The conditions on the extended spaces in Assumption \ref{ass_init_set}, iii) follow for $p=q=r=2$. Boundedness of the linear operators $(K^m)_m$ implies weak-weak continuity as demanded in Assumption \ref{ass_init_set}, v). Finally the regularization functional $\mathcal{R}: X_\varphi^L\times \mathcal{V}^L\times \Theta^m\to [0,\infty]$ with 
\[ X_\varphi^L\times \mathcal{V}^L\times \Theta^m\ni (\varphi, u, \theta)\mapsto \sum_{l=1}^L(\Vert \varphi^l\Vert_{H^1(\Omega)}^2+\Vert u^l\Vert_{\mathcal{V}}^2)+\Vert f_\theta\Vert_{L^2(U)}^2+\Vert\nabla f_\theta\Vert_{L^\infty(U)}+\nu^m\Vert\theta\Vert,
\]
for a sufficiently large interval $U$ and regularization parameters $\nu^m>0$, is coercive and weakly lower semicontinuous (which follows by Proposition \ref{prop:ass5_nn}). Note that $N=1$. Furthermore, the choice of the image space as $W=L^2(\Omega)$ is justified since $\partial_t u-\varphi\cdot\nabla u\in L^2(\Omega)$ for $\varphi\in H^1(\Omega)\hookrightarrow L^\infty(\Omega)$ and $u\in \mathcal{V} = W^{1,2,2}(0,T;V)$ as $\partial_t u\in L^2(\Omega)$ and $\nabla u\in L^2(\Omega)$ by  $u\in \mathcal{V} = W^{1,2,2}(0,T;V)$. The inclusion for parameterized nonlinearities $f_\theta(u)\in L^2(\Omega)$ is fulfilled by the following considerations.\\

\noindent\textbf{Ad Assumption \ref{ass_param_app_class}:} The extendability to a well-defined Nemytskii operator and continuity property of the parameterized nonlinearities follow as mentioned in the paragraph right before Proposition \ref{prop:ass3_nn} by \cite[Lemma 4, Lemma 5]{AHN23} under the regularity condition in Assumption \ref{ass_uniqueness}, i) that is addressed below.\\

\noindent\textbf{Ad Assumption \ref{ass_phys_term}:} The extendability to a well-defined Nemytskii operator and continuity property of the physical term follow by Proposition \ref{prop:ass4_linear}. The physical term is of the form in \eqref{phys_term_linear} with $\omega=1$, $\Psi=0$, $\Phi_\beta=0$ for $\beta=0$ and $\Phi_\beta(t,\varphi)=\varphi$ for $\beta=1$, $t\in(0,T)$ and $\varphi\in X_\varphi$. The embedding $\tilde{V}\hookrightarrow W^{1,2}(\Omega)$ follows by the choices made above. Furthermore, since $H^1(\Omega)\hookdoubleheadrightarrow \mathcal{C}(\overline{\Omega})$ we can set $s_1=\infty$ and recover also the additional condition in Proposition \ref{prop:ass4_linear}.\\

\noindent\textbf{Ad Assumption \ref{ass_uniqueness}:} Condition i) follows by Remark \ref{rem:regul_cond} with $\tilde{\kappa}=d=1$ and $\eta=2$. The result in Proposition \ref{prop:ass5_nn} implies condition ii). Condition iv) is a consequence of Proposition \ref{prop:belo}.
Condition v) follows from the assumptions in Proposition \ref{prop:demo}: injectivity of $K^\dagger$ is assumed there, and since the $(K^m)_m$ converge to $K^\dagger$ in the operator norm, Remark \ref{rem:conv_meas} yields the convergence in \eqref{operator_conv_weak}. Explicitly, for a weakly convergent $(u^m)_m\subset\mathcal{V}^N$ the norms $\Vert u^m\Vert_{\mathcal{V}^N}$ are bounded by \cite[Proposition 3.5 (iii)]{Brezis2010}, so that $\Vert K^mu^m-K^\dagger u^m\Vert_{\mathcal{Y}^N}\leq \Vert K^m-K^\dagger\Vert_{\mathcal{L}(\mathcal{V},\mathcal{Y})}\Vert u^m\Vert_{\mathcal{V}^N}\to 0$ as $m\to \infty$. Weak-strong continuity of $K^\dagger$, finally, is assumed in Proposition \ref{prop:demo} as well. Condition vii) follows essentially for sufficiently large $U$. Affine linearity in condition viii) is trivial whereas the stated continuity property follows by the considerations on Assumption \ref{ass_phys_term} above.

\section{Example on choice of regularization}
\label{app:analytic_example} In this section, based on a simple example, we show that choosing a $W^{1,\infty}$-type norm in accordance with \eqref{intro_regularizations}, as opposed to an $L^p$-type norm, can indeed be necessary in general for the identification of a hidden physics component. For that, we consider a one-dimensional time-independent equation on the unit interval with known physical term $F$ that depends on the first spatial derivative of the state. Furthermore, we suppose that the state $u$ is given via the full measurement operator (e.g. equal to the identity with noiseless measurement data) and is approximated by a known sequence of states $(u_m)_m$. Note that, although this setup is simpler than the general one considered in this work, the example shows that already in this simplified situation appropriate regularization, as discussed, matters. Considering classes of continuous functions parameterized by sets $\Theta^m$ that can approximate more and more complicated functions better for increasing $m\in\mathbb{N}$ (according to Assumption \ref{ass_uniqueness}, ii) and iv)), we provide an example where solutions $(f_{\theta^m})_m$ to
\begin{align}
	\label{pml2_ex} \tag{$\mathcal{P}^m_{L^2}$} \min_{\theta\in \Theta^m}\Vert f_\theta\Vert_{L^2(0,1)}^2 +\lambda^m\Vert f_\theta(u_m)+F(u_m')\Vert^2_{L^2(0,1)}+\mu^m\Vert \theta\Vert
\end{align}
with $\theta^m\in \Theta^m$, do not converge to the unique solution $f^\dagger\in L^2(0,1)$ of
\begin{align}
	\label{pdl2_ex} \tag{$\mathcal{P}^\dagger_{L^2}$} \min_{f\in L^2(0,1)}\Vert f\Vert_{L^2(0,1)}^2 \quad \text{s.t.} ~ f(u)=g.
\end{align}
It is important to note that the constraint in \eqref{pdl2_ex} is not well-defined for general $f,u,g\in L^2(0,1)$ but for the setup discussed below it is. In view of well-definedness of the minimization problem \eqref{pml2_ex} we refer to the considerations below. Here $(\lambda^m)_m,(\mu^m)_m\subset \mathbb{R}_+$ are suitable regularization parameters, the former monotonically increasing and divergent, and the latter a zero sequence. Finally, in contrast to above situation we show that penalizing the gradient term similarly as in \eqref{intro_regularizations} allows one to identify the unique solution $f^\dagger$ of
\begin{align}
	\label{pdw1_ex} \tag{$\mathcal{P}^\dagger_{W^{1,\infty}}$} \min_{f\in W^{1,\infty}(0,1)}\Vert f\Vert_{L^2(0,1)}^2 +\Vert \nabla f\Vert_{L^\infty(0,1)}\quad \text{s.t.} ~ f(u)=g
\end{align}
as limit of solutions $(f_{\theta^m})_m$ for $\theta^m\in \Theta^m$ to
\begin{align}
	\label{pmw1_ex} \tag{$\mathcal{P}^m_{W^{1,\infty}}$} \min_{\theta\in \Theta^m}\Vert f_\theta\Vert_{L^2(0,1)}^2 +\Vert \nabla f_\theta\Vert_{L^\infty(0,1)}+\lambda^m\Vert f_\theta(u_m)+F(u_m')\Vert^2_{L^2(0,1)}+\mu^m\Vert\theta\Vert.
\end{align}
Concretely, we choose the known physical term $F$ by
\[ \mathbb{R}\ni v\mapsto F(v)=2(v-1).
\] Furthermore, we suppose that the state is given by the identity map $u(x)=x$ for $x\in[0,1]$ and is approximated by the sequence of states $(u_m)_m$ given for $m\in \mathbb{N}$ by
\[ u_m(x) = x +\frac{1}{4\pi m}\sin(2\pi m x)\quad \text{for} ~ x\in [0,1].
\] The sequence $(u_m)_m$ converges to $u$ in $L^\infty(0,1)$ since $\Vert u_m-u\Vert_{L^\infty(0,1)} =(4\pi m)^{-1}$ for $m\in\mathbb{N}$ (in fact, even weakly in $H^1(0,1)$). Note that $u_m:[0,1]\to [0,1]$ is bijective and even diffeomorphic since $u_m'(x)=1+\cos(2\pi m x)/2\in [1/2, 3/2]$ for all $x\in [0,1]$ and $m\in \mathbb{N}$. The sequence $(g_m)_m$ with $g_m:=-F(u_m')$, i.e., for $m\in \mathbb{N}$
\begin{align}
	\label{g_m_def} g_m(x) = -\cos(2\pi m x)\quad \text{for} ~  x\in [0,1],
\end{align}
converges to $g=0$ weakly in $L^2(0,1)$. Note that the limit is consistent with the limit equation, in the sense that $-F(u')=-F(1)=0=g$. The regularization parameters are chosen by $\lambda^m =\lambda_0 m^{1/2}$ for $m\in\mathbb{N}$ and fixed $\lambda_0>0$ whereas $(\mu^m)_m$ is a zero sequence satisfying the following conditions. We assume that the zero function can be parameterized for any $m\in \mathbb{N}$ with suitable parameters $\hat{\theta}^m\in \Theta^m$. Furthermore, suppose that the periodic functions $h_m:=g_m\circ u_m^{-1}$ (which have period $1/m$) can be parameterized with parameters $\tilde{\theta}^m\in \Theta^m$ for $m\in\mathbb{N}$. Then we choose $(\mu^m)_m$ such that both $(\mu^m\Vert\hat{\theta}^m\Vert)_m$ and $(\mu^m\Vert\tilde{\theta}^m\Vert)_m$ are zero sequences. Note that the PDE data term $\Vert f(u_m)-g_m\Vert_{L^2(0,1)}^2$ is well-defined for any $f\in L^2(0,1)$. Indeed, since $u_m$ is a diffeomorphism of the interval $[0,1]$ onto itself, a change of variables yields
\begin{align}
	\label{integral_subs} \Vert f(u_m)\Vert_{L^2(0,1)}^2 = \int_0^1\vert f(x)\vert^2(u_m'(u_m^{-1}(x)))^{-1}\dx x.
\end{align}
Using that $1/2\leq u_m'(x)\leq 3/2$ for $x\in [0,1]$ and $m\in\mathbb{N}$ implies
\begin{align}
	\label{two_side} \frac{2}{3}\Vert f\Vert_{L^2(0,1)}^2\leq\Vert f(u_m)\Vert_{L^2(0,1)}^2 \leq 2\Vert f\Vert_{L^2(0,1)}^2.
\end{align}
Note that by similar arguments we derive for $h_m = g_m\circ u_m^{-1}$ that
\begin{align}
	\label{two_side_composition} \Vert h_m\Vert_{L^2(0,1)}^2\leq\frac{3}{2}\Vert g_m\Vert_{L^2(0,1)}^2 =3/4
\end{align}
using that $\Vert g_m\Vert_{L^2(0,1)}^2=1/2$ for $m\in \mathbb{N}$. For $f\in W^{1,\infty}(0,1)$ well-definedness follows from the embedding $W^{1,\infty}(0,1)\hookrightarrow \mathcal{C}(0,1)$. Furthermore, note that well-definedness of \eqref{pmw1_ex} follows by the considerations in this work, whereas well-definedness of \eqref{pml2_ex} is a consequence of the direct method together with \eqref{integral_subs} applied to parameterizations instead of $f$ and continuity with respect to the parameterization similar as in Proposition \ref{prop:ass3_nn}.

We now start by considering the $L^2$-regularized problem. For $m\in \mathbb{N}$ the function $h_m= g_m\circ u_m^{-1}$ is representable by $f_{\tilde{\theta}^m}$ with $\tilde{\theta}^m\in \Theta^m$ such that the objective functional of \eqref{pml2_ex} in an optimum can be estimated by
\begin{multline*}
	\Vert f_{\theta^m}\Vert_{L^2(0,1)}^2+\lambda_0m^{1/2}\Vert f_{\theta^m}(u_m)-g_m\Vert_{L^2(0,1)}^2+\mu^m\Vert\theta^m\Vert\\
	\leq \Vert f_{\tilde{\theta}^m}\Vert_{L^2(0,1)}^2+\lambda_0m^{1/2}\Vert \underbrace{f_{\tilde{\theta}^m}(u_m)-g_m}_{=0}\Vert_{L^2(0,1)}^2+\mu^m\Vert\tilde{\theta}^m\Vert.
\end{multline*}
Due to the choice of $(\mu^m)_m$ and \eqref{two_side_composition} the right-hand side is uniformly bounded for sufficiently large $m\in\mathbb{N}$. Thus, it follows that $\Vert f_{\theta^m}(u_m)-g_m\Vert_{L^2(0,1)}^2$ converges to zero as $m\to \infty$. Hence, there exists a constant $c>0$ such that by the reverse triangle inequality for sufficiently large $m\in\mathbb{N}$
\begin{align}
	\label{lower_bound_composition} \Vert f_{\theta^m}(u_m)\Vert_{L^2(0,1)}\geq \underbrace{\Vert g_m\Vert_{L^2(0,1)}}_{=2^{-1/2}}-\Vert f_{\theta^m}(u_m)-g_m\Vert_{L^2(0,1)}\geq c >0.
\end{align}
Thus, the sequence $(f_{\theta^m})_m$ cannot converge to zero in $L^2(0,1)$ as \eqref{two_side} would immediately lead to a contradiction to \eqref{lower_bound_composition}. However, the constant zero function is the unique solution to \eqref{pdl2_ex} (of course up to representatives in the Lebesgue sense). As a consequence, the reconstruction of the hidden physical term fails. In fact, one can show for the minimization problem similar to \eqref{pml2_ex} considered over general $L^2(0,1)$-functions (and without parameter regularization) that by analyzing the first variation of the resulting strictly convex objective functional under \eqref{integral_subs}, the corresponding minimizer $f_m$ is given by
\[ f_m(x) =\lambda_0 \left(\lambda_0+m^{-1/2} u_m'(u_m^{-1}(x))\right)^{-1}g_m(u_m^{-1}(x)) \quad \text{for} ~ x\in [0,1],
\] whose $L^2(0,1)$-norm can be shown to be uniformly bounded from below. Another interesting point is that the $(f_m)_m$ are continuously differentiable in the open unit interval and $\lim_{m\to \infty}\Vert \nabla f_m\Vert_{L^\infty(0,1)}=\infty$. Thus, also in a purely analytic setup, the reconstruction of the hidden physics fails.

Let us consider next the $W^{1,\infty}$-type regularized problems. The essential point is that, in contrast to \eqref{pml2_ex}, a constraint on the gradient prevents the parameterizations from following the oscillations of $g_m$, and that the resulting deficit in the data term can be quantified. More concretely, we argue that for $G>0$ and any $f\in W^{1,\infty}(0,1)$ with $\Vert \nabla f\Vert_{L^\infty(0,1)}\leq G$ it holds true that
	\begin{align}
		\label{osc_lower_bound} \Vert f-g_m\Vert_{L^2(0,1)}^2\geq \frac{1}{2}-\frac{2G}{\pi^2 m}
	\end{align}
for $m\in\mathbb{N}$. For that, define $\Phi(x):=-(2\pi m)^{-1}\sin(2\pi m x)$ for $x\in [0,1]$ and note that $\nabla \Phi = g_m$ as well as $\Phi(0)=\Phi(1)=0$. Thus, integration by parts yields for any $f\in W^{1,\infty}(0,1)$ that
	\[ \int_0^1 f(x)g_m(x)\dx x = -\int_0^1 \nabla f(x)\Phi(x)\dx x\leq \Vert \nabla f\Vert_{L^\infty(0,1)}\Vert \Phi\Vert_{L^1(0,1)}.
	\] Using that $\int_0^1\vert \sin(2\pi m x)\vert\dx x = 2/\pi$ for $m\in\mathbb{N}$ we obtain $\Vert \Phi\Vert_{L^1(0,1)}=(\pi^2m)^{-1}$ and hence
	\[ \Vert f-g_m\Vert_{L^2(0,1)}^2 = \Vert f\Vert_{L^2(0,1)}^2-2\int_0^1 f(x)g_m(x)\dx x+\underbrace{\Vert g_m\Vert_{L^2(0,1)}^2}_{=1/2}\geq \frac{1}{2}-\frac{2G}{\pi^2m},
	\] where we dropped the nonnegative term $\Vert f\Vert_{L^2(0,1)}^2$ and used the gradient bound on $f$. Now problem \eqref{pdw1_ex} attains the unique solution $f^\dagger=0$. We argue that $f^\dagger$ is recovered by $(f_{\theta^m})_m$ solving \eqref{pmw1_ex}. For that, we first estimate the objective function of \eqref{pmw1_ex} in the optimum by its value at $f_{\hat{\theta}^m}$ representing the constant zero function, yielding
\begin{multline}
	\label{objective_estimation}
	\Vert f_{\theta^m}\Vert_{L^2(0,1)}^2+\Vert\nabla f_{\theta^m}\Vert_{L^\infty(0,1)}+\lambda_0m^{1/2}\Vert f_{\theta^m}(u_m)-g_m\Vert_{L^2(0,1)}^2+\mu^m\Vert \theta^m\Vert\\
	\leq  \lambda_0m^{1/2}\Vert g_m\Vert_{L^2(0,1)}^2+\mu^m\Vert\hat{\theta}^m\Vert=\lambda_0m^{1/2}/2+\mu^m\Vert\hat{\theta}^m\Vert.
\end{multline}
We derive by \eqref{objective_estimation} that
\[ \Vert\nabla f_{\theta^m}\Vert_{L^\infty(0,1)}\leq \lambda_0m^{1/2}/2+\mu^m\Vert\hat{\theta}^m\Vert,
\] which implies the existence of a constant $c>0$ such that $\Vert \nabla f_{\theta^m}\Vert_{L^\infty(0,1)} \leq c m^{1/2}$ for $m\in\mathbb{N}$ since $(\mu^m\Vert \hat{\theta}^m\Vert)_m$ is a zero sequence. Under this constraint on the gradient, the data term is bounded from below as follows. By the mean value theorem it holds true for $f\in W^{1,\infty}(0,1)$ that
\[ \Vert f(u)-f(u_m)\Vert_{L^2(0,1)}\leq \Vert \nabla f\Vert_{L^\infty(0,1)}\Vert u-u_m\Vert_{L^2(0,1)}=\tilde c\Vert \nabla f\Vert_{L^\infty(0,1)}m^{-1}
\] with $\tilde c=32^{-1/2}\pi^{-1}$, such that, using $u=\id$ and the reverse triangle inequality,
\begin{align}
	\label{inf_estim_data} \Vert f_{\theta^m}(u_m)-g_m\Vert_{L^2(0,1)}\geq \Vert f_{\theta^m}-g_m\Vert_{L^2(0,1)}-\tilde c\Vert \nabla f_{\theta^m}\Vert_{L^\infty(0,1)}m^{-1}.
\end{align}
Applying estimate \eqref{osc_lower_bound} with $G=cm^{1/2}$ to the first term on the right-hand side of \eqref{inf_estim_data}, we obtain for sufficiently large $m\in \mathbb{N}$ that
\begin{align}
	\label{data_lower_bound} \Vert f_{\theta^m}(u_m)-g_m\Vert_{L^2(0,1)}\geq \left(\frac{1}{2}-\frac{2c}{\pi^2m^{1/2}}\right)^{1/2}-\frac{c}{\sqrt{32}\pi m^{1/2}}.
\end{align}
Combining these arguments, finally, yields by \eqref{objective_estimation} that
\begin{align*}
	\Vert f_{\theta^m}\Vert_{L^2(0,1)}^2+\Vert\nabla f_{\theta^m}\Vert_{L^\infty(0,1)} \leq  \lambda_0m^{1/2}\left[\frac{1}{2}-\left(\left(\frac{1}{2}-\frac{2c}{\pi^2m^{1/2}}\right)^{1/2}-\frac{c}{\sqrt{32}\pi m^{1/2}}\right)^2\right]+\mu^m\Vert\hat{\theta}^m\Vert.
\end{align*}
The right-hand side converges to the constant $\frac{c\lambda_0}{4\pi}\left(\frac{8}{\pi}+1\right)$ as $m\to\infty$. Thus, there exist $M\in\mathbb{N}$ and a constant $\alpha>0$ such that for all $m\geq M$ it holds true that
\[ \Vert \nabla f_{\theta^m}\Vert_{L^\infty(0,1)}\leq \alpha.
\] Repeating the arguments starting from \eqref{inf_estim_data}, but now with the refined bound $\Vert \nabla f_{\theta}\Vert_{L^\infty(0,1)}\leq \alpha$ on the gradients, that is, applying estimate \eqref{osc_lower_bound} with $G=\alpha$, yields for $m\geq M$ that
\begin{align*}
	\Vert f_{\theta^m}\Vert_{L^2(0,1)}^2+\Vert\nabla f_{\theta^m}\Vert_{L^\infty(0,1)} \leq  \lambda_0m^{1/2}\left[\frac{1}{2}-\left(\left(\frac{1}{2}-\frac{2\alpha}{\pi^2m}\right)^{1/2}-\frac{\alpha}{\sqrt{32}\pi m}\right)^2\right]+\mu^m\Vert\hat{\theta}^m\Vert,
\end{align*}
which converges to zero as $m\to \infty$. As a consequence, it holds true that
\[ \lim_{m\to \infty}\Vert \nabla f_{\theta^m}\Vert_{L^\infty(0,1)}=0\quad \text{and} \quad \lim_{m\to \infty}\Vert  f_{\theta^m}\Vert_{L^2(0,1)}=0,
\] implying that $(f_{\theta^m})_m$ converges uniformly to zero, the unique solution to \eqref{pdw1_ex}.
\paragraph*{Numerical experiments.} To show that the above counterexample is also observable in practice, we implemented it numerically using two-hidden layers neural networks with input- and output dimension one. The first layer consists of $10$ nodes and is sine-activated with frequency $2$ (to enable the representation of high frequency oscillations). The weights are initialized uniformly in $[-\pi m,\pi m]$. The second layer consists of $m\in\mathbb{N}$ nodes and is ReLU-activated. The network training for  \eqref{pml2_ex} and \eqref{pmw1_ex} was performed over $1500$ epochs for $m\in\{10,100,1000\}$ using the Adam optimizer with weight decay $\mu^m=0.1\cdot m^{-1}$ and a learning rate cosine-annealed over the epoch budget from $3\cdot10^{-3}$ to $1.5\cdot10^{-4}$, for $\max\{200,10m\}$ equidistant training points in the unit interval. Furthermore, we chose $\lambda_0=1/2$. The gradient $\nabla f_\theta$ in view of the loss calculation of \eqref{pmw1_ex} is approximated via $\max\{30,2m\}$ uniformly sampled approximations of the gradient using finite differences with step size $\min\{10^{-3},(100m)^{-1}\}$. An illustration of the results is given in Figure \ref{fig:example}.
\begin{figure}[htbp]
	\begin{center}
		\vspace*{-0.6cm}
		\includegraphics[width=\textwidth]{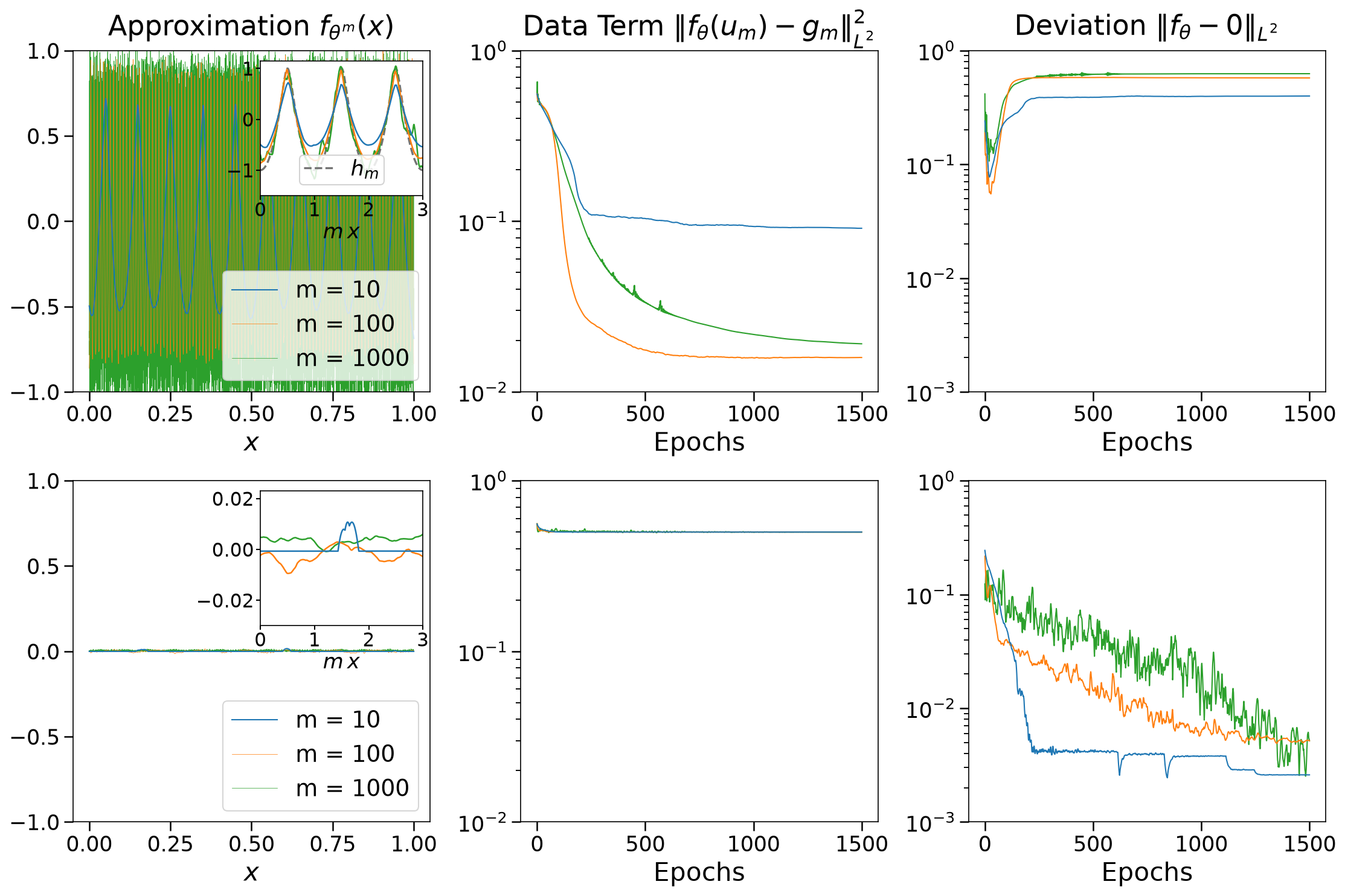}
		\caption{Numerical example with different regularizations. Top: $L^2$ regularization, bottom: $W^{1,\infty}$-type regularization.}
		\alttext{A two-by-three grid of subplots. The top row corresponds to $L^2$ regularization and the bottom row to the $W^{1,\infty}$-type regularization. The three columns show, respectively, the learned hidden-physics approximations after 1500 training epochs for $m\in\{10,100,1000\}$, the data-term loss curves over training, and the $L^2$-deviation of the learned approximations from the zero function (the unique solution of the limit problem). The first column carries an inset in which the approximations are drawn against $mx$ over three periods of $g_m$, so that the three $m$ show the same number of oscillations and can be compared with the target $h_m=g_m\circ u_m^{-1}$ (dashed). Note that the inset of the bottom row resolves a residual two orders of magnitude below the amplitude of the top row. In the bottom row the approximations flatten toward zero and the deviation falls below $10^{-2}$ for every $m$, while in the top row the approximations follow the oscillations of $g_m$ and the deviation does not vanish.}\vspace*{-0.6cm}
		\label{fig:example}
	\end{center}
\end{figure}
In Figure \ref{fig:example} the first and second line of subplots correspond to the considerations on \eqref{pml2_ex} and \eqref{pmw1_ex}, respectively. The first column of subplots depicts approximations of hidden physics after $1500$ epochs for the different $m$ above. In the second column the corresponding data term losses are plotted. Finally, in the third column the $L^2$-deviation of the approximated hidden physics from the unique solutions of the limit problem \eqref{pdl2_ex} and \eqref{pdw1_ex}, the zero function, is depicted, respectively. The main conclusion is that under $L^2$-regularization the physical term reconstruction fails, whereas with the $W^{1,\infty}$-type penalty the hidden physics is identified correctly. For the latter it is further clear that for each of the three $m$ the deviation decreases throughout the training, by roughly two orders of magnitude, and that the parameterized approximations are relatively flat due to the gradient penalty. Note that the data term loss of \eqref{pmw1_ex} stagnates as predicted by our analysis above.

\FloatBarrier
\bibliographystyle{plainnat}
\bibliography{references}

\end{document}